\documentclass[a4paper, 11pt, reqno]{amsart}
\usepackage[margin=1.2in]{geometry}

\usepackage{amssymb}
\usepackage{amsmath} 
\usepackage{mathrsfs}
\usepackage[all]{xy}
\usepackage{mathtools}
\usepackage{color}
\usepackage[dvipsnames]{xcolor}  
\usepackage[hidelinks]{hyperref}
\usepackage{url}
\usepackage[english]{babel}
\usepackage{comment}
\usepackage{cases}
\usepackage{todonotes}
\usepackage{enumitem}
\usepackage{colortbl}
\usepackage{longtable}
\usepackage{colonequals}
\usepackage{tikz}
\usepackage{tikz-cd}
\usepackage{pgf,tikz}
\usetikzlibrary{arrows}
\usepackage{fancyhdr}
\usepackage[style=alphabetic, sorting=nty, maxbibnames=6]{biblatex}
\usepackage{csquotes}              

\usepackage[OT2,T1]{fontenc}
\DeclareSymbolFont{cyrletters}{OT2}{wncyr}{m}{n}

\numberwithin{equation}{section} \numberwithin{figure}{section}

\DeclareSymbolFont{cyrletters}{OT2}{wncyr}{m}{n}
\DeclareMathSymbol{\Sha}{\mathalpha}{cyrletters}{"58}
\DeclareMathSymbol{\Be}{\mathalpha}{cyrletters}{"42}

 \newcommand{\term}[1]{\textbf{#1}}          
 \newcommand{\textred}[1]{\textcolor{red}{#1}}
 
\DeclareMathOperator{\ord}{ord}
\DeclareMathOperator{\Pic}{Pic}
\DeclareMathOperator{\Div}{Div}

\DeclareMathOperator{\Gal}{Gal}

\DeclareMathOperator{\supp}{supp}

\DeclareMathOperator{\Aut}{Aut}

\DeclareMathOperator{\Spec}{Spec}

\DeclareMathOperator{\Frob}{Frob}

\DeclareMathOperator{\inv}{inv}
\DeclareMathOperator{\Br}{Br}

\newcommand\blfootnote[1]{%
  \begingroup
  \renewcommand\thefootnote{}\footnote{#1}%
  \addtocounter{footnote}{-1}%
  \endgroup
}

\def \F {{\mathbb F}}

\def \P {{\mathbb P}}
\def \Q {{\mathbb Q}}

\def \Z {{\mathbb Z}}

\newcommand{\ra}{\rightarrow}

\newtheorem{lemma}{Lemma}

\newtheorem{theorem}[lemma]{Theorem}
\newtheorem{proposition}[lemma]{Proposition}
\newtheorem{corollary}[lemma]{Corollary}

\newtheorem{claim}[lemma]{Claim}

\theoremstyle{definition}
\newtheorem{example}[lemma]{Example}

\newtheorem{definition}[lemma]{Definition}
\newtheorem{remark}[lemma]{Remark}

\newtheorem{notation}[lemma]{Notation}
\newtheorem*{notation*}{Notation}

\numberwithin{lemma}{section}

\begin{document}

\title{Existence of minimal Del Pezzo surfaces of degree 1 with conic bundles over finite fields}

\author{Manoy T. Trip}

\begin{abstract}
    We study minimal del Pezzo surfaces of degree 1 with a conic bundle over a finite field $\mathbb{F}_q$ according to the action of the absolute Galois group on the singular fibers (which is known as their type). We give a lower bound on the size of the field over which they exist, and determine values of $q$ for which certain types cannot exist. In particular, we solve the inverse Galois problem for certain types of minimal del Pezzo surfaces of degree 1 over finite fields with a conic bundle structure. Additionally, we give bounds on the values of $q$ for which del Pezzo surfaces of degree 1 of index 8 exist over $\F_q$.
\end{abstract}

\begin{comment}
        Minimal del Pezzo surfaces of degree 1 with conic bundles over finite fields can be divided into seven types, based on the Galois action on the components of the singular fibers of the conic bundle structure. For each of these types, we give a bound such that if $q$ is larger than this value, there exists a minimal del Pezzo surface of degree 1 with conic bundle of that type. We also show that there are types which cannot exist if $q$ is equal to 2, 3 or 4. Together, for 4 of the 7 types, this gives a complete classification of the existence over $\F_q$ of minimal del Pezzo surfaces of degree 1 with conic bundles. We similarly give bounds on the values of $q$ for which del Pezzo surfaces of degree 1 of index 8, i.e. blow-ups of $\P^2$, exist over $\F_q$.
\end{comment}

\maketitle

\blfootnote{\textit{Date}: \today.}

\thispagestyle{empty}

\section{Introduction} \label{sec:intro}

A del Pezzo surface $X$ over a field $k$ is a smooth, projective, geometrically integral surface with an ample anticanonical divisor $-K_X$. Its degree $d$ is the self-intersection number $K_X^2$, and we have $1 \leq d \leq 9$. All $k$-minimal rational surfaces are either del Pezzo surfaces or standard conic bundles \cite[Theorem 1]{iskovskih}, which motivates the study of these two types of minimal surfaces.

Let $k$ be a perfect field, $\overline{k}$ a fixed algebraic closure of $k$, and let $\overline{X} \colonequals X \times_{\Spec k} \Spec \overline{k}$. Del Pezzo surfaces over $k$ can be classified based on the action of the absolute Galois group of $k$ on $\overline{X}$, which induces automorphisms on $\Pic(\overline{X})$ that preserve $K_{\overline{X}}$ and the intersection pairing. The subgroup of such automorphisms is isomorphic to the Weyl group of a root system $R_{9-d}$, which only depends on the degree $d$ of $X$.

Let us restrict ourselves to the case where $k$ is a finite field. Let $X$ be a del Pezzo surface over $\F_q$, of degree $d \leq 6$. We consider the action of $\Gal(\overline{\F}_q/\F_q)$ on $\Pic(\overline{X})$, which can be identified with a cyclic subgroup in $W(R_{9-d})$, generated by the image of the Frobenius automorphism. Let $C(X)$ denote the conjugacy class of this image of the Frobenius automorphism in $W(R_{9-d})$. We call this class the \term{type} of $X$. We consider the following problem:\\

\textbf{The inverse Galois problem for del Pezzo surfaces over finite fields.} \textit{Let $q$ be a prime power, and $1 \leq d \leq 6$. For each conjugacy class $C$ of $W(R_{9-d})$, does there exist a del Pezzo surface $X/\F_q$ of degree $d$ such that $C(X) = C$?}\\

\subsection{Existing literature}
For any fixed degree, it is shown in \cite[Corollary 1.8]{BFL} that for $q$ large enough, every type can be realized. For each type, the question remains which exact values of $q$ are large enough, and what happens for smaller values of $q$. For del Pezzo surfaces of degree at least two, the problem is solved completely, for all values of $q$. For del Pezzo surfaces of degree 5 and higher, it is explained in \cite[Proposition 3.1]{trepalin} that every type can be realized for every value of $q$. For del Pezzo surfaces of degree 4, there are 18 different types (see \cite[Table 4]{trepalin}). Theorem 1.4 in \cite{trepalin} summarizes the solution to the inverse Galois problem for these surfaces, collecting results from \cite{rybakov, BFL, trepalindeg2, trepalin}. There are 25 types of del Pezzo surfaces of degree 3 \cite[Table 7.1]{BFL}. In \cite[Theorem 1.1]{LoughranTrepalin}, it is shown precisely for which $q$ a del Pezzo surface of degree 3 of each type exists over $\F_q$, based on results in \cite{rybakov, SD, rybakovtrepalin, trepalindeg2, BFL, trepalin, LoughranTrepalin} (see the proof of Theorem 1.3 in \cite{trepalin} for an overview of what is proven where). For del Pezzo surfaces of degree 2 over $\F_q$, there are 60 possible types (\cite[Table 1]{urabe2}, \cite[Table 2]{trepalin}). The solution to the inverse Galois problem for degree 2 is given in \cite[Theorem 1.2]{LoughranTrepalin}, based on results in \cite{trepalindeg2, BFL, trepalin, LoughranTrepalin}. The proof of Theorem 1.2 in \cite{trepalin} contains an overview of what is proven where.

For del Pezzo surfaces of degree 1, the relevant root system $R_{9-d}$ is $E_8$ (see Theorem \ref{thm:manin23.9}). The Weyl group $W(E_8)$ has 112 conjugacy classes, which are tabulated in \cite[Table 2]{urabe2} with some corresponding properties. In \cite{BFL}, Banwait, Fit\'{e} and Loughran treat the inverse Galois problem for del Pezzo surfaces of degree 1 over finite fields for one type of each possible trace value. More precisely, they treat the following types (numbering after \cite{urabe1, urabe2}): 91, 92, 94, 97, 102, 107, 111, 112, and their Bertini twists 8, 44, 15, 1, 19, 62 and 20. 

In \cite{ICP}, another coarser analogue of the inverse Galois problem is treated, named the inverse curve problem for del Pezzo surface. They classify del Pezzo surfaces based on the number of lines and conics on the surface that are defined over the base field $\F_q$, which is a coarser classification than the classification by type. They completely solve the inverse line problem and inverse conic problem for del Pezzo surfaces of any degree over any finite field.

\subsection{Our contribution}
Types 1 to 37 in \cite{urabe2} correspond to minimal del Pezzo surfaces. In this paper, we treat the types corresponding to minimal del Pezzo surfaces of degree 1 which have a conic bundle structure, which are types 1 to 7. We obtain the following main result, solving the inverse Galois problem for these types for all but some small values of $q$. 

\begin{theorem}\label{thm:mainthm}
    \begin{enumerate}
        \item There exists a del Pezzo surface of degree 1 over $\F_q$ of type 1 if and only if $q \geq 5$.
        \item There exists a del Pezzo surface of degree 1 over $\F_q$ of type 2 if and only if $q \geq 3$.
        \item There exists a del Pezzo surface of degree 1 over $\F_q$ of type 3 if $q \geq 4$. It does not exist when $q = 2$.
        \item There exists a del Pezzo surface of degree 1 over $\F_q$ of type 4 if $q \geq 7$.
        \item There exists a del Pezzo surface of degree 1 over $\F_q$ of type 5 if $q \geq 4$.
        \item There exists a del Pezzo surface of degree 1 over $\F_q$ of type 6 for all $q$.
        \item There exists a del Pezzo surface of degree 1 over $\F_q$ of type 7 for all $q$.
    \end{enumerate}
\end{theorem}

We obtain a similar result about weak del Pezzo surfaces.

\begin{theorem}\label{thm:weakdPexistence}
    There exists a weak del Pezzo surface of degree 1 over $\F_q$ which is a minimal standard conic bundle in the following cases:
    \begin{enumerate}[label=(\roman*)]
        \item of type 1 when $q \geq 5$; it does not exist when $q = 2, 3$;
        \item of type 3 if and only if $q \geq 3$;
        \item of type 4 when $q \geq 3$;
        \item of type 2, 5, 6 and 7 for all $q$.
    \end{enumerate}
    with the notion of type as in Definition \ref{def:typeconicbundles}.
\end{theorem}

In the process of proving Theorem \ref{thm:mainthm}, we adapt a strategy laid out in \cite[\S5.2]{BFL} to obtain a solution to the inverse Galois problem for del Pezzo surfaces of degree 1 and index 8, that is, a solution for all values of $q$, for all but one type. Using the fact that the existence of a del Pezzo surface of degree 1 is equivalent to the existence of its Bertini twist, we obtain the following result.

\begin{theorem}\label{thm:index8+twist}
\begin{enumerate}
    \item Del Pezzo surfaces of degree 1 of type 8 and 91 exist over $\F_q$ if and only if $q = 16$ or $q \geq 19$.
    \item Del Pezzo surfaces of degree 1 of type 44 and 92 exist over $\F_q$ if and only if $q \geq 11$.
    \item Del Pezzo surfaces of degree 1 of type 15, 38, 69, 93, 94 and 95 exist over $\F_q$ if and only if $q \geq 7$.
    \item Del Pezzo surfaces of degree 1 of type 1 and 97 exist over $\F_q$ if and only if $q \geq 5$.
    \item Del Pezzo surfaces of degree 1 of type 40, 48, 96, 98 and 99 exist over $\F_q$ if and only if $q \geq 4$.
    \item Del Pezzo surfaces of degree 1 of type 2, 16, 19, 50, 62, 76, 100, 101, 102, 103, 104, 105 and 107 exist over $\F_q$ if and only if $q \geq 3$.
    \item Del Pezzo surfaces of degree 1 of type 20, 22, 49, 106, 109, 110, 111, 112 exist over $\F_q$ for all $q$.
\end{enumerate}
\end{theorem}

\subsection{Methods in relation to the literature}

To prove Theorem \ref{thm:mainthm}, we use two distinct strategies. For the first strategy, which is treated in Section \ref{sec:method1}, we use the existence of conic bundles of a certain type to construct corresponding del Pezzo surfaces with conic bundle structure. This method is modeled after \cite{trepalindeg2}, which deals with the analogous problem for del Pezzo surfaces of degree 2. We prove Proposition \ref{prop:classificationCB}, which sorts minimal standard conic bundles of degree 1 into five possible cases. This is analogous to \cite[Proposition 2.3]{trepalindeg2} for degree 2. Each of the seven types of minimal del Pezzo surfaces of degree 1 with conic bundle structure is uniquely determined by the Galois action on the singular fibers of the conic bundle. We use \cite[Theorem 2.11]{rybakov}, which guarantees the existence of a minimal conic bundle with the corresponding Galois action for each type (except for some very small values of $q$, see Lemma \ref{lem:smallfields}). We then construct birational maps to transform it into a del Pezzo surface of the corresponding type.

The second strategy (Section \ref{sec:method2}) uses that for any del Pezzo surface of degree 1, there exists a quadratic twist, called the Bertini twist, which is again a del Pezzo surface of degree 1. For surfaces over $\F_q$ of types 1, 2 and 4, the corresponding twist is no longer a minimal surface, and in fact it contains an exceptional curve defined over $\F_q$, which can be contracted to obtain a del Pezzo surface of degree 2 of a specific type. In \cite[Theorem 1.2]{LoughranTrepalin}, we find for which $q$ such a del Pezzo surface of degree 2 exists. If such a surface $Y$ exists, we show for large enough $q$ that we can find an $\F_q$-point $P$ on $Y$ such that the blow-up of $Y$ in $P$ results in a del Pezzo surface of degree 1 of the required type. For the twist of type 1, this is already done in \cite{BFL}, and we adapt their strategy for all other types of del Pezzo surfaces which contain an exceptional curve defined over $\F_q$, which includes in particular the twists of type 2 and 4. For small values of $q$ and index 8, we also adapt Magma code from \cite{BFL} to computationally confirm existence or non-existence, completing the solution in those cases. 

\subsection{Organization of the paper}

We start this paper by giving a background of the relevant theory of del Pezzo surfaces (\S\ref{sec:delPezzosurfaces}) and conic bundles (\S\ref{sec:conicbundles}). In \S\ref{sec:autpicdelPezzo}, we introduce the notion of the type of a del Pezzo surface over $\F_q$. We also define the index, trace and Picard rank of a del Pezzo surface, properties which are uniquely determined by its type. In \S\ref{sec:autpicconicbundle}, we recall an explicit description of the Picard group of a standard conic bundle over an algebraically closed field, and describe the Galois action on this group. We explain a result by Rybakov \cite{rybakov} about the existence of minimal standard conic bundles in \S\ref{sec:rybakov}.

In Section \ref{sec:typesCB}, we use Theorem \ref{thm:rybakov} to obtain an explicit characterization of minimal standard conic bundles over $\F_q$ in the degree 1 case, based on the Galois action on their singular fibers. In \S\ref{sec:Galoisorbits} and \S\ref{sec:Galoisorbitsfibers} we link this characterization to the classification of minimal del Pezzo surfaces of degree 1 with conic bundles, and demonstrate a correspondence between these two classifications. This allows us in \S\ref{sec:conjugacyclassfrob} to give, for a surface $X$ of each type, an explicit element in the conjugacy class of the action of the Frobenius automorphism on $\Pic(\overline{X})$. This gives us a first result about the non-existence of certain types of del Pezzo surfaces of degree 1 (\S\ref{sec:nonexistence}). 

In Section \ref{sec:method1}, we get into the first method used to partially prove Theorem \ref{thm:mainthm}. To construct del Pezzo surfaces of degree 1, we use elementary transformations of conic bundles, which are introduced in \S\ref{sec:elementarytransformations}. We divide minimal standard conic bundles of degree 1 in 5 different cases, numbered (1), (2a), (2b), (2c) and (3), depending on which curves with negative self-intersection they contain (\S\ref{sec:classificationCB}, Proposition \ref{prop:classificationCB}). Case (3) is the case where the surface is a del Pezzo surface. We treat the other cases one by one in \S\ref{sec:case(1)} to \S\ref{sec:case(2c)}. For conic bundle surfaces $X$ satisfying each of the cases, we find an explicit description of the classes in $\Pic(\overline{X})$ of curves with negative self-intersection. We determine lower bounds for $q$ above which there is a point on a smooth fiber over $\F_q$ that allows an elementary transformation resulting in a conic bundle satisfying a different case in Proposition \ref{prop:classificationCB}, eventually transforming the surface into a del Pezzo surface. This results in Theorem \ref{thm:mainthmpart1}, which is a partial result to Theorem \ref{thm:mainthm}, and also allows us to prove Theorem \ref{thm:weakdPexistence}.

In Section \ref{sec:method2}, we lay out a different proof strategy, where we show the existence of del Pezzo surfaces of degree 1 over $\F_q$ of types that contain a $(-1)$-curve defined over $\F_q$. The types 1, 2 and 4 have quadratic twists of this form, which allows us to improve the result of Theorem \ref{thm:mainthmpart1} and conclude our main result in Theorem \ref{thm:mainthm}. Using this method, we also get a complete solution for the inverse Galois problem for all but one of the types of index 8. This results in Theorem \ref{thm:index8+twist}.

\subsection*{Acknowledgements}

I want to thank Cec\'{i}lia Salgado for many useful discussions, and in particular for introducing me to the topic of del Pezzo surfaces of degree 1 over finite fields. I also thank Ander Arriola Corpion for a 
discussion on the details of Lemma \ref{lem:(2)singpointsanticanonical}.

\section{Background} \label{sec:background}

We will call a smooth, projective, geometrically integral surface a \term{nice} surface. A nice surface $X$ defined over a field $k$ is called \term{$k$-minimal} (or just \term{minimal} if the field is clear from the context) if any birational $k$-morphism from $X$ to a smooth surface $X'$ is an isomorphism. Let $\overline{k}$ be an algebraic closure of $k$. We say a surface $X/k$ is \term{(geometrically) rational} if it is birational to $\P^2$ over $\overline{k}$. 
Minimal rational surfaces can be classified as follows.

\begin{theorem}[{\cite[Theorem 1]{iskovskih}}]\label{thm:iskovskih1}
    Let $X$ be a minimal rational surface over a field $k$. Then $X$ satisfies one of the following:
    \begin{enumerate}
    \item $\Pic(X) \cong \Z$ and $X$ is a del Pezzo surface (see \S\ref{sec:delPezzosurfaces});
    \item $\Pic(X) \cong \Z^2$ and it admits the structure of a standard conic bundle (see \S\ref{sec:conicbundles}).
\end{enumerate} 
\end{theorem}

Any nice surface defined over a field $k$ is birational over $k$ to a $k$-minimal surface \cite[Proposition II.16]{beauville}. Several arithmetic properties of varieties, such as unirationality and the Hilbert property \cite[Section 3.1]{SerreTopics}, are birational invariants. In the case of rational surfaces, Theorem \ref{thm:iskovskih1} implies that it is enough to consider questions about such properties for minimal del Pezzo surfaces and minimal conic bundles.

Recall that any birational morphism of smooth projective surfaces over $k$ can be decomposed as a finite sequence of blow-ups in closed points \cite[Theorem 21.4]{manin}. Suppose $k$ is algebraically closed, and let $\phi \colon X' \ra X$ be a blow-up of $X$ in a closed point $P$. Then $\phi^{-1}(P)$ is a curve $E$ on $X'$ satisfying $E^2 = -1$ and $E \cong \P^1$.

\begin{definition}
    Let $X$ be a nice surface defined over a field $k$, and let $\overline{X} \colonequals X \times_{\Spec k} \Spec \overline{k}$. A class $D \in \Pic(\overline{X})$ is called an \term{exceptional class} if $D^2 = -1 = D \cdot K_{\overline{X}} = -1$. An irreducible curve $C$ on $\overline{X}$ is called an \term{exceptional curve}, a \term{$(-1)$-curve}, or simply a \term{line}, if its class in $\Pic(\overline{X})$ is an exceptional class.
\end{definition}

By the adjunction formula \cite[V, Exercise 1.3]{hartshorne}, a curve $C$ is a $(-1)$-curve if and only if $C \cong \P^1$ and it satisfies $C^2 = -1$.

\subsection{Del Pezzo surfaces}\label{sec:delPezzosurfaces}
Recall that a del Pezzo surface $X$ of degree $d$ over a field $k$ is a nice surface with an ample anticanonical divisor $-K_X$ such that $K_X^2 = d$. Over an algebraically closed field, del Pezzo surfaces can be described as follows.

\begin{theorem}[{\cite[Theorem 24.3, Theorem 24.4, Remark 26.3]{manin}}]\label{thm:delpezzoclassification}
    Let $X$ be a del Pezzo surface over an algebraically closed field $k$, and let $d$ be its degree. Then $1 \leq d \leq 9$, and we have the following characterization.
    \begin{enumerate}
        \item[(i)] If $d = 9$, then $X \cong \P^2$.
        \item[(ii)] If $d = 8$, then $X \cong \P^1 \times \P^1$, or $X$ is a blow-up of $\P^2$ in a closed point.
        \item[(iii)] If $1 \leq d \leq 7$, then $X$ is a blow-up of $\P^2$ in $9-d$ closed points in \term{general position} (i.e. no three points lie on a line, no six points lie on a conic, and no eight points lie on a singular cubic with one of them being the singularity).
    \end{enumerate}
    Conversely, any surface described in (i), (ii), (iii) for $d \geq 3$ is a del Pezzo surface.
\end{theorem}

\begin{comment}
what if $d = 2, 1$?
\end{comment}
If $X$ is a nice surface over $k$ such that $-K_X$ is big and nef, we say that $X$ is a \term{weak del Pezzo surface} (see for example \cite[Th\'{e}or\`{e}me 1]{demazure} for equivalent definitions). We also have the following geometric characterization.

\begin{theorem}[{\cite[Corollary 8.1.17]{dolgachev}}]\label{thm:weakdelpezzoclassification}
    Let $X$ be a weak del Pezzo surface over an algebraically closed field $k$ of degree $3 \leq d \leq 6$. Then $X$ is a blow-up of $\P^2$ in $9-d$ closed points in \term{almost general position}. That is, $X$ can be obtained by a sequence of blow-ups
    \[X_s \ra \cdots \ra X_1 \ra X_0 \ra \P^2,\]
    for some $s < 9-d$, where
    \begin{itemize}
        \item $X_0 \ra \P^2$ is a blow-up of $\P^2$ in $9-d-s$ closed points such that no four points lie on a line, and no seven points lie on a conic;
        \item $X_i \ra X_{i-1}$ for $i = 0, \ldots, s-1$ is a blow-up of one closed point which lies on a $(-1)$-curve, avoiding all curves of self-intersection $-2$.
    \end{itemize}
\end{theorem}

In particular, from the descriptions in Theorem \ref{thm:delpezzoclassification} and Theorem \ref{thm:weakdelpezzoclassification}, (weak) del Pezzo surfaces are rational surfaces.

\subsection{Conic bundles}\label{sec:conicbundles}
\begin{definition}
    A nice surface $X$ over a field $k$ is called a \term{conic bundle} over $k$ (or we say $X$ has a conic bundle structure over $k$) if there exists a surjective $k$-morphism $f \colon X \ra B$ where $B$ is a smooth curve and the generic fiber of $f$ is a smooth, irreducible curve of genus 0.
    A conic bundle is called \term{standard} if every singular geometric fiber is a pair of $(-1)$-curves intersecting transversally in one point. 
\end{definition}

\begin{comment}
\begin{proposition}[{\cite[Proposition I.8(i)]{beauville}}]
    Let $f \colon X \ra B$ be a conic bundle over a field $k$. Then the fibers of $f$ above rational points on $B$ are linearly equivalent divisors of self-intersection $0$.
\end{proposition}

\begin{proof}
    This is shown in \cite[Proposition I.8(i)]{beauville} for $k$ algebraically closed. The result for arbitrary $k$ follows by considering the base change $\overline{f} \colon \overline{X} \ra \overline{B}$ to $\overline{k}$. 
\end{proof}
\end{comment}

\begin{remark}\label{rem:rationalconicbundle}
     Let $f \colon X \ra B$ be a conic bundle over a field $k$. If $X$ is a geometrically rational surface, then $B$ has genus 0 by L\"{u}roth's Theorem \cite{luroth}. If moreover $X(k) \neq \emptyset$, then $B \cong \P^1_{k}$. Conversely, if $B$ has genus 0, then $X$ is a geometrically rational surface (Noether's Lemma, see \cite[Proposition 1(b)]{iskovskih}). In this case, we say $X$ is a \term{rational conic bundle}.
\end{remark}

\begin{comment}
\begin{proposition}\label{prop:rationalconicbundle}
    Let $f \colon X \ra B$ be a conic bundle over a field $k$. If $X$ is a geometrically rational surface, then $B$ has genus 0. If moreover $X(k) \neq \emptyset$, then $B \cong \P^1_{k}$. Conversely, if $B$ has genus 0, then $X$ is a geometrically rational surface. In this case, we say $X$ is a \term{rational conic bundle}.
\end{proposition}

\begin{proof}
    The first statement follows from the fact that if $B$ is unirational, it is rational by L\"{u}roth's Theorem \cite{luroth}. If $X(k) \neq \emptyset$, then $B(k) \neq \emptyset$. Then $B$ is a curve of genus 0 with a rational point, and thus isomorphic to $\P^1$ over $k$. The converse statement follows from Noether's Lemma (see \cite[Proposition 1(b)]{iskovskih}).
\end{proof}
\end{comment}

\begin{definition}
    Let $X$ be a nice surface defined over a field $k$, and let $f \colon X \ra B$ and $f' \colon X \ra B'$ be two conic bundle structures on $X$. We say that $f$ and $f'$ are \term{distinct} if there is no isomorphism $\phi \colon B \ra B'$ satisfying $\phi \circ f = f'$. If $B \cong B' \cong \P^1$, this is equivalent to saying $f$ and $f'$ are distinct if the fibers of ${f}$ and $f'$ have different classes in $\Pic(X)$ \cite[II, Theorem 7.1, Example 7.1.1]{hartshorne}.
\end{definition}


We recall the following theorem.

\begin{theorem}[{\cite[Theorem 5]{iskovskih}}]\label{thm:twoconicbundles}
    Let $X$ be a minimal del Pezzo surface of degree $d = 1, 2$ or $4$ over a field $k$, such that $\Pic(X) \cong \Z^2$. Then $X$ has precisely two distinct conic bundle structures.
\end{theorem}

\subsection{Automorphisms of the geometric Picard group of a del Pezzo surface}\label{sec:autpicdelPezzo}

From the construction in Theorem \ref{thm:delpezzoclassification}, it can be derived that a del Pezzo surface over an algebraically closed field has a fixed number of $(-1)$-curves, which only depends on the degree of the surface.

\begin{theorem}[{\cite[Theorem 26.2(iii)]{manin}}]
    Let $X$ be a del Pezzo surface of degree $d$ over an algebraically closed field $k$, not isomorphic to $\P^1 \times \P^1$. Then the number of $(-1)$-curves on $X$ is as follows:
    \begin{table}[h]
    \centering
    \begin{tabular}{c|c c c c c c c c c}
        $d$& 1 & 2 & 3 & 4 & 5 & 6 & 7 & 8 & 9\\
        \hline
        $\# (-1)$-curves & 240 & 56 & 27 & 16 & 10 & 6 & 3 & 1 & 0
    \end{tabular}
    \label{tab:numberofexceptionalclasses}
\end{table}
\end{theorem}

Let $X$ be a del Pezzo surface of degree $d \leq 6$ defined over a perfect field $k$. Let us fix an algebraic closure $\overline{k}$ of $k$, and write $\overline{X} \colonequals X \times_{\Spec(k)} \Spec(\overline{k})$. The $(-1)$-curves on $\overline{X}$ are not necessarily defined over $k$. The group of automorphisms of $\Pic(\overline{X})$ which preserve the intersection pairing acts on the set of $(-1)$-curves on $\overline{X}$. 

\begin{theorem}[{\cite[Theorem 23.9]{manin}}] \label{thm:manin23.9}
    Let $X$ be a del Pezzo surface of degree $1 \leq d \leq 6$ defined over $k$, and let $r = 9 - d$. Let $R_r$ be the root system in the following table.
    
    \begin{table}[h]
    \centering
    \begin{tabular}{c|c c c c c c}
        $r$ & 3 & 4 & 5 & 6 & 7 & 8\\
        \hline
        $R_r$ & $A_1 \times A_2$ & $A_4$ & $D_5$ & $E_6$ & $E_7$ & $E_8$.
    \end{tabular}
    \label{tab:rootsystems}
    \end{table}
    
    The following groups are isomorphic:
    \begin{enumerate}
        \item The group of automorphisms of $\Pic(\overline{X})$ preserving $K_{\overline{X}}$ and the intersection pairing;
        \item The group of permutations of the $(-1)$-curves on $\overline{X}$ preserving the intersection pairing;
        \item The Weyl group $W(R_r)$ of the root system $R_r$.
    \end{enumerate}
\end{theorem}

Any field automorphism $\sigma \in \Gal(\overline{k} /k)$ induces an automorphism of $\Pic(\overline{X})$ which preserves the intersection pairing and the canonical divisor class, and thus can be identified with an element of $W(R_{9-d})$. 

Now let $k = \F_q$ be a finite field, and let $X/\F_q$ be a del Pezzo surface of degree $d \leq 6$. By the discussion above, we obtain a representation
\begin{equation}\label{eq:GaloisrepdP}
    \rho_X \colon \Gal(\overline{\F}_q /\F_q) \ra W(R_{9-d}),
\end{equation}
where $W(R_{9-d})$ denotes the Weyl group of $R_{9-d}$, which is a finite group. The image of $\rho_X$ is isomorphic to $\Gal(\F_{q^n}/\F_q)$, where $\F_{q^n}$ is the smallest field such that all $(-1)$-curves on $\overline{X}$ are defined over $\F_{q^n}$. It is cyclic and generated by the image of the Frobenius automorphism $\Frob_q \colon x \mapsto x^q$ on $\overline{\F}_q$.

\begin{definition}
    Let $X$ be a del Pezzo surface of degree $d \leq 6$ over $\F_q$. We define the \term{type} of $X$, denoted by $C(X)$, to be to be the conjugacy class of $\rho_X(\Frob_q)$ in the group $W(R_{9-d})$.
\end{definition}

The type of a del Pezzo surface uniquely determines properties such as the Picard rank, index and trace of the surface, which we introduce in \S\ref{sec:InverseGaloisProblem}. This is why it makes sense to consider the existence of del Pezzo surfaces up to this notion of type.

\subsection{Automorphisms of the geometric Picard group of a standard rational conic bundle}\label{sec:autpicconicbundle}

In this paper, we consider del Pezzo surfaces of degree 1 of types 1 to 7, which are the minimal del Pezzo surfaces with a conic bundle structure. Therefore, we study conic bundles which are simultaneously del Pezzo surfaces, and thus in particular geometrically rational. Furthermore, any conic bundle structure on a del Pezzo surface is standard (this follows from the Nakai--Moishezon Criterion \cite[V, Theorem 1.10]{hartshorne}, see the proof of \cite[Lemma 5.1]{FLS}). We thus focus our attention on standard rational conic bundles. Over an algebraically closed field, there is the following explicit description of the Picard group of a standard rational conic bundle. 

\begin{theorem}[{\cite[Theorem 2.2.1, Theorem 2.2.2]{manintsfasman}, \cite[Proposition 0.3, Proposition 0.4]{KunyavskiiTsfasman}}]\label{thm:MT2.2.2}
    Let $f \colon X \ra \P^1$ be a standard rational conic bundle over an algebraically closed field. We say it has \term{degree} $d \colonequals K_X^2$. Set $r \colonequals 8 - d$. Then
    \begin{enumerate}[label=(\alph*)]
        \item $d \leq 8$.
        
        \item $\Pic(X)$ is torsion-free of rank $\rho(X) = r + 2$.
        \item $f$ has $r$ singular fibers.
        
        \item  $\Pic(X)$ has a free basis $\{C, F, E_1, \ldots, E_r\}$, where $F$ is the class of a fiber of $f$, such that 
    \begin{align*}
        C^2 = F^2 = 0,\ C \cdot F &= 1,&\\
        E_i^2 = -1,\ C \cdot E_i = F \cdot E_i &= 0 \quad \text{for} \ i = 1, \ldots, r\\
        E_i \cdot E_j &= 0 \quad \text{ for }\ i, j \in \{1, \ldots, r\}, i \neq j,
    \end{align*}
    and furthermore $K_X = -2C - 2F + \sum_{n=1}^r E_n$. \label{thm:MT2.2.2(d)}
    
    \item Let $E_i' \colonequals F - E_i$. Then $\{E_1, \ldots, E_r, E_1', \ldots, E_r'\}$ is the set of classes of $(-1)$-curves which are contained in the fibers of $f$.
    
    \item The following three groups are isomorphic:
    \begin{enumerate}[label=(\roman*)]
        \item The group of automorphisms of $\Pic(X)$ preserving $K_X$, $F$ and the intersection pairing;
        \item The group of permutations of the set $\{E_1, \ldots, E_r, E_1', \ldots, E_r'\}$ taking an even number of elements of $\{E_1, \ldots, E_r\}$ to elements of $\{E_1', \ldots, E_r'\}$, and preserving the pairwise intersection pairing.
        \item The Weyl group $W(D_r)$ of the root system $D_r$. 
    \end{enumerate}\label{thm:MT2.2.2(f)}
    \end{enumerate}
\end{theorem}

\begin{proof}
    Parts (a), (b) and (c) are proven in \cite[Theorem 3]{iskovskih}. A variation of part (d) is proven in \cite[Proposition 0.4(a)]{KunyavskiiTsfasman}, with $C^2 = a$ for some $a \in \Z$, and $K_X = -2C + (a-2)F + \sum_{i=1}^r E_i$. We can reduce to the case $C^2 = 0$, using an argument similar to the proof of \cite[Lemma 2.10]{KST}. Suppose $C^2 = a$, and let \[ \begin{cases}
        C' = C - \frac{a}{2}F, \ E_1' = E_1 &\text{if $a$ is even}\\
        C'= C - \frac{a-1}{2}F - E_1, \ E_1' = F - E_1 &\text{if $a$ is odd.}
    \end{cases}\]
    Then replacing $C$ by $C'$, and $E_1$ by $E_1'$, the set $\{C', F, E_1', E_2, \ldots, E_7\}$ is a free basis for $\Pic(X)$ with the desired properties. Part (e) and (f) are also shown in \cite[Proposition 0.4]{KunyavskiiTsfasman}.
\end{proof}

\begin{remark}
    Our proof technique is similar to the argument in \cite[Lemma 2.10]{KST}. However, we noted that their substitution of $s$ by $s'$ does not preserve the intersection number $s \cdot l_1$. Our choice of substitution of $C$ by $C'$ fixes this issue.
\end{remark}

Now let $f \colon X \ra B$ be a standard rational conic bundle over a perfect field $k$. We can consider the base change $\overline{f} \colon \overline{X} \ra \overline {B}$ to a fixed algebraic closure $\overline{k}$, which is again a standard rational conic bundle and thus satisfies the properties in Theorem \ref{thm:MT2.2.2}. There is an induced injective homomorphism $\iota_X^* \colon \Pic(X) \ra \Pic(\overline{X})$ with $\iota_X^*(K_X) = K_{\overline{X}}$, via which we view $\Pic(X)$ as a subgroup of $\Pic(\overline{X})$. If $X(k) \neq \emptyset$, we have $\Pic(X) \cong \Pic(\overline{X})^{\Gal(\overline{k}/k)}$ \cite[Theorem 4.2.5]{BLT}.

\begin{remark}\label{rem:sectionC}
    When $X$ is minimal, by \cite[Proposition 2.6, Corollary 2.7]{trepalindeg2} we can choose the class $C$ in the basis of $\Pic(\overline{X})$ in Theorem \ref{thm:MT2.2.2}\ref{thm:MT2.2.2(d)} in such a way that it is the class of a section of $\overline{f}$, and hence an effective divisor.
\end{remark}

\begin{proposition}\label{prop:picconicbundle}
    Let $f \colon X \ra B$ be a standard rational conic bundle over a perfect field $k$. Then $\rho(X) \geq 2$.
\end{proposition}

\begin{proof}
    Because $B$ is a genus 0 curve, the pullback map $\Pic(B) \ra \Pic(X)$ ensures that $2F$, where $F \in \Pic(\overline{X})$ is the class of a fiber of $\overline{f}$, is an element of $\Pic(X)$. Because $F$ and $K_{\overline{X}}$ are linearly independent (Theorem \ref{thm:MT2.2.2}\ref{thm:MT2.2.2(d)}), this shows that $\rho(X) \geq 2$.
\end{proof}

\begin{comment}
    \begin{proposition}
    Let $f \colon X \ra B$ be a standard rational conic bundle over a perfect field $k$, and let $\overline{k}$ be an algebraic closure of $k$. Let $F \in \Pic(\overline{X})$ be the class of a fiber of $\overline{f}$. Then $K_{\overline{X}}, F \in \Pic(\overline{X})$ are fixed by the action of $\Gal(\overline{k}/k)$. In particular, if $X(k) \neq \emptyset$, we have $\rho(X) \geq 2$.
\end{proposition}

\begin{proof}
    Let $r = 9 - K_X^2$. By Theorem \ref{thm:MT2.2.2}, $\Pic(\overline{X})$ has rank $r + 2$ and a basis $\{C, F, E_1, \ldots, E_r\}$. Let $P$ be a closed point on $B$ of degree $m$. The base change of $P$ to $\overline{X}$ is fixed by the action of $\Gal(\overline{k}/k)$, and so is the divisor $f^{-1}(P)$ on $\overline{X}$, which has class $mF$ in the Picard group. Then also $F$ is fixed by the action of $\Gal(\overline{k}/k)$. If $X(k) \neq \emptyset$ then we have $\Pic(X) \cong \Pic(\overline{X})^{\Gal(\overline{k}/k)}$, so because $K_{\overline{X}}$ and $F$ are linearly independent in $\Pic(\overline{X})$, we conclude that $\rho(X) \geq 2$ in this case.
\end{proof}
\end{comment}

Because $f$ is defined over $k$, the action of $\Gal(\overline{k}/k)$ on $\Pic(\overline{X})$ preserves the class $F$ of a fiber, the class $K_{\overline{X}}$ and the intersection pairing. We saw in Theorem \ref{thm:MT2.2.2}\ref{thm:MT2.2.2(f)} that the group of automorphisms of $\Pic(\overline{X})$ which preserve $K_{\overline{X}}$, $F$, and the intersection pairing can be identified with the group of permutations of the set $\{E_1, \ldots, E_r, E_1', \ldots, E_r'\}$ that preserve the pairwise intersection pairing, and taking an even number of elements of $\{E_1, \ldots, E_r\}$ to elements of $\{E_1', \ldots, E_r'\}$. This group can be identified with $(\Z/2\Z)^{r-1} \rtimes S_r$ \cite[Proof of Proposition 0.4]{KunyavskiiTsfasman}. We thus obtain a representation 
\begin{equation}\label{eq:galoisrep}
    \rho_X^c \colon \Gal(\overline{k} / k) \ra (\Z/2\Z)^{r-1} \rtimes S_r. 
\end{equation}

\begin{notation}\label{not:W(D7)}
    We can write elements of $(\Z/2\Z)^{r-1} \rtimes S_r$ in the form $\iota_{i_1 \cdots i_k} \sigma$ for some even number $k \leq r$, $i_j \in \{1, \ldots, r\}$ distinct and $\sigma \in S_r$.
In the identification with $\Aut(\Pic(\overline{X}))$, using the basis of $\Pic(\overline{X})$ in Theorem \ref{thm:MT2.2.2}\ref{thm:MT2.2.2(d)}, we have that
\[\sigma(C) = C, \quad \sigma(F) = F, \text{ and } \sigma(E_i) = E_{\sigma(i)},\]
that is, $\sigma$ permutes the singular fibers of $X$. We have \
\begin{align*}
    \iota_{i_1 \cdots i_k}(C) &= C + \frac{k}{2} F - E_{i_1} - \cdots - E_{i_k},\\
    \iota_{i_1 \cdots i_k}(F) &= F,\\
    \iota_{i_1 \cdots i_k}(E_i) &=
\begin{cases}
     F - E_i &\text{ if } i \in \{i_1, \ldots, i_k\}\\
     E_i &\text{ if } i \notin \{i_1, \ldots, i_k\}.
\end{cases}
\end{align*}
This corresponds to swapping the two components of the singular fibers containing the lines $E_{i_1}, \ldots, E_{i_k}$.
\end{notation}

\begin{comment}
\begin{proposition}
    Let $f \colon X \ra B$ be a standard rational conic bundle over $k$, with $\Pic(X) = K_X \Z \oplus F \Z$. Then $f$ is relatively minimal.
\end{proposition}

\begin{proof}
\end{proof}

\end{comment}

\subsection{The inverse Galois problem for del Pezzo surfaces of degree 1}\label{sec:InverseGaloisProblem}

In this paper, we consider the inverse Galois problem for del Pezzo surfaces of degree 1 over finite fields, the only degree in which it is still largely unsolved. The 112 conjugacy classes of $W(E_8)$, corresponding to the types of del Pezzo surfaces of degree 1 over finite fields, are tabulated in \cite[Table 2]{urabe2}, together with a number of corresponding properties.

For a smooth projective surface $X/\F_q$, we denote by $a(X)$ the trace of the Frobenius action on $\Pic(\overline{X})$, and call this the \term{trace of $X$}. The trace is uniquely determined by the type of $X$. The number of rational points on $X$ is determined by this trace, given by the following formula. 

\begin{theorem}[{\cite[Theorem 27.1]{manin}}]\label{thm:numberrationalpoints}
    Let $X$ be a smooth projective geometrically rational surface over $\F_q$. Let $a(X)$ be the trace of the Frobenius action on $\Pic(\overline{X})$. Then \[\#X(\F_q) = q^2 + a(X)q + 1.\]
    In particular, $X(\F_q) \neq \emptyset$.
\end{theorem}

In \cite[Theorem 1.5]{BFL}, Banwait, Fit\'{e} and Loughran determine the set of values of $q$ for which a del Pezzo surface of degree 1 of a given trace exists. Because the trace is uniquely determined for a given type, classifying by type gives a finer classification of del Pezzo surfaces of degree 1 than classifying by trace. The proof in \cite{BFL} involves, for each trace value, explicitly choosing a type of surface which has that trace, and demonstrating the existence of a surface of this type. In this paper, we treat a number of types which were not treated in \cite{BFL}.

The index of a surface, defined as follows, is a second invariant of a del Pezzo surface determined by its type.

\begin{definition}
    Let $X$ be a nice surface defined over a perfect field $k$. The \term{index} of $X$ is the size of the largest set of pairwise skew $(-1)$-curves which is stable under the action of the absolute Galois group of $k$. 
\end{definition}

It follows from Castelnuovo's contractibility criterion \cite[Theorem 9.3.3]{poonen} that a surface $X$ is minimal if and only if it has index 0. Table 2 in \cite{urabe2} lists the index for each type of del Pezzo surface of degree 1 over finite fields. Following the numbering in the first column of Urabe's table, we note that the minimal del Pezzo surfaces are the ones of types 1 to 37.

The Picard rank of a del Pezzo surface $X$, denoted by $\rho(X)$, is another property that is determined by its type. Among the minimal surfaces, we find using a Magma \cite{magma} computation (Magma code can be found on GitHub page \cite{github_Existence_dP1}) that surfaces of types 1 to 7 have Picard rank 2, and surfaces of types 8 to 37 have Picard rank 1. Using Theorem \ref{thm:iskovskih1} and Proposition \ref{prop:picconicbundle}, we conclude that surfaces of type 1 to 7 have a conic bundle structure, and surfaces of types 8 to 37 do not have a conic bundle structure.

\section{Classification of minimal standard conic bundles of degree 1}\label{sec:typesCB}

In what follows, we present a result by Rybakov which characterizes the singular fibers of minimal standard rational conic bundles over $\F_q$ (Theorem \ref{thm:rybakov}). This result also places conditions on the existence of standard rational conic bundles which are relatively minimal (Definition \ref{def:relativelyminimal}), which is a priori a weaker condition than minimality. We show that for del Pezzo surfaces of degree 1 with conic bundle structure, relative minimality implies minimality.

\subsection{Existence of minimal rational conic bundles}\label{sec:rybakov}

\begin{definition}\label{def:relativelyminimal}
    Let $f \colon X \ra B$ be a standard rational conic bundle over a perfect field $k$. We say $f$ is \term{relatively minimal} if there is no Galois-stable set of skew $(-1)$-curves contained in the set of singular fibers of $f$. 
\end{definition}

We now restrict to the case where $k = \F_q$ for some prime power $q$. In this case, if $f \colon X \ra B$ is a standard rational conic bundle, then $B \cong \P^1_{\F_q}$ by Remark \ref{rem:rationalconicbundle} and Theorem \ref{thm:numberrationalpoints}.
\begin{comment}Note that any genus 0 curve over $\F_q$ has a rational point and is thus isomorphic to $\P^1$. Therefore, we restrict our attention to conic bundles over $\F_q$ of the form $f \colon X \ra \P^1$.
\end{comment}
We consider the following result about the existence of relatively minimal standard conic bundles by \cite{rybakov}. 

\begin{theorem}[{\cite[Theorem 2.11]{rybakov}}]\label{thm:rybakov}
Let $x_1, \ldots, x_s \in \P^1_{\F_q}$ be a set of closed points. Then there exists a smooth rational surface $X$ and a relatively minimal standard conic bundle structure $f \colon X \ra \P^1$ over $\F_q$, such that the fibers above $x_1, \ldots, x_s$ are singular, and all other fibers are smooth, if and only if $s$ is even.
\end{theorem}

It is clear from the definitions that $\F_q$-minimal standard conic bundles are also relatively minimal. We will show that a theorem by Iskovskih implies that the converse is also true if the surface has degree 1, 2 or 4. Therefore, Theorem \ref{thm:rybakov} gives a characterization of minimal standard rational conic bundles over $\F_q$ in those cases. In order to prove this, we need a description of the action of Galois on those $(-1)$-curves which are components of singular fibers of a relatively minimal standard rational conic bundle over a perfect field $k$, with fixed algebraic closure $\overline{k}$.

Because the action of $\Gal(\overline{k}/k)$ commutes with morphisms defined over $k$, and permutes the $(-1)$-curves on the surface, we conclude that the singular fibers of $f$ lie above a collection of closed points in $\P^1_k$ whose degrees sum to $r$.

\begin{lemma}\label{lem:orbitsinsingularfibers}
    Let $f \colon X \ra B$ be a relatively minimal standard rational conic bundle over $k$. Then the $(-1)$-curves in each singular fiber above a closed point form a single Galois orbit. 
\end{lemma}

\begin{proof}
   Consider a point $P$ in $\P^1_{k}$ of degree $m$, which lies below a singular fiber. The action of $\Gal(\overline{k}/k)$ on the set of $m$ $\overline{k}$-points on $\P^1_{\overline{k}}$ above $P$ is transitive, and the action on the singular fibers above them is compatible. Therefore, the orbit of one of the $(-1)$-curves $L$ contains a $(-1)$-curve above each of the points. The components of the singular fibers therefore either form two orbits, or a single orbit. However, if they from two orbits, each orbit forms a skew Galois stable set of $(-1)$-curves, which implies the surface is not relatively minimal. We conclude that the $(-1)$-curves above $P$ form a single Galois orbit in $\Pic(\overline{X})$. 
\end{proof}

\begin{remark}\label{rem:orbitsinsingularfibersFq}
    If $k = \F_q$, the image of $\rho_X^c$ is cyclic and generated by the action of the Frobenius automorphism $\tau = \rho_X^c(\Frob_q)$. The action on the geometric points above $P$ and on the $(-1)$-curves is then cyclic. Recall from Theorem \ref{thm:MT2.2.2} that we can choose a basis for $\Pic(\overline{X})$ in such a way that the classes of the $(-1)$-curves in the singular fibers are $\{E_1, \ldots, E_r, F-E_1, \ldots, F - E_r\}$. We can freely swap $E_i$ with $F-E_i$, to obtain a basis with the same intersection properties. If there are $s$ closed points below the singular fibers, we can label the $(-1)$-curves in the singular fibers such that $\tau$ is of the form $\iota_{i_1, \ldots, i_s}\sigma$ (in the description of Notation \ref{not:W(D7)}). In particular, $s$ must be even. This gives a different proof of \cite[Corollary 2.10]{rybakov}.
\end{remark}

\begin{lemma}\label{lem:relativelyminalrank2}
     Let $f \colon X \ra B$ be a relatively minimal standard rational conic bundle over $k$. Then $\rho(X) = 2$. If $X(k) \neq \emptyset$, we have $\Pic(X) \cong K_{\overline{X}} \Z \oplus F \Z$, where $F$ is the class of a fiber of $\overline{f}$.
\end{lemma}

\begin{proof}   
    Recall from Theorem \ref{thm:MT2.2.2} that we can choose a basis $\{C, F, E_1, \ldots, E_r\}$ for $\Pic(\overline{X})$ in such a way that the classes of the $(-1)$-curves in the singular fibers are $\{E_1, \ldots, E_r, F-E_1, \ldots, F - E_r\}$. Because the $(-1)$-curves are rigid effective divisors, we abuse notation and also denote the curves themselves by $E_i, F-E_i$ for $i = 1, \ldots, r$. From Lemma \ref{lem:orbitsinsingularfibers}, we deduce that for any $j \in \{1, \ldots, r\}$, we have that $E_j$ and $F - E_j$ are in the same Galois orbit, so there is an element $\tau_j \colonequals \iota_{i_1, \ldots, i_{l_j}}\sigma_j \in \rho_X^c(\Gal(\overline{k}/k))$ such that $\tau_j(E_j) = F - E_j$ (i.e. $\sigma_j$ fixes $E_j$, and $j \in \{i_1, \ldots, i_{l_j}\}$). Recall that $K_{\overline{X}} = -2C -2F + \sum_{n=1}^r E_n$, and any element of $\rho_X^c(\Gal(\overline{k}/k))$ fixes $K_{\overline{X}}$ and $F$. We have
    \[K_{\overline{X}} = \tau_j(K_{\overline{X}}) = -2 \tau(C) -2F + \sum_{n=1}^r E_n + l_j F -2 \sum_{n \in \{i_1, \ldots, i_{l_j}\}} E_n.\] We obtain \[\tau_j(C) = C  + \frac{l_j}{2}F - \sum_{n \in \{i_1, \ldots, i_{l_j}\}} E_n.\]
    Let $D \in \Pic(\overline{X})^{\Gal(\overline{k}/k)}$ be arbitrary, then we can write $D = cC + aF - \sum_{n=1}^7 b_n E_n$ for some $a, c, b_1, \ldots, b_r \in \Z$. From comparing the coefficient of $E_j$ in the expansion of $D$ and $\tau_j(D)$, we obtain that $-b_j = c + b_j$, so $c = -2b_j$. Because this has to hold for every $j = 1, \ldots, r$, we obtain that $c$ is even, and $D$ is of the form $D = b_1 K_{\overline{X}} + (a - 2)F \in \Pic(X)$. Because $K_{\overline{X}}$ and $F$ are linearly independent (Theorem \ref{thm:MT2.2.2}\ref{thm:MT2.2.2(d)}), This shows that $K_{\overline{X}}$ and $F$ form a basis for $\Pic(\overline{X})^{\Gal(\overline{k}/k)}$. We have $\Pic(X) \subseteq \Pic(\overline{X})^{\Gal(\overline{k}/k)}$, and $\rho(X) \geq 2$ by Proposition \ref{prop:picconicbundle}, so $\rho(X) = 2$. If $X(k) \neq \emptyset$, then $\Pic(X) = \Pic(\overline{X})^{\Gal(\overline{k}/k)}$.
\end{proof}

\begin{comment}
  manoy{I cannot find a reference for this :(}  
\end{comment}

\begin{comment}
second attempt
\begin{lemma}
    Let $f \colon X \ra B$ be a relatively minimal standard rational conic bundle over $k$. Then $\Pic$
\end{lemma}

\begin{proof}
    Recall from Theorem \ref{thm:MT2.2.2} that we can choose a basis $\{C, F, E_1, \ldots, E_r\}$ for $\Pic(\overline{X})$ in such a way that the classes of the $(-1)$-curves in the singular fibers are $\{E_1, \ldots, E_r, F-E_1, \ldots, F - E_r\}$. Recall that $K_{\overline{X}} = -2C -2F + \sum_{n=1}^r E_n$, and any element of $\rho_X^c(\Gal(\overline{k}/k))$ fixes $K_{\overline{X}}$ and $F$. Let $D \in \Pic(\overline{X})^{\Gal(\overline{k}/k)}$ be arbitrary, then we can write $D = cC + aF - \sum_{n=1}^7 b_n E_n$ for some $a, c, b_1, \ldots, b_r \in \Z$.
\end{proof}
\end{comment}

\begin{proposition}\label{prop:relativelyminimal-minimal}
    Let $f \colon X \ra \P^1$ be a relatively minimal standard rational conic bundle of degree 1, 2 or 4 over $\F_q$. Then $X$ is $\F_q$-minimal.
\end{proposition}

\begin{proof}
    We have that $\rho(X) = 2$ by Lemma \ref{lem:relativelyminalrank2}. The result follows from \cite[Theorem 4]{iskovskih}.
\end{proof}

\begin{corollary}[Corollary of Theorem \ref{thm:rybakov}]\label{cor:rybakov}
    Let $x_1, \ldots, x_s \in \P^1_{\F_q}$ be a set of closed points of one of the following forms:
    \begin{enumerate}
        \item five $\F_q$-rational points and one point of degree 2;
        \item two $\F_q$-rational points, one point of degree 2 and one point of degree 3;
        \item one $\F_q$-rational point and three points of degree 2;
        \item three $\F_q$-rational points and one point of degree 4;
        \item one $\F_q$-rational point and one point of degree 6;
        \item one point of degree 2 and one point of degree 5;
        \item one point of degree 3 and one point of degree 4.
    \end{enumerate}
    Then there exists a smooth minimal rational surface $X$ and a standard conic bundle structure $f \colon X \ra \P^1$ over $\F_q$, such that the fibers above $x_1, \ldots, x_s$ are singular, and all other fibers are smooth. 
    
    Conversely, let $f \colon X \ra \P^1$ be a minimal standard rational conic bundle of degree 1. Then the singular fibers of $f$ lie above points in $\P^1_{\F_q}$ of one of the seven forms listed above.
\end{corollary}

\begin{proof}
    Consider Theorem \ref{thm:rybakov} in the case of degree 1, that is, when the sum of the degrees of $x_1, \ldots, x_s$ is equal to 7. In this case, by Proposition \ref{prop:relativelyminimal-minimal}, we can replace ``relatively minimal'' by ``minimal'' in the statement of Theorem \ref{thm:rybakov}. There are seven partitions of the integer 7 into an even number of parts, which correspond to the seven listed configurations of an even number of closed points whose degrees sum to seven. 
\end{proof}

Corollary \ref{cor:rybakov} gives us a characterization of possible minimal standard rational conic bundles of degree 1 over $\F_q$ in terms of the Galois configuration of the singular fibers. There are seven possible configurations. Recall that the types 1 to 7 in \cite[Table 2]{urabe2} correspond to minimal del Pezzo surfaces of degree 1 which are standard rational conic bundles. In this section we match the types 1 to 7 with the characterization in \ref{cor:rybakov}.

\subsection{Description of Galois orbits of $(-1)$-curves}\label{sec:Galoisorbits}

Let $X$ be a del Pezzo surface of degree 1 over a finite field $\F_q$. Recall that the type of $X$ is determined by the image of the Frobenius automorphism under the map $\rho_X$ in \eqref{eq:GaloisrepdP}. Because this map is obtained via the identification of $W(E_8)$ with the group of permutations of the $(-1)$-curves preserving the intersection pairing, the type of $X$ determines the size and intersection pattern of the Galois orbits of the $(-1)$-curves. For each type, we can partition the set of 240 $(-1)$-curves into Galois orbits, and determine the intersection pattern of the lines within each orbit. 

\begin{example}
    Consider a del Pezzo surface of degree 1 over $\F_q$ of type 1. We can find in Urabe's table \cite[Table 2]{urabe2} that there are 30 orbits of size 2, which separate into orbits with two different intersection patterns. The other 180 lines are divided into 45 orbits of size four of various intersection types.
\end{example}

Note that the Galois action on the set of lines is cyclic, generated by the image of the Frobenius automorphism $\tau = \rho_X(\Frob_q)$. Hence, it is enough to specify for one element in the orbit how it intersects the others. 

\begin{notation}\label{not:orbits}
Consider Galois orbit of $(-1)$-curves of size $m$, containing a line $L_1$. Let us write $a_i \colonequals L_1 \cdot \tau^i(L_1)$ for $i = 1, \ldots, m-1$. We write $[m \mid a_1, \ldots, a_{m-1}]$ to specify the intersection pattern of this orbit. 
\end{notation}

We use Magma \cite{magma} to compute the intersection numbers between lines within each orbit (see the ancillary file, or \cite{github_Existence_dP1}). Because the action of $\tau$ preserves the intersection pairing, the numbers $a_1, \ldots, a_{m-1}$ completely determine the intersection numbers of all pairs of lines in the orbit.

\begin{remark}
    We note two mistakes in Table 2 in \cite{urabe2}: In line 22, the orbit $3^1$ in the last column should be replaced by $2^1$. Furthermore, in line 87, in the last column, there is a collection of orbits of the form $3^6$ missing.
\end{remark}

\subsection{Galois orbits of $(-1)$-curves in singular fibers of conic bundle}\label{sec:Galoisorbitsfibers}
	
Now suppose $X$ is moreover a minimal standard conic bundle over $\F_q$. Note that $X$ is a geometrically rational surface, and $X(\F_q) \neq \emptyset$ (Theorem \ref{thm:numberrationalpoints}). By Remark \ref{rem:rationalconicbundle} and Theorem \ref{thm:MT2.2.2}, there is an $\F_q$-morphism $f \colon X \ra \P^1$, with 7 singular fibers over $\overline{\F}_q$, which consist of pairs of $(-1)$-curves intersecting transversally. From Corollary \ref{cor:rybakov}, we have seven possibilities for the degrees of points in $\P^1$ above which the singular fibers lie. In Remark \ref{rem:orbitsinsingularfibersFq}, we explicitly described the Galois action on the $(-1)$-curves in the singular fibers above a point of degree $m$. The lines form a Galois orbit of size $2m$, with intersection pattern $[ 2m \mid 0, \ldots, 0, 1, 0 \ldots, 0]$ with the 1 in the $m$-th (middle) position (Notation \ref{not:orbits}). 
This translates to the following orbits of $(-1)$-curves for each possibility in Corollary \ref{cor:rybakov}:
\begin{enumerate}
    \item Five orbits of type $[2 \mid 1]$ and one orbit of type $[4 \mid 0, 1, 0]$;
    \item  Two orbits of type $[2 \mid 1]$, one orbit of type $[4 \mid 0, 1, 0]$ and one orbit of type $[6 \mid 0, 0, 1, 0, 0]$;
    \item One orbit of type $[2 \mid 1]$ and three orbits of type $[4 \mid 0, 1, 0]$;
    \item Three orbits of type $[2 \mid 1]$ and one orbit of type $[8 \mid 0, 0, 0, 1, 0, 0, 0]$;
    \item One orbit of type $[2 \mid 1]$ and one orbit of type $[12 \mid 0, 0, 0, 0, 0, 1, 0, 0, 0, 0, 0]$;
    \item  One orbit of type $[4 \mid 0, 1, 0]$ and one orbit of type $[10 \mid 0, 0, 0, 0, 1, 0, 0, 0, 0]$;
    \item  One orbit of type $[6 \mid 0, 0, 1, 0, 0]$ and one orbit of type $[8 \mid 0, 0, 0, 1, 0, 0, 0]$.
\end{enumerate}

Recall that the types 1 to 7 in \cite[Table 2]{urabe2} correspond to minimal del Pezzo surfaces of degree 1 with a standard conic bundle structure. The Galois action on the $(-1)$-curves, which is uniquely determined by the type, must therefore admit orbits of one of the seven listed possibilities above. For each of the types 1 to 7, we computed the intersection pattern for all occurring orbits using Magma (ancillary file or \cite{github_Existence_dP1}). We then determine which of the orbit combinations above occurs for each type, and we conclude that there is a one-to-one correspondence between the types 1 to 7 and the seven possible configurations of the singular fibers in Corollary \ref{cor:rybakov}. The correspondence is given in Table \ref{tab:pointsbelowsingularfibers}.

\begin{table}[h]
    \centering
   \begin{tabular}{|c|l|}
   \hline
    Type & Singular fibers over: \\
    \hline
    1 & five $\F_q$-rational points and one point of degree 2; \\
    2 & two $\F_q$-rational points, one point of degree 2 and one point of degree 3;\\
    3 & one $\F_q$-rational point and three points of degree 2;\\
    4 & three $\F_q$-rational points and one point of degree 4;\\
    5 & one $\F_q$-rational point and one point of degree 6;\\
    6 & one point of degree 2 and one point of degree 5;\\
    7 & one point of degree 3 and one point of degree 4.\\
    \hline
\end{tabular}
    \caption{}
    \label{tab:pointsbelowsingularfibers}
\end{table}

\begin{definition}\label{def:typeconicbundles}
    Let $f \colon X \ra \P^1$ be a minimal rational standard conic bundle of degree 1 over $\F_q$. Then we say $f$ is of \term{type} $1, \ldots, 7$, based on the configuration of its singular fibers as in Table \ref{tab:pointsbelowsingularfibers}. If $X$ is moreover a del Pezzo surface, this corresponds to our previous definition of type. 
\end{definition}

Recall from Theorem \ref{thm:twoconicbundles} that every minimal del Pezzo surface of degree 1 with a conic bundle structure has two distinct conic bundle structures. Because the type of a del Pezzo surface uniquely determines the degrees of the points over which the singular fibers of any conic bundle structure lie, this configuration has to be the same for both conic bundle structures. 

\begin{corollary}
    Let $X$ be a minimal del Pezzo surface of degree 1 over $\F_q$ with two distinct conic bundle structures. Then both are of the same type.
\end{corollary}

\begin{comment}
\end{comment}

\subsection{The Frobenius action as an element in $(\Z/2\Z)^{6} \rtimes S_7$}\label{sec:conjugacyclassfrob}

Let $f \colon X \ra \P^1$ be a minimal standard conic bundle over $\F_q$ of degree 1. The action of $\Gal(\overline{\F}_q/\F_q)$ on $\Pic(\overline{X})$ is cyclic. We write $\Gamma_X \colonequals \langle \rho_X^c(\Frob_q)\rangle$ for the subgroup of $(\Z/2\Z)^{6} \rtimes S_7 \subseteq \Aut(\Pic(\overline{X}))$ that is induced by this Galois action. In \S\ref{sec:Galoisorbitsfibers} we described the Galois orbits of the $(-1)$-curves within the singular fibers, depending on the type of $f$. This description allows us to determine $\rho_X^c(\Frob_q)$ as an element of $(\Z/2\Z)^{6} \rtimes S_7$ in terms of Notation \ref{not:W(D7)}, up to conjugation, and its order, given in the table below.
\begin{table}[h]
    \centering
    \begin{tabular}{c|c c c c c c c}
        Type & 1 & 2 &3 &4 &5 &6 & 7\\
        \hline
        order of Frobenius action & 4 & 12 & 4 & 8 & 12 & 20 & 24
    \end{tabular}
    \caption{}
    \label{tab:orders}
\end{table}

We also obtain the following lemma.

\begin{lemma}\label{lem:galoisorbitnn/2}
    Let $\Gamma_X = \langle \tau \rangle$, where $\tau$ is the Frobenius action. Let $n = \ord(\Gamma_X)$. Then $\tau^{n/2}$ is of the form
    \begin{enumerate}
        \item $\iota_{ij}$ if $f$ is of type 1, 2 or 6,
        \item $\iota_{ijkl}$ if $f$ is of type 4 or 7,
        \item $\iota_{ijklmn}$ if $f$ is of type 3 or 5, 
    \end{enumerate}
    for distinct $i, j, k, l, m, n \in \{1, \ldots, 7\}$.
\end{lemma}

\begin{proof}
    Note that each $(-1)$-curve in a singular fiber is fixed by $\tau^{n/2}$ if and only if its orbit size divides $n/2$. Otherwise, it is swapped with the other $(-1)$-curve in the same geometric fiber. 
\end{proof}

\subsection{Existence of minimal standard conic bundles over $\F_q$}\label{sec:nonexistence}

We finish this section with a first observation about the non-existence of del Pezzo surfaces of types 1 and 3 for small $q$.

\begin{lemma}\label{lem:smallfields}
    Let $\F_q$ be a finite field with $q$ elements. There exists a minimal rational standard conic bundle $X \ra \P^1$ over $\F_q$ with $K_X^2 = 1$ of each of the types 1 to 7, except for type 1 when $q = 2, 3$ and type 3 when $q = 2$. It follows that a del Pezzo surface of degree 1 of type 1 cannot be achieved over $\F_2$ and $\F_3$, and a del Pezzo surface of degree 1 of type 3 cannot be achieved over $\F_2$. 
\end{lemma}

\begin{proof}
    There are no five $\F_q$-points in $\P^1_{\F_q}$ for $q = 2, 3$, and there are no three points of degree 2 in $\P^1_{\F_2}$. The result follows directly from Corollary \ref{cor:rybakov}.
\end{proof}

\section{Existence of del Pezzo surfaces of types 1 to 7.}\label{sec:method1}

In what follows, we use the existence of a minimal standard rational conic bundle of degree 1 over $\F_q$ of each type in Table \ref{tab:pointsbelowsingularfibers}, granted by Lemma \ref{lem:smallfields}, to construct a birationally equivalent minimal standard conic bundle $X$ of the same type with $-K_X$ ample, that is, a minimal del Pezzo surface with a conic bundle structure.

\subsection{Elementary transformations}\label{sec:elementarytransformations}

We will repeatedly use the following type of transformation (see for example \cite[\S2.3]{sarkisov}). In this subsection, let $k$ be an arbitrary field.
\begin{definition}
    Let $f \colon X \ra B$ be a standard rational conic bundle over $k$. Let $P \in X$ be a closed point on a smooth fiber $F_1$ of $f$, chosen such that $f(P)$ and $P$ are points of the same degree. Let $\pi_P \colon X'\ra X$ be the blow-up of $X$ at the point $P$. The strict transform of $F_1$ on $\overline{X'}$ is a skew orbit of $(-1)$-curves. Let $\pi_{F_1} \colon X' \ra Y$ be the morphism obtained by contracting this orbit on $X'$. The birational map $\pi_{F_1} \circ \pi_P^{-1} \colon X \dashrightarrow Y$ is called the \term{elementary transformation centered at $P$}. We will denote this map by $\varphi_{\scriptscriptstyle P}$. Any birational map constructed in this way is called an \term{elementary transformation}.
\end{definition}
 
\begin{proposition}\label{prop:elementarytransformations}
    Let $f \colon X \ra B$ be a standard rational conic bundle over $k$. Let $\varphi \colon X \dashrightarrow Y$ be an elementary transformation. Then there is a morphism $f' \colon Y \ra B$ such that $f' \circ \varphi = f$, and $f'$ is again a standard rational conic bundle with $K_{Y}^2 = K_X^2$.
\end{proposition}

\begin{comment}
\end{comment}

By construction, the inverse of an elementary transformation is again an elementary transformation.

\begin{comment}
     Conversely, if $X$ and $Y$ are two minimal standard rational conic bundles which are birational, then there exists a birational map $X \dashrightarrow Y$ which is the composition of a finite number of elementary transformations. 
\end{comment}

\begin{definition}
    Let $f \colon X \ra B$ be a conic bundle over $k$. Let $D$ be an irreducible curve on $X$. If $D$ is not a (component of a) fiber of $f$, then $D \cdot F = n \geq 1$. We call $D$ a \term{multisection}, or, more specifically, an \term{$n$-section}, of the conic bundle. If $n = 1$ we call $D$ a \term{section}, and if $n = 2$ we call $D$ a \term{bisection}. 
\end{definition}

\begin{lemma}\label{lem:nsection}
    Let $f \colon X \ra B$ be a standard rational conic bundle over $k$, and let $P$ be a rational point on a smooth fiber $F_1$ of $f$. Consider the elementary transformation $\varphi_{\scriptscriptstyle P} = \pi_{F_1} \circ \pi_P^{-1} \colon X \dashrightarrow Y$ at $P$. Let $D$ be an $n$-section on $X$. Let $\pi_P^{-1}(D)$ denote its strict transform under the morphism $\pi_P$. Then $\pi_{F_1}(\pi_P^{-1}(D))$ is an $n$-section on the conic bundle $f \circ \varphi_{\scriptscriptstyle P}^{-1} \colon Y \ra B$. We will denote this curve by $\varphi_{\scriptscriptstyle P}(D)$.
\end{lemma}

\begin{proof}
    Let $m = m_P(D)$ be the multiplicity of $D$ at $P$. Note that $m \leq n$. Let $F_1'$ and $D'$ be the strict transforms on $X'$ of $F_1$, the fiber containing $P$, and $D$, respectively, and let $E_P$ be the exceptional divisor above $P$. Then $\pi_P^*(F_1) = F_1' + E_P$ and $\pi_P^*(D) = D' + m E_P$. We obtain $D' \cdot E_P = m$ and $D' \cdot F_1' = n - m$. Let $\pi_{F_1}(D') = D''$ and $\pi_{F_1}(E_P) = F_2$. Then $F_2$ is a fiber of the conic bundle $f \circ \varphi_{\scriptscriptstyle P}^{-1}$. We obtain $\pi_{F_1}^*(D'') = D' + (n-m) F_1'$ and $\pi_{F_1}^*(F_2) = E_P + F_1'$. We conclude that $D'' \cdot F_2 = n$, so $D''$ is an $n$-section of the conic bundle $f \circ \varphi_{\scriptscriptstyle P}^{-1}$.
\end{proof}

\begin{comment}
\end{comment}

\begin{proposition}\label{prop:preserveminimaltype}
    Let $f \colon X \ra B$ be a standard rational conic bundle of degree 1 over $k$, and let $\varphi \colon X \dashrightarrow Y$ be an elementary transformation. If $X$ is minimal, then so is $Y$. If moreover $k = \F_q$, then $f \colon X \ra \P^1$ and $f \circ \varphi^{-1} \colon Y \ra \P^1$ are conic bundles of the same type (Definition \ref{def:typeconicbundles}).
\end{proposition}

\begin{proof}
    Let $f' \colon Y \ra B$ be the standard rational conic bundle such that $f' \circ \varphi = f$. Note that $\varphi$ is an isomorphism outside one smooth fiber of $f$ and $f'$, and in particular, the Galois orbits in the singular fibers of $f$ and $f'$ are the same. If $X$ is minimal, it is relatively minimal, and then so is $Y$. By Proposition \ref{prop:relativelyminimal-minimal}, $Y$ is minimal. Furthermore, the singular fibers of $f$ and $f \circ \varphi^{-1}$ lie above the same points on $\P^1$. This uniquely characterizes the type of the conic bundles (see Table \ref{tab:pointsbelowsingularfibers}).
\end{proof}

It follows that given a minimal standard rational conic bundle of degree 1 of type $t \in \{1, \ldots, 7\}$, we can use elementary transformations to obtain another minimal standard rational conic bundle of type $t$ with desired properties. We use this to construct del Pezzo surfaces of degree 1 of the types $1, \ldots, 7$.

\subsection{Characterization of minimal standard conic bundles of degree 1}\label{sec:classificationCB}

The following proposition gives a characterization of curves that must exist on a minimal standard rational conic bundle of degree 1, in the case where it is not a del Pezzo surface. The statement is analogous to \cite[Proposition 2.3]{trepalindeg2} in the degree 2 case. If a minimal standard conic bundle of degree 1 of a given type exists by Lemma \ref{lem:smallfields}, then it has to satisfy one of the five cases of Proposition \ref{prop:classificationCB}.

In what follows, we use the notation $[A]$ for the class of a divisor $A \in \Div(X)$ in the Picard group, and $\overline{f}$ for the base change of a morphism $f$ over a field $k$ to the algebraic closure $\overline{k}$. 

\begin{proposition}\label{prop:classificationCB}
    Let $f \colon X \ra B$ be a minimal standard rational conic bundle of degree $1$ over a perfect field $k$. Then we have one of the following cases.
    \begin{enumerate}
        \item $-K_X$ is not nef, and there is a geometrically integral bisection $A$ on $X$ such that $-K_X \cdot [A] = -1$ and $A^2 = -3$.
        \item $-K_X$ is nef but not ample, so there is an integral curve $A$ on $X$ such that $-K_X \cdot [A] = 0$. Its base change to $\overline{X}$ is of one of the following forms:
        \begin{enumerate}
            \item $A = C_1 + \cdots + C_8$, where $C_1, \ldots, C_8$ are disjoint irreducible sections on $\overline{X}$, each with self-intersection $-2$.
        
            \item $A = C_1 + \cdots + C_4$, where $C_1, \ldots, C_4$ are irreducible sections on $\overline{X}$, each with self-intersection $-2$, such that $C_i \cdot C_j = 1$ if $i \neq j$ and $i \equiv j \mod 2$, and $C_i \cdot C_j = 0$ if $i \not \equiv j \mod 2$.

            \item $A = C_1 + C_2$, where $C_1$ and $C_2$ are disjoint irreducible bisections on $\overline{X}$ with self-intersection $-2$.
        \end{enumerate}
         \item The surface $X$ is a del Pezzo surface of degree 1 admitting exactly two conic bundle structures.
    \end{enumerate}
\end{proposition}

 To prove this proposition, we separately consider the cases where $-K_X$ is not nef (Lemma \ref{lem:-KXnotnef}) and where $-K_X$ is nef but not ample (Lemma \ref{lem:-KXnef}). When $-K_X$ is not nef, the following characterization follows from \cite[Lemma 19]{kollarmella}.

    \begin{lemma}\label{lem:-KXnotnef}
    Let $f \colon X \ra B$ be a minimal standard rational conic bundle of degree $1$ over a perfect field $k$ such that $-K_X$ is not nef. Then there is precisely one geometrically integral curve $A$ on $X$ with $-K_X \cdot [A] < 0$. It satisfies \[[A] = -K_X - F,\] where $F$ is the class of a fiber of $f$, so $A^2 = -3$. Any other integral curve $D \subset \overline{X}$ satisfies $D^2 \geq -2$. 
\end{lemma}

\begin{proof}
    In \cite[Lemma 19]{kollarmella}, it is shown that $|-K_X| = A + |F|$, where $A$ is a geometrically integral curve with self-intersection $-3$. This curve satisfies $-K_X \cdot [A] = -K_X \cdot (-K_X - F) = -1 < 0$. 
    Let  $D \subset \overline{X}$ be an integral curve satisfying $-K_{\overline{X}} \cdot [D] < 0$. Then $A \cdot D = -K_{\overline{X}} \cdot [D] - F \cdot [D] < 0$, so $D = A$. If, on the other hand, $D \subset \overline{X}$ is any integral curve with $-K_{\overline{X}} \cdot [D] \geq 0$, the adjunction formula gives $D^2 \geq 2p_a(D) - 2 \geq -2$.
\end{proof}
   
   \begin{comment}
        We have $s(0) = 0$ because it is the irregularity of $X$ (see Castelnuovo's rationality criterion), $\ell(K_X) = 0$ because it is the first plurigenus of $X$, and $\ell(0) = 1$ \cite[Chapter 3, Corollary 3.21]{liu}.
        $\ell(2K_X) = 0$ because it is the second plurigenus of $X$.
    \end{comment}

\begin{comment}
\begin{proposition}\label{prop:classificationCB}
    Let $\pi \colon X \ra \P^1$ be a minimal 
    \begin{enumerate}
        \item The surface $X$ is a del Pezzo surface of degree 1 admitting two conic bundle structures.
        \item $-K_X$ is nef, and there exist two disjoint, Galois-conjugate, geometrically irreducible bisections with self-intersection $-2$.
        \item $-K_X$ is nef, and there exist four Galois-conjugate, geometrically irreducible sections $C_1, C_2, C_3$ and $C_4$, each with self-intersection number $-2$, such that $C_i \cdot C_j = 1$ if $i \neq j$ and $i \equiv j \mod 2$, and $C_i \cdot C_j = 0$ if $i \not \equiv j \mod 2$.
        
        \item$-K_X$ is nef, and there is a cyclic Galois orbit consisting of eight disjoint geometrically irreducible sections, each with self-intersection $-2$.
        
        \item $K_X$ is not nef, and there is an irreducible bisection $B$ on $X$ such that $B^2 = -3$ and $-K_X \cdot B = -1$.
    \end{enumerate}
\end{proposition}
\end{comment}

\begin{lemma}\label{lem:-KXnef}
    Let $f \colon X \ra B$ be a minimal standard rational conic bundle of degree $1$ over a perfect field $k$ such that $-K_X$ is nef, but not ample. Then there is precisely one integral curve $A$ on $X$ with $-K_X \cdot [A] = 0$. On $\overline{X}$, we have \[A = C_1 + \cdots + C_n\] for some geometrically integral $C_i \subset \overline{X}$ with $C_i^2 = -2$ for $i = 1, \ldots, n$. More specifically, we obtain one of the following three possibilities, where $F$ is the class of a fiber of $f$:
    \begin{enumerate}
            \item $n = 8$, $C_i \cdot F = 1$ and $C_i \cdot C_j = 0$ for all $i \neq j$.
        
            \item $n = 4$, $C_i \cdot F = 1$ for $i = 1, \ldots, 4$ and $C_i \cdot C_j = \begin{cases}
                1 &\text{if $\{i, j\} \in \{ \{1, 3\}, \{2, 4\}\}$;}\\
                0 &\text{otherwise.}
            \end{cases}$
            
            \item $n = 2$, $C_1 \cdot F = C_2 \cdot F = 2$, and $C_1 \cdot C_2 = 0$.
        \end{enumerate}
        
        Moreover, for all integral curves $D \notin \{C_1, \ldots, C_n\}$ on $\overline{X}$, we have $D^2 \geq -1$.
\end{lemma}

\begin{proof}
    Because $-K_X$ is nef but not ample, there is an integral curve $A$ on $X$ such that $-K_X \cdot [A] = 0$ by the Nakai--Moishezon Criterion \cite[V, Theorem 1.10]{hartshorne}. Let us write $A = C_1 + \cdots + C_n$, where $C_1, \ldots, C_n$ are distinct integral curves defined on $\overline{X}$. Then the curves $C_i$ form a Galois orbit. Because the Galois action preserves the intersection pairing, the values of $C_i^2$, $-K_{\overline{X}} \cdot [C_i]$ and $[C_i] \cdot F$ are independent of $i$. Let us write $[A] = -aK_X - bF$ (this is a unique representation by Lemma \ref{lem:relativelyminalrank2}). Then $-K_X \cdot [A] = a -2b = 0$, so $a = 2b$. We note that \[n ([C_1] \cdot F) = [A] \cdot F = 2a,\] so $n \mid 2a$ and $a > 0$. Furthermore, $-K_X \cdot [A] = \sum_{i=1}^n (-K_{\overline{X}} \cdot [C_i]) = - n K_{\overline{X}} \cdot [C_1] = 0$, so $K_{\overline{X}} \cdot [C_i] = 0$ for all $i = 1, \ldots, n$. By the adjunction formula, we get $2 p_a(C_i) - 2 = C_i^2$, so $C_i^2 \geq -2$ is even. We have $n C_i^2 = \sum_{i=1}^n C_i^2 \leq A^2 = a^2 - 4ab = -a^2 < 0$, so $C_i^2 = -2$. We conclude $a^2 = -A^2 \leq 2n \leq 4a$, so $1 \leq a \leq 4$. Furthermore, note that $a^2 = -A^2 = 2n - 2\sum_{i < j} C_i \cdot C_j$ is even, so $a \in 2, 4$.

    If $a = 4$, the constraints $a^2 \leq 2n$ and $n \mid 2a$ together imply $n = 8$. We then must have $C_i \cdot C_j = 0$ for all $i \neq j$, so $A$ consists of $8$ disjoint geometrically integral sections (note that $[C_i] \cdot F = \frac{1}{8} ([A] \cdot F) = \frac{1}{8} \cdot 2a = 1$).

    If $a = 2$, the constraints $a^2 \leq 2n$ and $n \mid 2a$ imply that $n \in \{2, 4\}$. Note that $2\sum_{i<j} C_i\cdot C_j = 2n - a^2$. If $n = 2$, we obtain $C_1 \cdot C_2 = 0$. We get $2 [C_i] \cdot F = [A] \cdot F = 4$, so the two disjoint curves $C_i$ are bisections of the conic bundle. If $n = 4$, we get $2n -a^2 = 4$. Because the Galois action on the $C_i$ is transitive, the only possibility is if the $C_i$ come in two disjoint pairs of pairwise intersecting curves with intersection number 1. In this case, we have $4 ([C_i] \cdot F) = [A] \cdot F = 4$, so the $C_i$ are sections of the conic bundle. 

    Suppose $D \in \Div(X)$ is another integral curve satisfying $-K_X \cdot [D] = 0$. Then for the same reason, we must have $[D] = -b'(2K_X + F)$ for some $b' \in \Z_{>0}$. We get $A \cdot D = b b' (2K_X + F)^2 < 0$, but this is not possible because $A$ and $D$ are distinct irreducible curves.

    Let $D_1 \in \Div(\overline{X})$ be an integral curve, and let $D = D_1 + \cdots + D_n \in \Div(X)$ be the irreducible curve on $X$ containing $D_1$ as a component. If $-K_{\overline{X}} \cdot D_1 \leq 0$, then also $-K_X \cdot D \leq 0$, so this is only possible if $D_1$ is a component of $A$. Otherwise, we have $-K_{\overline{X}} \cdot D_1 \geq 1$, and by adjunction we obtain $D_1^2 = 2p_a(D_1) - 2 - K_{\overline{X}} \cdot D_1 \geq -1$.
\end{proof}

\begin{proof}[Proof of Proposition \ref{prop:classificationCB}]
    If $-K_X$ is not nef, we obtain case (1) by Lemma \ref{lem:-KXnotnef}. If $-K_X$ is nef, but not ample, we obtain case (2) by Lemma \ref{lem:-KXnef}. If $-K_X$ is ample, then $X$ is a del Pezzo surface of degree 1, which has two distinct conic bundle structures by Theorem \ref{thm:twoconicbundles}.
\end{proof}

Recall that Lemma \ref{lem:smallfields} tells us that for each of the types 1 to 7, there exists a minimal standard rational conic bundle of this type over $\F_q$, except for type 1 when $q = 2, 3$ and type 3 when $q = 2$. Such a conic bundle must satisfy one of the five characterizations listed in Proposition \ref{prop:classificationCB}. In the rest of this section, we consider how each of the cases in this proposition is compatible with the seven different types of conic bundles. We will separately consider a minimal standard rational conic bundle of each of the five cases, and find elementary transformations to a minimal standard rational conic bundle of the same type which is a del Pezzo surface, thus showing the existence of del Pezzo surfaces of each type for most values of $q$.

\subsection{From case (1) to case (2) and (3)}\label{sec:case(1)}
We consider a minimal standard conic bundle $f \colon X \ra \P^1$ of degree 1 over $\F_q$, such that $-K_X$ is not nef. Let $F \in \Pic(X)$ denote the class of a fiber of $f$. We are in case (1) of Proposition \ref{prop:classificationCB}, so there is an integral bisection $A$ of $f$ on $X$ with $A^2 = -3$. From Lemma \ref{lem:-KXnotnef}, we know that $[A] = -K_X - F$.

\begin{comment}
In the basis $\{C, F, E_1, \ldots, E_7\}$ of the geometric Picard group, let us write \[B = cC + aF - \sum_{i=1}^7 b_i E_i.\] Because $B$ is a bisection, we have $B \cdot F = 2$ and thus $c = 2$. Furthermore, because $B$ and $C$ are classes of effective divisors without common components, we have $B \cdot C \geq 0$, so $a \geq 0$. Moreover, because each $E_i$ is a component of a fiber, we get $ 0 \leq B \cdot E_i \leq 2$, and hence $0 \leq b_i \leq 2$ for all $i = 1, \ldots, 7$. Recall that $K_X = -2C - 2F + \sum_{i=1}^7 E_i$, so $-K_X \cdot B = 2a + 4 - \sum_{i=1}^7 b_i = -1$. We also have $B^2 = 4a - \sum_{i=1}^7 b_i^2 = - 3$.
We thus need to find a tuple $(a, b_1, \ldots, b_7)$ satsifying $a \geq 0$ and $0 \leq b_i \leq 2$, and
\begin{align}
    \sum_{i=1}^7 b_i = 2a + 5, \label{eq:sumbi}\\
    \sum_{i=1}^7 b_i^2 = 4a + 3. \label{eq:sumbi2}
\end{align}
Because each $b_i \leq 2$, we get $a \leq 4$ from \eqref{eq:sumbi}. Under these restrictions, the only possible solution is \[[B] = 2C + F - E_1 - \cdots - E_7 = -K_X - F.\]
\end{comment}

\begin{proposition}\label{prop:(1)}
    Let $f \colon X \ra \P^1$ be a minimal standard conic bundle over $\F_q$ of degree 1. Suppose there is an $\F_q$-point $P$ on a smooth fiber of $f$ such that the following two conditions hold:
    \begin{enumerate}[label=(\alph*)]
         \item $P$ does not lie on a bisection with self-intersection $-3$;
        \item $P$ is not a singular point on a bisection with self-intersection $1$.
    \end{enumerate}
     Let $\varphi_{\scriptscriptstyle P} \colon X \dashrightarrow Y$ be the elementary transformation centered at $P$. Then $f \circ \varphi_{\scriptscriptstyle P}^{-1} \colon Y \ra \P^1$ is a minimal standard conic bundle of degree 1 such that $-K_Y$ is nef.
\end{proposition}

\begin{proof}
    By Proposition \ref{prop:elementarytransformations} and Proposition \ref{prop:preserveminimaltype}, $f \circ \varphi_{\scriptscriptstyle P}^{-1} \colon Y \ra \P^1$ is a minimal standard conic bundle of degree 1. Suppose that $-K_Y$ is not nef. Then by Proposition \ref{prop:classificationCB}, there is an irreducible bisection $A$ on $Y$ which is defined over $\F_q$, with $A^2 = -3$. By Lemma \ref{lem:nsection}, $\varphi_{\scriptscriptstyle P}^{-1}(A)$ is an integral bisection on $X$. Let $D$ be any geometrically integral bisection on $X$. Because $[D] \cdot F = 2$, $D$ has multiplicity at most two at $P$. Note that \[\varphi_{\scriptscriptstyle P}(D)^2 = \begin{cases}
        D^2 - 4 &\text{if $P \in \supp(D)$ with multiplicity 2},\\
        D^2 &\text{if $P \in \supp(D)$ with multiplicity 1},\\
        D^2 + 4 &\text{if} \ P \notin \supp(D).
    \end{cases}\]
    From Lemma \ref{lem:-KXnotnef} and Lemma \ref{lem:-KXnef}, we know that $D^2 \geq -3$, so $\varphi_{\scriptscriptstyle P}(D)^2 = -3$ can only happen in two cases: either $D^2 = -3$ and $P \in \supp(D)$, or $D^2 = 1$ and $P$ is a singular point on $D$. By assumption, these two cases are not satisfied, so we conclude that $f \circ \varphi_{\scriptscriptstyle P}^{-1}$ does not have any geometrically integral bisections defined over $\F_q$ of self-intersection $-3$. In particular, $f \circ \varphi_{\scriptscriptstyle P}^{-1}$ cannot satisfy case (1) of Proposition \ref{prop:classificationCB}, so it must satisfy case (2) or (3), which implies that $-K_Y$ is nef.
\end{proof} 

\begin{theorem}\label{thm:(1)to(2,3)}
    Let $f \colon X \ra \P^1$ be a minimal standard conic bundle over $\F_q$ of degree 1 such that $-K_X$ is not nef, that is, $f$ satisfies case (1) in Proposition \ref{prop:classificationCB}. Unless $q = 2$ and $f$ is of type 4 or $q = 4$ and $f$ is of type 1, there is an elementary transformation $\varphi \colon X \dashrightarrow Y$ such that $f \circ \varphi^{-1} \colon Y \ra \P^1$ is a minimal standard conic bundle of degree 1 of the same type as $f$, satisfying case (2) or (3) in Proposition \ref{prop:classificationCB}.
\end{theorem}

\begin{proof}
    The conic bundle $f$ has a smooth fiber defined over $\F_q$, except when $q = 2$ and $X$ is of type 4, or when $q = 4$ and $X$ is of type 1. Suppose this is the case, and let $Q \in \P^1(\F_q)$ such that $f^{-1}(Q)$ is a smooth fiber defined over $\F_q$. We saw in Lemma \ref{lem:-KXnotnef} that there is a bisection $A$ of $f$ on $X$ such that $[A] = -K_X-F$. The fiber $f^{-1}(Q)$ has at least three rational points, and $A \cdot f^{-1}(Q) = [A] \cdot F = 2$, so there is a point $P \in f^{-1}(Q)(\F_q)$ such that $P$ does not lie on $A$. Let $\varphi_{\scriptscriptstyle P} \colon X \dashrightarrow Y$ be the elementary transformation centered at $P$. By Propositions \ref{prop:elementarytransformations} and \ref{prop:preserveminimaltype}, $f \circ \varphi_{\scriptscriptstyle P}^{-1} \colon Y \ra \P^1$ is a minimal standard conic bundle of degree 1 of the same type as $f$. Let $D$ be a bisection of $f$ on $X$ which is defined over $\F_q$, different from $A$. By Lemma \ref{lem:relativelyminalrank2}, we can write $[D] = -a K_X - bF \in \Pic(X)$ for some $a, b \in \Z$. Then $[D] \cdot F = 2a = 2$, so $a = 1$. Because $D$ is distinct from $A$, we must have $D \cdot A = -2b-1 \geq 0$, so $b < 0$. It follows that $D^2 = 1 - 4b \geq 5$. By Proposition \ref{prop:(1)}, the map $\varphi_{\scriptscriptstyle P}$ is the desired elementary transformation.
\end{proof}

This allows us to prove Theorem \ref{thm:weakdPexistence} about the existence of minimal weak del Pezzo surfaces of degree 1 with conic bundles. 

\begin{proof}[Proof of Theorem \ref{thm:weakdPexistence}]
    This follows from Lemma \ref{lem:smallfields} and Theorem \ref{thm:(1)to(2,3)}.
\end{proof}



\subsection{Case (2): $-K_X$ is nef}

In what follows, we consider a minimal standard conic bundle $f \colon X \ra \P^1$ over $\F_q$ of degree 1 with nef anticanonical divisor, and investigate when we can use elementary transformations to obtain a del Pezzo surface of degree 1. Our first strategy is to find an elementary transformation of $X$ such that the resulting surface is a minimal standard conic bundle which satisfies case (2c) or (3). We will show that for $q$ large enough, there exists a point $P$ on $X$ such that the resulting surface has no geometrically integral bisection of self-intersection $-3$ defined over $\F_q$, and no integral sections with self-intersection $-2$ over $\overline{\F}_q$. As we see in the following proposition, we can achieve this by choosing a rational point which does not lie on any curve of self-intersection $-1$ on $\overline{X}$, and which is not a singular point of an \term{anticanonical curve}, that is, a curve in the linear system $|-K_X|$.

\begin{proposition}\label{prop:(3)(4)to(1)(2)}
    Let $f \colon X \ra \P^1$ be a minimal standard conic bundle over $\F_q$ of degree 1 such that $-K_X$ is nef. Suppose there is a point $P \in X(\F_q)$ on a smooth fiber of $f$ such that
    \begin{enumerate}[label=(\alph*)]
        \item $P$ is not a singular point of an anticanonical curve;
        \item $P$ does not lie on any geometric section with negative self-intersection.
    \end{enumerate}
    Let $\varphi_{\scriptscriptstyle P} \colon X \dashrightarrow Y$ be the elementary transformation centered at $P$. Then $f \circ \varphi_{\scriptscriptstyle P}^{-1} \colon Y \ra \P^1$ is a minimal standard conic bundle of degree 1 of the same type as $f$, satisfying case (2c) or (3) in Proposition \ref{prop:classificationCB}. 
\end{proposition}

\begin{proof}
    By Proposition \ref{prop:elementarytransformations} and Proposition \ref{prop:preserveminimaltype}, the morphism $f \circ \varphi_{\scriptscriptstyle P}^{-1} \colon Y \ra \P^1$ defines a minimal standard conic bundle of degree 1 of the same type as $f$. Because $-K_X$ is nef, the adjunction formula implies that all integral curves on $\overline{X}$ have self-intersection at least $-2$. Note that $P$ does not lie on any of the geometric sections with negative self-intersection. By Lemma \ref{lem:nsection}, every section on $\overline{Y}$ comes from a section on $\overline{X}$, so we conclude that all integral sections on $\overline{Y}$ have self-intersection at least $-1$. It follows that $Y$ does not satisfy case (2a) or (2b) of Proposition \ref{prop:classificationCB}. Furthermore, let $D$ be a geometrically integral bisection on $X$. Then $D^2 \geq 2$, and thus $D = -K_X + b F$ for some $b \in \Z_{\geq 0}$. We have $D^2 = 1$ if and only if $b = 0$, in which case $D$ is an anticanonical curve and $P$ is not a singular point on $D$ by assumption. It follows from Proposition \ref{prop:(1)} that $Y$ does not satisfy case (1) of Proposition \ref{prop:classificationCB}.
\end{proof}

To further find an elementary transformation that eliminates case (2c) and thus results in a del Pezzo surface, we use the following observation.

\begin{lemma}\label{lem:(2c)to(2c)}
    Let $f \colon X \ra \P^1$ be a minimal standard conic bundle over $\F_q$ of degree 1 such that $-K_X$ is nef. Let $P$ be an $\F_q$-point on a smooth fiber of $f$, and let $\varphi_{\scriptscriptstyle P} \colon X \dashrightarrow Y$ be the elementary transformation centered at $P$. Suppose $f \circ \varphi_{\scriptscriptstyle P}^{-1} \colon Y \ra \P^1$ is a minimal standard conic bundle which satisfies case (2c) in Proposition \ref{prop:classificationCB}. 
    Then $P$ is a singular point on two integral bisections of $\overline{f} \colon \overline{X} \ra \P^1$ of self-intersection $2$ and arithmetic genus 1, which form a Galois orbit of size 2. 
\end{lemma}

\begin{proof}
    Because $f \circ \varphi_{\scriptscriptstyle P}^{-1}$ satisfies case (2c) of Proposition \ref{prop:classificationCB}, $\overline{Y}$ contains a Galois orbit of two integral bisections $B_1'$ and $B_2'$ of $\overline{f \circ \varphi_{\scriptscriptstyle P}^{-1}}$, such that $B_i'^2 = -2$ and $-K_{\overline{Y}} \cdot [B_i'] = 0$ for $i = 1, 2$. The preimages under $\overline{\varphi_{\scriptscriptstyle P}}$ of $B_1'$ and $B_2'$ are integral bisections of $\overline{f}$ on $\overline{X}$ by Lemma \ref{lem:nsection}. Let $D$ be an integral bisection of $\overline{f}$. Note that $P$ can have multiplicity at most $2$ on $D$, because $[D] \cdot F = 2$. Then the self-intersection of $\overline{\varphi_{\scriptscriptstyle P}}(D)$ on $\overline{Y}$ is as follows:
    \begin{align*}
        \overline{\varphi_{\scriptscriptstyle P}}(D)^2 = \begin{cases}
             D^2 + 4 &\text{if $P$ does not lie on $D$;}\\
             D^2 &\text{if $P$ is a point of multiplicity $1$ on $D$;}\\
             D^2 - 4 &\text{if $P$ is a point of multiplicity $2$ on $D$.}\\ 
        \end{cases}
    \end{align*}
Let us denote by $D_i$ the preimage of $B_i'$ for $i = 1, 2$. Then $D_i^2 \geq -2$ by Lemma \ref{lem:-KXnef}, and $(B_i')^2 = -2$ yields that $D_i^2 \in \{-2, 2\}$ and $P$ lies on $D_i$. By Lemma \ref{lem:-KXnef}, there are no irreducible bisections with self-intersection $-2$ on $\overline{X}$ which contain rational points. We must conclude that $D_i^2 = 2$ and $P$ has multiplicity 2 on $D_i$. This yields $p_a(D_i) = p_a(B_i') + 1 = 1$.
\end{proof}

We conclude that the following properties of an $\F_q$-rational point $P$ are sufficient to induce an elementary transformation $\varphi_{\scriptscriptstyle P}$ resulting in a del Pezzo surface.

\begin{proposition}\label{prop:(2c)to(3)}
    Let $f \colon X \ra \P^1$ be a minimal standard conic bundle over $\F_q$ of degree 1 such that $-K_X$ is nef. Suppose there is a point $P \in X(\F_q)$ on a smooth fiber of $f$ such that $P$ satisfies the following three conditions:
    \begin{enumerate}[label=(\alph*)]
        \item $P$ is not a singular point of an anticanonical curve;
        \item $P$ is not a singular point on two integral bisections of $\overline{f} \colon \overline{X} \ra \P^1$ of self-intersection $2$ and arithmetic genus 1, which form a Galois orbit of size 2;\label{item:(2c)to(3)part(b)}
        \item $P$ does not lie on any geometric section with negative self-intersection.\label{item:(2c)to(3)part(c)}
    \end{enumerate}
    Let $\varphi_{\scriptscriptstyle P} \colon X \dashrightarrow Y$ be the elementary transformation centered at $P$. Then $Y$ is a del Pezzo surface of degree 1, and $f \circ \varphi_{\scriptscriptstyle P}^{-1} \colon Y \ra \P^1$ is a minimal standard conic bundle of the same type as $f$.
\end{proposition}

\begin{proof}
    This follows from Proposition \ref{prop:(3)(4)to(1)(2)} and Lemmas \ref{lem:(2c)to(2c)} and \ref{lem:(2c)selfintersection2classes}.
\end{proof}

To find a point $P$ which satisfies the first condition in this proposition, we describe the distribution of singular points on curves in the anticanonical linear system on the fibers of $f$. 

\begin{lemma}\label{lem:(2)singpointsanticanonical}
    Let $f \colon X \ra \P^1$ be a minimal standard conic bundle over $\F_q$ of degree 1 such that $-K_X$ is nef. Then there are at most two singular points of anticanonical curves on a smooth fiber of $f$. If the fiber contains the base point of $|-K_X|$, it cannot contain a singular point of an anticanonical curve.
\end{lemma}

\begin{proof}
    The surface $X$ is a weak del Pezzo surface, and thus we have $\ell(-K_X) =  2$ (by the Kawamata–Viehweg vanishing theorem, \cite{kawamata, viehweg}) and the anticanonical linear system induces a rational map $\phi \colon X \dashrightarrow \P^1$. Because $K_X^2 = 1$, the linear system $|-K_X|$ has a unique base point $P_B$. Let $Q \in \P^1(\F_q)$ such that $f^{-1}(Q)$ is a smooth fiber.
    
    Let us first suppose that $f^{-1}(Q)$ does not contain $P_B$. Then $\phi|_{f^{-1}(Q)} \colon f^{-1}(Q) \ra \P^1$ is a morphism, and because $-K_X \cdot [f^{-1}(Q)] = 2$, it is of degree 2. By the Riemann-Hurwitz formula, we get that $\sum_{R \in f^{-1}(Q)} (e_R - 1) = 2$, so $\phi|_{f^{-1}(Q)}$ ramifies in two points with ramification index 2. Let $D \in |-K_X|$ such that $D$ has a singularity which lies on $f^{-1}(Q)$. Then this is the only intersection point of $D$ and $f^{-1}(Q)$, and thus $\phi|_{f^{-1}(Q)}$ ramifies above the point $\phi(D) \in \P^1(\F_q)$. This shows that $f^{-1}(Q)$ can contain at most two points which are singular points of anticanonical curves. 
    
    Now suppose $f^{-1}(Q)$ contains the base point $P_B$, and suppose there exists $D \in |-K_X|$ such that $D$ has a singularity which lies on $f^{-1}(Q)$. Then, because $D \cdot f^{-1}(Q) = 2$ and $P_B$ lies on $D$, we must have that $P_B$ is the singular point of $D$. But this is not possible, because $(-K_X)^2 = 1$, and $D$ would intersect all other anticanonical curves with intersection number at least 2 in $P_B$. 
\end{proof}

We want to similarly bound the number of points on a fiber that do not satisfy the second condition in Proposition \ref{prop:(2c)to(3)}. We characterize the Picard classes of pairs of integral Galois-conjugate bisections of self-intersection $2$ and arithmetic genus 1, to obtain a bound for the number of points on a fiber which are singular points on such curves.

\begin{lemma}\label{lem:(2c)selfintersection2classes}
    Let $f \colon X \ra \P^1$ be a minimal standard conic bundle over $\F_q$ of degree 1. Let $\{D_1, D_2\}$ be a Galois orbit of irreducible bisections of $\overline{f} \colon \overline{X} \ra \P^1$ of self-intersection 2 and arithmetic genus 1. Then there is a singular fiber of $\overline{f}$ above an $\F_q$-point, with components $E$ and $F-E$, such that the classes of $D_1$ and $D_2$ in $\Pic(\overline{X})$ are of the form $-K_{\overline{X}} + E$ and $-K_{\overline{X}} + (F-E)$.
\end{lemma}

\begin{proof}
    Using a basis of $\Pic(\overline{X})$ as introduced in Theorem \ref{thm:MT2.2.2} (such that $C$ is the class of a section, see Remark \ref{rem:sectionC}), let us write $[D_i] = 2C + aF - \sum_{n=1}^7 b_n E_n$, where $a \geq 0$ (because $C$ is effective) and $0 \leq b_n \leq 2$ (because $D_i$ is a bisection). Then $D_i^2 = 4a - \sum_{n=1}^7 b_n^2 = 2$, so \[\sum_{n=1}^7 b_n^2 = 4a - 2.\] By adjunction, we get $[D_i] \cdot K_{\overline{X}} = -2$, so \[\sum_{n=1}^7 b_n = 2a + 2.\] The only classes satisfying these two equations are
\begin{align*}
    2C + 2F - \sum_{n=1}^7 E_n + E_s &= -K_{\overline{X}} + E_s,\\
    2C + 3F - \sum_{n=1}^7 E_n - E_s &= -K_{\overline{X}} +(F - E_s)
\end{align*}
for $s \in \{1, \ldots, 7\}$. Because $D_1$ and $D_2$ form a Galois orbit of size 2, we obtain \[\{[D_1], [D_2]\} = \{-K_{\overline{X}} + E_s, -K_{\overline{X}} + (F - E_s)\},\] where $s \in \{1, \ldots, 7\}$ is such that the fiber of $\overline{f}$ containing $E_s$ is Galois-stable.
\end{proof}

The number of possibilities for the class $E$ in Lemma \ref{lem:(2c)selfintersection2classes} is equal to the number of singular fibers that lie above an $\F_q$-point, which is determined by the type of $f$ (see Table \ref{tab:pointsbelowsingularfibers}).

Consider a singular fiber $E + (F-E)$ of $f$ above an $\F_q$-rational point in $\P^1$. We can contract the $(-1)$-curve $E$ over $\F_{q^2}$, and the resulting surface $Y$ is a weak del Pezzo surface of degree 2. We can relate the singular points on a curve in the linear system $|-K_{\overline{X}} + E|$ to the singular points on a curve in $|-K_Y|$ on $Y$. The exact result depends on which case of Proposition \ref{prop:classificationCB} is satisfied by $X$, and on which $(-1)$-curve is contracted. We will use the following statement about conic bundles of degree 2. This result is partially proven in \cite[Proposition 2.10]{trepalindeg2}. 

\begin{proposition}\label{prop:anticanonicalcurvesdP2}
    Let $f \colon Y \ra \P^1$ be a standard conic bundle over $\F_q$ of degree 2, such that $-K_Y$ is nef, and let $F \in \Pic(Y)$ be the class of a fiber of $f$. Then the singular points of anticanonical curves lie on a multisection $R$ on $Y$. Moreover, 
    \begin{enumerate}[label=(\alph*)]
        \item If $Y$ is a del Pezzo surface, $[R] \cdot F = 4$ and there are at most four points on each fiber of $\overline{f} \colon \overline{X} \ra \P^1$ which are singular points on anticanonical curves.\label{item:dP}
        \item If $Y$ is not a del Pezzo surface, but if there is a unique integral curve $D$ on $Y$ such that $D \in |-K_Y - F|$, then $[R] \cdot F = 2$ and there are at most two points on each fiber of $\overline{f} \colon \overline{X} \ra \P^1$ which are singular points on anticanonical curves.\label{item:notdP}
    \end{enumerate}
\end{proposition}

\begin{proof}
         Let us first consider case \ref{item:dP}. Since $Y$ is a del Pezzo surface, the map $\phi \colon Y \ra \P^2$ induced by $|-K_Y|$ is a finite morphism of degree 2, branched over a quartic curve $B$ in $\P^2$. Because $\phi$ is a finite morphism of degree 2, we have $\phi^*(B) = 2R$, where $R$ is a curve on $Y$ which is isomorphic to $B$ \cite[Proposition 4.1.6]{lazarsfeldI}. Let $H$ be the Picard class of a hyperplane divisor on $\P^2$. Then $[B] = 4H$, and $\phi^*(H) = -K_Y$, so $[\phi^*(B)] = -4K_Y$. Therefore $[R] = -2K_Y$ and $[R] \cdot F = 4$, so each fiber contains at most 4 $\F_q$-points that lie on $R$.

         Now consider case \ref{item:notdP}. In this case, the morphism $\phi \colon Y \ra \P^2$ induced by $|-K_Y|$ is a morphism of degree 2 which contracts the curve $D$ to a point $P$, and is branched over a singular quartic curve $B$ with a singularity of multiplicity 2 at $P$. In this case, we have $\phi^*(B) = 2(D + R)$, where $R$ is a curve such that $\phi(R) = B$. We again have $[\phi^*(B)] = -4K_Y$, so $[R] = - K_Y + F$, and $[R] \cdot F = 2$, so each fiber contains at most 2 $\F_q$-points that lie on $R$.

         In either case, let $Q$ be a singular point on an anticanonical curve $G$. Then $\phi|_G \colon G \ra \P^1$ ramifies at $Q$, so $Q$ lies on the ramification divisor of $\phi$. In case \ref{item:notdP}, we have $G \cdot D = 0$, so in both cases $Q \in \supp(R)$. This concludes the proof. 
\end{proof}

Our strategy is now as follows. For each of the cases (2a), (2b) and (2c), we determine and upper bound for the number of $\F_q$-points which lie on sections with negative self-intersection. In \S\ref{sec:case(2a)} and \S\ref{sec:case(2b)}, we apply Proposition \ref{prop:(3)(4)to(1)(2)} and Lemma \ref{lem:(2)singpointsanticanonical} to transform a conic bundle satisfying case (2a) or (2b) into a conic bundle satisfying (2c) or (3). For case (2c), in \S\ref{sec:case(2c)}, we moreover determine an upper bound for the number of points that do not satisfy condition \ref{item:(2c)to(3)part(b)} in Proposition \ref{prop:(2c)to(3)}, using Lemma \ref{lem:(2c)selfintersection2classes} and the result on surfaces of degree 2 in Proposition \ref{prop:anticanonicalcurvesdP2}. We use this to construct a del Pezzo surface from a surface satisfying case (2c).  

\subsection{From case (2a) to case (2c) and (3)}\label{sec:case(2a)}
Let $f \colon X \ra \P^1 $ be a minimal conic bundle over $\F_q$ of degree 1 which satisfies case (2a) in Proposition \ref{prop:classificationCB}. Then $-K_X$ is nef, and there are 8 pairwise disjoint geometric curves $C_1, \ldots, C_8 \in \Div(\overline{X})$, with $C_i^2 = -2$ and $-K_{\overline{X}} \cdot C_i = 0$ for all $i = 1, \ldots, 8$. Using the construction in the proof of \cite[Proposition 2.6]{trepalindeg2}, we can choose the basis $\{C, F, E_1, \ldots, E_7\}$ of $\Pic(\overline{X})$ defined in Theorem \ref{thm:MT2.2.2}\ref{thm:MT2.2.2(d)} in such a way that $[C_1] = C - E_1 - E_2$. For any $i \in \{2, \ldots, 8\}$,  let us write \[[C_i] = C + aF - \sum_{n=1}^7 b_n E_n.\] Because $[C_i] \neq C$ (as $C^2 = 0$), we get $[C_i] \cdot C = a \geq 0$. Because the $E_j$ are classes of components of the singular fibers, we must have $0 \leq [C_i] \cdot E_j \leq [C_i] \cdot F = 1$ for $j \in \{1, \ldots, 7\}$. Hence $b_n \in \{0, 1\}$. From $C_1 \cdot C_i = 0$ and $C_i^2 = -2$, we obtain the conditions 
\begin{align*}
    b_1 + b_2 &= a;\\
    b_3 + b_4 + b_5 + b_6 + b_7 &= a + 2.
\end{align*}

This implies that $a \in \{0, 1, 2\}$ and we obtain the following possibilities for $C_i$:
\begin{itemize}
    \item $[C_i] = C - E_j - E_k$, for some $j, k \in \{3, \ldots, 7\}$, $j \neq k$; 
    \item $[C_i] = C + F - \sum_{n=1}^7 E_n + E_j + E_k + E_l$ for some $j \in \{1, 2\}$, $k, l \in \{3, \ldots, 7\}$, $k \neq l$;
    \item $[C_i] = C + 2F - \sum_{n=1}^7 E_n + E_j$ for some $j \in \{3, \ldots, 7\}$.
\end{itemize}

Noting that we need $C_i \cdot C_j = 0$ for all $i \neq j$, there is a unique way (up to permutation of the $E_i$ and $C_i$) of choosing 7 classes for $C_2, \ldots, C_8$ from the listed possibilities (up to permutation of the $E_i$ and $C_i$), which is the following: 
\begin{align*}
    [C_1] &= C - E_1 - E_2,\\
    [C_2] &= C - E_3 - E_4,\\
    [C_3] &= C - E_5 - E_6,\\
    [C_4] &= C + 2F - E_1 - E_2 - E_3 - E_4 - E_5 - E_6,\\
    [C_5] &= C + F - E_1 - E_3 - E_5 - E_7,\\
    [C_6] &= C + F - E_1 - E_4 - E_6 - E_7,\\
    [C_7] &= C + F - E_2 - E_3 - E_6 - E_7,\\
    [C_8] &= C + F - E_2 - E_4 - E_5 - E_7.
\end{align*}
We will use the above choice for the basis of $\Pic(\overline{X})$ and the labeling of the curves $C_i$ of self-intersection $-2$ throughout this subsection. Recall that we write $\Gamma_X = \langle \rho_X^c(\Frob_q)\rangle \subset (\Z/2\Z)^{6} \rtimes S_7$. 

\begin{lemma}\label{lem:type4type7}
    Let $f \colon X \ra \P^1 $ be a minimal standard conic bundle over $\F_q$ of degree 1 satisfying case (2a) in Proposition \ref{prop:classificationCB}. Then $f$ is of type 4 or 7.
\end{lemma}

\begin{proof}
    The sections $C_1, \ldots, C_8$ form a $\Gamma_X$-stable set. If $f$ is of type 3 or 5, $\Gamma_X$ contains an element of the form $\iota_{1\cdots \hat{j} \cdots 7}$ for some $j \in \{1, \ldots, 7\}$ (where $\hat{j}$ denotes the omission of the $j$-th entry) by Lemma \ref{lem:galoisorbitnn/2}. It should send $C_1$ to one of the $C_i$. We have
\begin{align*}
    \iota_{1\cdots \hat{j} \cdots 7}(C_1) = \begin{cases}
        C + F - \sum_{i=n}^7 E_n + E_1 + E_2 + E_j &\text{if} \ j \notin \{1, 2\},\\
        C + 2F - \sum_{n=1}^7 E_n + E_2 &\text{if} \ j = 1,\\
        C + 2F - \sum_{n=1}^7 E_n + E_1 &\text{if} \ j = 2.
    \end{cases}
\end{align*}
None of these possibilities is the class of one of the $C_i$, so $f$ cannot be of type 3 or 5.

If $f$ is of type $1, 2$ or $6$, $\Gamma_X$ contains an element of the form $\iota_{jk}$ for some distinct $j, k \in \{1, \ldots, 7\}$ by Lemma \ref{lem:galoisorbitnn/2}. We see that
\begin{align*}
    \iota_{jk}(C_1) = \begin{cases}
        C-F &\text{if} \ \{j, k\} = \{1, 2\},\\
        C - E_k - E_2 &\text{if} \ j = 1, k \neq 2,\\
        C - E_k - E_1 &\text{if} \ j = 2, k \neq 1,\\
        C + F - E_1 - E_2 - E_j - E_k &\text{if} \ j, k \notin \{1, 2\}.
    \end{cases}
\end{align*}
None of these are the classes of one of the $C_i$, so $f$ cannot be of type 1, 2 or 6.
\end{proof}

To show that we can find a rational point $P$ on $X$ satisfying the conditions of Proposition \ref{prop:(3)(4)to(1)(2)}, we explicitly determine the classes of sections with negative self-intersection on $\overline{X}$, and an upper bound for the number of rational points on them. Recall from Lemma  \ref{lem:-KXnef} that if $D \notin \{C_1, \ldots, C_8\}$ is an integral section on $\overline{X}$ with $D^2 < 0$, we have $D^2 = -1$. We have $- K_{\overline{X}} \cdot D \geq 0$ because $-K_{\overline{X}}$ is nef. From the adjunction formula, we deduce that $-K_{\overline{X}} \cdot D = - 2 p_a(D) + 1 \geq 0$, so $p_a(D) = 0$ and $-K_{\overline{X}} \cdot D = 1$.
\begin{comment}
    Explain that $p_a(D) \geq 0$ for irreducible reduced curves, with reference.
\end{comment}
Let us write $[D] = C + a F - \sum_{n=1}^7 b_n E_n$, with $a \geq 0$, and $b_n \in \{0, 1\}$ for all $n = 1, \ldots, 7$. 

From $-K_{\overline{X}} \cdot D = 1$, we obtain that 
\begin{equation}
    \sum_{n=1}^7 b_n = 2a + 1.\label{eq:(2a)equationa}
\end{equation}

From $D \cdot C_i \geq 0$ for $i = 1, \ldots, 8$ we get the following conditions:
\begin{align*}
    b_1 + b_2 &\leq a;\nonumber\\
    b_3 + b_4 &\leq a;\nonumber\\
    b_5 + b_6 &\leq a;\nonumber\\
    b_1 + b_3 + b_5 + b_7 &\leq a+1;\nonumber\\
    b_1 + b_4 + b_6 + b_7 &\leq a+1;\nonumber\\
    b_2 + b_3 + b_6 + b_7 &\leq a+1;\nonumber\\
    b_2 + b_4 + b_5 + b_7 &\leq a+1;\nonumber\\
    \sum_{i=1}^6 b_i &\leq a + 2.
\end{align*}

This system of inequalities together with \eqref{eq:(2a)equationa} has the following possible solutions for $D$:
\begin{align*}
    [D_1] &= C - E_7,\\
    [D_2] &= C + F - E_1 - E_3 - E_6,\\
    [D_3] &= C + F - E_1 - E_4 - E_5,\\
    [D_4] &= C + F - E_2 - E_3 - E_5,\\
    [D_5] &= C + F - E_2 - E_4 - E_6,\\
    [D_6] &= C + 2F - E_1 - E_2 - E_3 - E_4 - E_7,\\
    [D_7] &= C + 2F - E_1 - E_2 - E_5 - E_6 - E_7,\\
    [D_8] &= C + 2F - E_3 - E_4 - E_5 - E_6 - E_7.
\end{align*}
\begin{comment}
    \begin{itemize}
    \item Let $a = 0$. Then we must have $D = C - E_7$.
    \item Let $a = 1$. Then the possibilities are 
    \begin{align*}
        [D] &= C + F - E_1 - E_3 - E_6\\
        [D] &= C + F - E_1 - E_4 - E_5\\
        [D] &= C + F - E_2 - E_3 - E_5\\
        [D] &= C + F - E_2 - E_4 - E_6.
    \end{align*}
    \item Let $a = 2$. Then we get $b_7 = 1$ and the following possibilities:
        \begin{align*}
            [D] = C + 2F - E_1 - E_2 - E_3 - E_4 - E_7;\\
            [D] = C + 2F - E_1 - E_2 - E_5 - E_6 - E_7;\\
            [D] = C + 2F - E_3 - E_4 - E_5 - E_6 - E_7.
        \end{align*}
    \item $a = 3$ is not possible. 
\end{itemize}
\end{comment}
\begin{comment}
It follows from Riemann-Roch that 
\[\ell(D) = \frac{1}{2}D \cdot (D - K_X) + 1 + p_a(X) + s(D) - \ell(K_X - D) = 1 + s(D) \geq 1,\]
using that $D_i \cdot (D - K_X) = 0$, $p_a(X) = 0$ and $\ell(K_X - D) = 0$ (because $(K_X - D) \cdot F = -1$, so $K_X - D$ cannot be effective). Hence there exists an effective divisor $D_i$ representing each of the 8 classes listed above, which is irreducible. and it cannot be reducible 
\end{comment}
This gives us a total of 8 possible sections with self-intersection $-1$. 

\begin{lemma}\label{lem:(4)rationalpointsections}
    Let $f \colon X \ra \P^1$ be a minimal standard conic bundle over $\F_q$ of degree 1 which satisfies case (2a) in Proposition \ref{prop:classificationCB}. Then there is at most one $\F_q$-point which lies on a geometric section of $\overline{f}$ with self-intersection $-1$.
\end{lemma}

\begin{proof}
    Recall from Lemma \ref{lem:type4type7} that $f$ is of type 4 or 7. If $\ord(\rho_X^c(\Frob_q)) = n$, then by Lemma \ref{lem:galoisorbitnn/2}, we have $\rho_X^c(\Frob_q)^{n/2} = \iota_{jklm}$ for some $j, k, l, m \in \{1, \ldots, 7\}$. Note that
    \begin{align*}
        \iota_{12jk}(C_1) &= C - E_j - E_k \\
        \iota_{1jkl}(C_1) &= C + F - E_2 - E_j - E_k - E_l &\text{for} \ j, k, l \neq 2,\\
        \iota_{2jkl}(C_1) &= C + F - E_1 - E_j - E_k - E_l &\text{for} \ j, k, l \neq 1,\\
        \iota_{jklm}(C_1) &= C + 2F - E_1 - E_2 - E_j - E_k - E_l - E_m&\text{for} \ j, k, l, m \notin\{1, 2\}.\\
    \end{align*}
    Because $C_1$ has to be mapped to one of the $C_i$, we obtain that $\iota_{jklm}$ is an element of the set $S_{\iota} = \{\iota_{1234}, \iota_{1256}, \iota_{1367}, \iota_{1457}, \iota_{2357}, \iota_{2467}, \iota_{3456}\}$. We determined above that the only sections with self-intersection $-1$ are $D_1, \ldots, D_8$. For any $\iota \in S_{\iota}$, and for $i = 1, \ldots, 8$, we have $\iota(D_i) \neq D_i$. This implies that $D_1, \ldots, D_8$ constitute an orbit of size 8 under the action of $\Gamma_X$. Because $D_i \cdot D_j = 1$ for $i \neq j$, we conclude that there is at most one point that lies on all $D_i$, so there is at most one rational point on this orbit of sections with self-intersection $-1$. 
\end{proof}

\begin{comment}
    Can we show that there is no rational point?
\end{comment}

\begin{lemma}\label{lem:(4)existencerationalpoint}
    Let $f \colon X \ra \P^1$ be a minimal standard conic bundle over $\F_q$ of degree 1, satisfying case (2a) in Proposition \ref{prop:classificationCB}. Unless $q = 2$ and $f$ is of type 4, there is a rational point $P$ on a smooth fiber of $f$ such that $P$ is not a singular point of an anticanonical curve, and $P$ does not lie on any section with negative self-intersection. 
\end{lemma}

\begin{proof}
    Recall that the only possible sections with negative self-intersections are $C_1, \ldots, C_8$ with $C_i^2 = -2$ and $D_1, \ldots, D_8$ with $D_i^2 = -1$. The sections $C_i$ form a Galois orbit of disjoint curves, and hence they do not contain any rational points. Using Lemma \ref{lem:(4)rationalpointsections}, we conclude that there is at most one rational point on a section with negative self-intersection. If it exists we denote it by $P_1$. We know (Lemma \ref{lem:(2)singpointsanticanonical}) that every smooth fiber of $f$ contains at most two singular points that lie on anticanonical curves. By Lemma \ref{lem:type4type7}, $f$ is of type 4 or 7. Unless $q = 2$ and $f$ is of type 4, the conic bundle will have at least one smooth fiber above a rational point of $\P^1$. If $q \geq 3$, this smooth fiber contains at least four rational points, and therefore there is a point $P \in X(\F_q)$ on the fiber which is not equal to $P_1$, and is also not a singular point on an anticanonical curve. If $q = 2$ and $f$ is of type 7, the conic bundle has three smooth fibers, so we can choose a smooth fiber which does not contain $P_1$. This fiber will then contain a rational point which is not a singular point on an anticanonical curve.
\end{proof}

\begin{theorem}\label{thm:(2a)to(2c,3)}
    Let $f \colon X \ra \P^1$ be a minimal standard conic bundle over $\F_q$ of degree 1, satisfying case (2a) in Proposition \ref{prop:classificationCB}. Then $f$ is of type 4 or 7, and unless $q = 2$ and $X$ is of type 4, there is an elementary transformation $\varphi \colon X \dashrightarrow Y$ such that $f \circ \varphi^{-1} \colon Y \ra \P^1$ is a minimal standard conic bundle of degree 1 of the same type as $f$, satisfying case (2c) or (3) in Proposition \ref{prop:classificationCB}.
\end{theorem}

\begin{proof}
    This follows from Lemma \ref{lem:type4type7}, Lemma \ref{lem:(4)existencerationalpoint} and Proposition \ref{prop:(3)(4)to(1)(2)}.
\end{proof}

\subsection{From case (2b) to case (2c) and (3)}\label{sec:case(2b)}
Let $f \colon X \ra \P^1$ be a minimal standard conic bundle over $\F_q$ of degree 1 which satisfies case (2b) in Proposition \ref{prop:classificationCB}. Then $-K_X$ is nef, and there is a Galois orbit consisting of four integral geometric sections $C_1, C_2, C_3, C_4$ of $f$ such that $C_i^2 = -2$,  $C_i \cdot C_j = 1$ for $i \neq j$ such that $i \equiv j \mod 2$, and $C_i \cdot C_j = 0$ for $i \not\equiv j \mod 2$. Using similar reasoning as in the previous section, without loss of generality we can choose the basis $\{C, F, E_1, \ldots, E_7\}$ of $\Pic(\overline{X})$ such that the classes of the curves $C_i$ in $\Pic(\overline{X})$ are
\begin{align*}
    [C_1] &= C - E_1 - E_2,\\
    [C_2] &= C - E_3 - E_4,\\
    [C_3] &= C + 2F - E_2 - E_3 - E_4 - E_5 - E_6 - E_7,\\
    [C_4] &= C + F - E_1 - E_5 - E_6 - E_7.
\end{align*}

We will use the above choice for the basis of $\Pic(\overline{X})$ and the labeling of the curves $C_i$ of self-intersection $-2$ throughout this subsection. 

\begin{lemma}\label{lem:(3)type3type5}
    Let $f \colon X \ra \P^1$ be a minimal standard conic bundle over $\F_q$ of degree 1 satisfying case (2b) in Proposition \ref{prop:classificationCB}. Then $f$ is of type 3 or 5.
\end{lemma}

\begin{proof}
    Recall that $C_1, \ldots, C_4$ form a Galois orbit. If $f$ is of type 1, 2 or 6, then by Lemma \ref{lem:galoisorbitnn/2}, $\Gamma_X$ contains an element of the form $\iota_{jk}$ for some $j, k \in \{1, \ldots, 7\}$. We have 
    \begin{align*}
    \iota_{jk}([C_1]) = \begin{cases}
        C - F &\text{if} \ \{j, k\} = \{1, 2\},\\
        C - E_2 - E_k &\text{if} \ j = 1, k \neq 2,\\
        C - E_1 - E_k &\text{if} \ j = 2, k \neq 1,\\
        C + F - E_1 - E_2 - E_j - E_k &\text{if} \ j, k \notin \{1, 2\}.
    \end{cases}
\end{align*}
In any of these cases, the image is not a class of one of the $C_1, \ldots, C_4$, therefore these are not possible.

Now suppose $f$ is of type 4 or 7. Write $\tau = \rho_X^c(\Frob_q)$, such that $\Gamma_X = \langle \tau \rangle$, with $n = \ord(\Gamma_X)$. If $f$ is of type 4 or 7, then $n$ is equal to 8 or 24 respectively (Table \ref{tab:orders}), and  $\tau^{n/2} = \iota_{jklm}$ for some distinct $j, k, l, m \in \{1, \ldots, 7\}$ by Lemma \ref{lem:galoisorbitnn/2}. Because $C_1, \ldots, C_4$ form a cyclic Galois orbit of size 4, and $4$ divides $n/2$ in both cases, we must have that $\tau^{n/2}(C_i) = C_i$ for $i = 1, \ldots, 4$. We see that 
\begin{align*}
        \iota_{12jk}(C_1) &= C - E_j - E_k\\
        \iota_{1jkl}(C_1) &= C + F - E_2 - E_j - E_k - E_l &\text{for} \ j, k, l \neq 2,\\
        \iota_{2jkl}(C_1) &= C + F - E_1 - E_j - E_k - E_l &\text{for} \ j, k, l \neq 1,\\
        \iota_{jklm}(C_1) &= C + 2F - E_1 - E_2 - E_j - E_k - E_l - E_m&\text{for} \ j, k, l, m \notin\{1, 2\},
    \end{align*}
so we cannot find $j,k,l, m$ such that $\iota_{jklm}(C_1) = C_1$. Therefore, it is not possible that $f$ is of type 4 or 7.
\end{proof}

Similarly to the previous case, we want to find an elementary transformation of $f$ such that the resulting surface is a minimal standard conic bundle which satisfies case (2c) or (3), using Proposition \ref{prop:(3)(4)to(1)(2)}. We start again by describing the classes of all sections of negative self-intersection. Let $D$ be an integral curve on $\overline{X}$ with $D^2 < 0$ and $D \cdot F = 1$, which is not one of the curves $C_1, \ldots, C_4$. Then $D^2 = -1$ by Lemma \ref{lem:-KXnef} and $-K_{\overline{X}} \cdot D = 1$ by the adjunction formula.

Let us write $[D] = C + a F - \sum_{n=1}^7 b_n E_n$, with $a \geq 0$ and $b_n \in \{0, 1\}$ for all $n = 1, \ldots, 7$. Because $D \cdot C_i \geq 0$ for $i = 1, 2, 3, 4$, we get the following conditions:
\begin{align}
    b_1 + b_2 &\leq a;\nonumber\\
    b_3 + b_4 &\leq a;\nonumber\\
    b_1 + b_5 + b_6 + b_7 &\leq a+1;\nonumber\\
    b_2 + b_3 + b_4 + b_5 + b_6 + b_7 &\leq a+2.\label{eq:(2b)conditions}
\end{align}

As in the previous case, we have $\sum_{n=1}^7 b_n = 2a + 1$, and using \eqref{eq:(2b)conditions} we obtain the following solutions for $D$:
\begin{align}
    [D] &= C - E_j &\text{for} \ j \in \{5, 6, 7\};\nonumber\\
    [D] &= C + F - E_j - E_k - E_l &\text{for} \ j \in \{1, 2\}, \ k \in \{3, 4\}, \ l \in \{5, 6, 7\};\nonumber\\
    [D] &= C + F -  E_{j} - E_{k} - E_l &\text{for} \ j \in \{2, 3, 4\},\ k, l \in \{5, 6, 7\};\nonumber\\
    [D] &= C + 2F - E_1 - E_{i_1} - \cdots - E_{i_4} &\text{for} \ 2 \leq i_j \leq 7 \ \text{such that} \ \{5, 6, 7\} \not\subseteq  \{i_1, \ldots, i_4\}.\label{eq:(3)-1sections}
\end{align}

This gives a total of 36 possible $(-1)$-sections.
\begin{comment}
    $3 + 2 \cdot 2 \cdot 3 + 3 \cdot {3 \choose 2} + {6 \choose 4} - 3 = 36$
\end{comment}

\begin{lemma}\label{lem:(2b)rationalpoints-1curves}
    Let $D$ be a section of $\overline{f}$ on $\overline{X}$ with self-intersection $-1$. Then any $\F_q$-point on $D$ is a singular point on a curve in the linear system $|-K_{\overline{X}} + E|$ for $E \in \{E_2, F - E_2\}$.
\end{lemma}

\begin{proof}
    By Lemma \ref{lem:(3)type3type5}, $f$ is of type 3 or 5. Let $\tau = \rho_X^c(\Frob_q)$, and let $n = \ord(\tau)$. Then by Lemma \ref{lem:galoisorbitnn/2}, $\tau^{n/2}$ is of the form $\iota_{1\cdots \hat{j} \cdots 7}$ for some $j \in \{1, \ldots, 7\}$. We note that
    \begin{align*}
        \iota_{1\cdots \hat{j} \cdots 7}(C_1) = \begin{cases}
            C + F - \sum_{n=1}^7 E_n + E_1 + E_2 + E_j &\text{if} \ j \notin \{1, 2\}\\
            C + 2F - \sum_{n=1}^7 E_n + E_2 &\text{if} \ j = 1\\
            C + 2F - \sum_{n=1}^7 E_n + E_1 &\text{if} \ j = 2.
        \end{cases}
    \end{align*}
    Because $\iota_{1\cdots \hat{j} \cdots 7}(C_1) \in \{C_1, C_2, C_3, C_4\}$, we conclude that $j = 2$. So $\Gamma_X$ contains the element $\iota_{134567}$. 
    
    Consider the action of $\iota_{134567}$ on the sections of self-intersection $-1$, whose classes are listed in \eqref{eq:(3)-1sections}. For all sections $D$ with self-intersection $-1$, we have $D \neq \iota_{134567}(D)$, and $[D + \iota_{134567}(D)] \in \{-K_{\overline{X}} + E_2, -K_{\overline{X}} + (F - E_2)\}$. Any $\F_q$-point on $D$ lies in the intersection of $D$ and $\iota_{134567}(D)$, and is therefore a singular point on $D + \iota_{134567}(D)$.
\end{proof}

\begin{remark}
    Because $\tau^{n/2} = \iota_{134567}$, it follows that the singular fiber containing $E_2$ lies above an $\F_q$-point.\end{remark}

 The following result bounds the number of such singular points on each fiber.

\begin{lemma}\label{lem:(2b)singpoints-KX+E}
    Let $f \colon X \ra \P^1$ be a minimal standard conic bundle over $\F_q$ of degree 1 satisfying case (2b) in Proposition \ref{prop:classificationCB}. Consider the base change $\overline{f} \colon \overline{X} \ra \P^1$ over $\overline{\F}_q$. Let $E \in \{E_2, F - E_2\}$ on $\overline{X}$, and let 
    \begin{align*}
        B_E = \begin{cases}
            C_1 + C_3 \qquad \text{if} \ E = F - E_2\\
            C_2 + C_4 \qquad \text{if} \ E = E_2.
        \end{cases} 
    \end{align*}
    Then each fiber of $\overline{f}$ contains at most two points that are singular points on a curve in the linear system $|-K_{\overline{X}} + E|$ which does not contain the curves $E$ and $B_E$ as a component.
\end{lemma}

\begin{proof}
    The curve $E$ is a $(-1)$-curve defined over $\F_{q^2}$, so we can contract this curve on $X_{2} \colonequals X \times_{\Spec(\F_q)} \Spec(\F_{q^2})$ via a morphism $\pi \colon X_{2} \ra Y$. The surface $Y$ is a standard conic bundle over $\F_{q^2}$ of degree 2 via the morphism $f \circ \pi^{-1} \colon Y \ra \P^1$. Recall that $F$ denotes the class of a fiber of $f$ on $X$. The image $\pi_*(F)$ is the class of a fiber of $f \circ \pi^{-1}$. For any curve $G$ on $Y$, we see that $-K_Y \cdot [G] = -K_{X_2} \cdot [\pi^*(G)] \geq 0$, so $-K_Y$ is nef.
    
    To see that we can apply Proposition \ref{prop:anticanonicalcurvesdP2}, note that $[B_E] = -K_{X_2} - F + E$, and $[\pi_*(B_E)] = -K_{Y} - \pi_*(F)$. Because $B_E$ is integral over $\F_{q^2}$, so is $\pi_*(B_E)$. Furthermore, if $W \in |-K_Y - \pi_*(F)|$ is any integral curve, then $[\pi^*(W)] = -K_{X_2} - F + E$, and $\pi^*(W)$ is integral (it cannot have $E$ as a component, because $-K_{X_2}$ being nef implies that $|-K_{X_2} - F| =\emptyset$). But $B_E$ is the only curve on $X_2$ in the linear system $|-K_{X_2} - F + E|$ (Lemma \ref{lem:-KXnef}), so $\pi^*(W) = B_E$. This implies that $Y$ satisfies the assumptions of Proposition \ref{prop:anticanonicalcurvesdP2}\ref{item:notdP}. 
    
    Let $G \in |-K_{\overline{X}} + E|$ in $\Pic(X_2)$ which does not contain $E$ or $B_E$ as a component. Then $[G] \cdot E = 0$, so $\pi_*(G) \cong G$. Hence, the number of singular points on $G$ and $\pi_*(G)$ is the same. We have $[\pi_*(G)] = -K_Y$, and $\pi_*(G)$ does not have $\pi_*(B_E)$ as a component, so by Proposition \ref{prop:anticanonicalcurvesdP2}\ref{item:notdP}, the singular points on $\pi_*(G)$ lie on on a divisor $R$, where $R$ is a bisection of the conic bundle $f \circ \pi^{-1} \colon Y \ra \P^1$. The singular points on $G$ thus lie on the strict transform $\pi^{-1}(R)$, which satisfies $[\pi^{-1}(R)] \cdot F = 2$. We conclude that there can be at most two such points on each fiber.
\end{proof}

\begin{lemma}\label{lem:(2b)atmost2ptsonnegativesection}
    Let $f \colon X \ra \P^1$ be a minimal standard conic bundle over $\F_q$ of degree 1  satisfying case (2b) in Proposition \ref{prop:classificationCB}. Then each fiber of $f$ contains at most two rational points that lie on a section with self-intersection $-1$.
\end{lemma}

\begin{proof}
    By Lemma \ref{lem:(2b)rationalpoints-1curves}, such a rational point is a singular point on a curve of the form $D + \iota_{134567}(D)$ in the linear system $|-K_{\overline{X}} + E|$ for $E \in \{E_2, F - E_2\}$. The divisor $D + \iota_{134567}(D)$ splits into two components which are irreducible sections of $\overline{f}$ with self-intersection $-1$, so it does not contain $E$ or $B_E$ as a component. Consider a fiber $F_1$ of $f$ above an $\F_q$-rational point. Suppose $P$ is an $\F_q$-point on $F_1$ which is a singular point on a curve $G$ in the linear system $|-K_{\overline{X}} + E_2|$, which does not contain $E_2$ or $B_{E_2}$ as a component. Then it is also a singular point on $\tau(G) \in |-K_{\overline{X}} + (F - E_2)|$ (which does not contain $F - E_2$ or $B_{F- E_2}$ as components), and vice versa. By Lemma \ref{lem:(2b)singpoints-KX+E}, there can be at most two such $\F_q$-points on $F_1$. In particular, there are at most two $\F_q$-points on $F_1$ that are $\F_q$-rational points on sections of self-intersection $-1$.
\end{proof}

\begin{lemma}\label{lem:(3)existencepointonfiber}
    Let $q \geq 4$, and let $f \colon X \ra \P^1$ be a minimal standard conic bundle over $\F_q$ of degree 1 satisfying case (2b) in Proposition \ref{prop:classificationCB}. Then there is an $\F_q$-rational point $P$ on a smooth fiber of $f$ such that $P$ is not a singular point of an anticanonical curve and does not lie on any section with negative self-intersection.
\end{lemma}

\begin{proof}
    Note that for $q \geq 4$, any smooth fiber above an $\F_q$-point contains at least 5 $\F_q$-points. We show that at most four of them do not satisfy the statement. Since $f$ is of type 3 or 5 by Lemma \ref{lem:(3)type3type5}, there is only one singular fiber above a rational point. Therefore, for any $q$, there exists a rational point $Q \in \P^1(\F_q)$ such that the fiber $f^{-1}(Q)$ on $X$ above $Q$ is smooth, and does not contain the base point of the linear system $|-K_X|$. By Lemma \ref{lem:(2b)atmost2ptsonnegativesection}, $f^{-1}(Q)$ contains at most two rational points that lie on sections with self-intersection $-1$. The sections $C_1, \ldots, C_4$ with self-intersection $-2$ do not contain any rational points. By Lemma \ref{lem:(2)singpointsanticanonical}, there are at most two points on $f^{-1}(Q)$ which are singular points on an anticanonical curve. This proves the result.
\end{proof}

\begin{theorem}\label{thm:(2b)to(2c,3)}
    Let $q \geq 4$, and let $f \colon X \ra \P^1$ be a minimal standard conic bundle over $\F_q$ of degree 1, satisfying case (2b) in Proposition \ref{prop:classificationCB}. Then $f$ is of type $3$ or $5$, and there is a birational map $\varphi \colon X \dashrightarrow Y$ such that $f \circ \varphi^{-1} \colon Y \ra \P^1$ is a minimal standard conic bundle of degree 1 of the same type as $f$, satisfying case (2c) or (3) in Proposition \ref{prop:classificationCB}.
\end{theorem}

\begin{proof}
    By Lemma \ref{lem:(3)type3type5}, $f$ is of type 3 or type 5. The result follows from Lemma \ref{lem:(3)existencepointonfiber} and Proposition \ref{prop:(3)(4)to(1)(2)}.
\end{proof}

\subsection{From case (2c) to case (3)}\label{sec:case(2c)}
Let $f \colon X \ra \P^1$ be a minimal standard conic bundle over $\F_q$ of degree 1 satisfying case (2c) in Proposition \ref{prop:classificationCB}. Then $-K_X$ is nef, and there are two disjoint irreducible curves $C_1$, $C_2$ on $\overline{X}$ which are bisections of the conic bundle $\overline{f}$ and satisfy $C_i^2 = -2$ and $-K_{\overline{X}} \cdot [C_i] = 0$. Since $C_1$ and $C_2$ are bisections, their classes in $\Pic(\overline{X})$ must be of the form \[[C_i] = 2C + a F - \sum_{n=1}^7 b_n E_n\]
for some $a \geq 0$ and $b_n \in \{0, 1, 2\}$.
From the conditions $C_i^2 = -2$ and $-K_{\overline{X}} \cdot [C_i] = 0$ we obtain \[\sum_{n=1}^7 b_{n}^2 = 4a + 2 \quad \text{and} \quad \sum_{n=1}^7 b_{n} = 2a + 4.\]
The only classes satisfying these two equations are of the form $2C + F - \sum_{n=1}^7 E_n + E_k$ and $2C + 2F - \sum_{n=1}^7 E_n - E_k$ for some $k \in \{1, \ldots 7\}$. By requiring $C_1 \cdot C_2 = 0$, without loss of generality, we can set
\begin{align*}
    [C_1] &= 2C + F - E_1 - E_2 - E_3 - E_4 - E_5 - E_6\\
    [C_2] &= 2C + 2F -E_1 - E_2 - E_3 - E_4 - E_5 - E_6 - 2E_7.
\end{align*}

We will use the above choice for the basis of $\Pic(\overline{X})$ and the labeling of the curves $C_i$ of self-intersection $-2$ throughout this subsection.

Because $C_1$ and $C_2$ form a Galois orbit of size 2, it follows that $\{E_7, F - E_7\}$ is a Galois orbit of $(-1)$-curves, and hence the fiber containing $E_7$ lies above an $\F_q$-rational point of $\P^1$.

\begin{lemma}\label{lem:(2c)not67}
    Let $f \colon X \ra \P^1$ be a minimal standard conic bundle over $\F_q$ of degree 1 satisfying case (2c) in Proposition \ref{prop:classificationCB}. Then $f$ is not of type 6 or 7.
\end{lemma}

\begin{proof}
    In Table \ref{tab:pointsbelowsingularfibers}, we see that standard conic bundles of type 6 and 7 have no singular fiber above a rational point of $\P^1$. Because the fiber containing $E_7$ lies above a rational point, $f$ cannot be of type 6 or 7.
\end{proof}

\begin{comment}
\begin{proof}
    If $\Gamma_X = \langle \tau \rangle$ is of type 6, we have $\tau = \iota \circ (ij)(klmnr)$ for some $\iota \in (\Z/2\Z)^6$, with $\{i, j, k, l, m, n, r\} = \{1, \ldots, 7\}$. Then $\tau^2 = \iota' \circ (knrlm)$ with $\iota' \in (\Z/2\Z)^6$. Because $C_1$ and $C_2$ form a Galois orbit they are fixed by $\tau^2$, and we conclude that $7 \notin \{k, l, m, n, r\}$. From Lemma \ref{lem:galoisorbitnn/2}, we know that $\tau^10 = \iota_{ij}$, which also fixes $C_1$ and $C_2$. This implies $7 \notin \{i, j\}$, which is not possible. Hence, $X$ cannot be of type 6.

    If $\Gamma = \langle \tau \rangle$ is of type 7, it has order 24, and we have $\tau = \iota \circ (mnr)(ijkl)$ for some $\iota \in (\Z/2\Z)^6$. Then $\tau^4 = \iota' \circ (mnr)$ for some $\iota' \in (\Z/2\Z)^6$. Since $C_1$ and $C_2$ are fixed by $\tau^4$, we conclude that $7 \notin \{m, n, r\}$. From Lemma \ref{lem:galoisorbitnn/2}, we know that $\tau^{12} = \iota_{ijkl}$, which has to fix $C_1$ and $C_2$, implying $7 \notin \{i, j, k, l\}$. We conclude that $X$ cannot be of type 7.
\end{proof}
\end{comment}

In what follows, we determine the values of $q$ for which we can prove the existence of a point which satisfies the three conditions in Proposition \ref{prop:(2c)to(3)}. Recall that Lemma \ref{lem:(2)singpointsanticanonical} bounds the number of points on a fiber which are singular points on an anticanonical curve. To satisfy condition \ref{item:(2c)to(3)part(b)} of Proposition \ref{prop:(2c)to(3)}, according to Lemma \ref{lem:(2c)selfintersection2classes}, we have to consider singular curves in the linear systems $|-K_{\overline{X}} + E_s|$ and $|-K_{\overline{X}} + (F- E_s)|$ for all values of $s$ such that the fiber containing $E_s$ lies above an $\F_q$-rational point. Note that $s = 7$ always satisfies this. In the case $s = 7$, the following result bounds the number of singular points on such curves per fiber. 

\begin{lemma}\label{lem:(2c)singpoints-KX+E7}
    Let $f \colon X \ra \P^1$ be a minimal standard conic bundle over $\F_q$ of degree 1 satisfying case (2c) in Proposition \ref{prop:classificationCB}. Consider the base change $\overline{f} \colon \overline{X} \ra \P^1$ over $\overline{\F}_q$. Let $E \in \{E_7, F - E_7\}$ on $\overline{X}$, and let
    \begin{align*}
        C_E = \begin{cases}
            C_1 \qquad \text{if} \ E = E_7\\
            C_2 \qquad \text{if} \ E = F - E_7.
        \end{cases} 
    \end{align*}
    Then each fiber of $\overline{f}$ contains at most two points that are singular points on a curve in the linear system $|-K_{\overline{X}} + E|$ which does not contain the curves $E$ and $C_E$ as a component.
\end{lemma}

\begin{proof}
    The proof is completely analogous to the proof of Lemma \ref{lem:(2b)singpoints-KX+E}.
\end{proof}

In the proof of Lemma \ref{lem:(2c)singpoints-KX+E7}, we contract the $(-1)$-curve $E$ to obtain a weak del Pezzo surface of degree 2, so that we can apply Proposition \ref{prop:anticanonicalcurvesdP2}\ref{item:notdP}. Depending on its type, the conic bundle $X$ can have more singular fibers above $\F_q$-points besides the fiber containing $E_7$. If we contract a $(-1)$-curve in these other singular fibers, the resulting conic bundle of degree 2 has an ample anticanonical divisor (as we will see in Lemma \ref{lem:(2c)singpoints-KX+Ei}), so we can apply Proposition \ref{prop:anticanonicalcurvesdP2}\ref{item:dP} instead.

We use the following characterization of del Pezzo surfaces, which follows directly from the geometric constructions of (weak) del Pezzo surfaces in Theorem \ref{thm:delpezzoclassification} and Theorem \ref{thm:weakdelpezzoclassification}.

\begin{lemma} \label{lem:weakdphas-2}
    Let $X$ be a weak del Pezzo surface over a field $k$. If $\overline{X}$ contains no integral curve of self-intersection $-2$, then $X$ is a del Pezzo surface.
\end{lemma}

\begin{comment}
\begin{proof}
    We know that $K_X^2 > 0$, because $-K_X$ is big. Because $-K_X$ is not ample, there is an irreducible curve $C$ on $\overline{X}$ such that $-K_{\overline{X}} \cdot C = 0$. Note that $C$ is not a numerically trivial divisor manoy{why? clear if conic bundle}. Then it follows from the Hodge index theorem \cite[Corollary 2.4]{badescu} that $C^2 < 0$. From the adjunction formula, we see that $C^2 = 2 p_a(C) - 2$, so the only possibility is that $C^2 = -2$. 
\end{proof}
\end{comment}


\begin{lemma}\label{lem:(2c)singpoints-KX+Ei}
    Let $f \colon X \ra \P^1$ be a minimal standard conic bundle over $\F_q$ of degree 1 satisfying case (2c) in Proposition \ref{prop:classificationCB}. Let $s \in \{1, \ldots, 6\}$ be such that $\{E_s, F - E_s\}$ is a Galois-stable set on $\overline{X}$, and let $E \in \{E_s, F - E_s\}$. Then each fiber of $\overline{f} \colon \overline{X} \ra \P^1$ contains at most four points that are singular points on a  curve in the linear system $|-K_{\overline{X}} + E|$ which does not contain $E$ as a component.
\end{lemma}

\begin{proof}
    As in the proof of Lemma \ref{lem:(2b)singpoints-KX+E}, we can contract the curve $E$ over $\F_{q^2}$ via a morphism $\pi \colon X_2 \ra Y$. The surface $Y$ is a standard conic bundle over $\F_{q^2}$ of degree 2 via the morphism $f \circ \pi^{-1} \colon Y \ra \P^1$. If $F$ is the class of a fiber of $f$, then the image $\pi_*(F)$ is the class of a fiber of $f \circ \pi^{-1}$. Note that $-K_Y \cdot G = -K_{X_2} \cdot [\pi^*(G)] \geq 0$ for any curve $G$ on $Y$, so $-K_Y$ is nef. By Lemma \ref{lem:-KXnef}, $C_1$ and $C_2$ are the only integral curves on $\overline{X}$ of self-intersection $-2$. Note that $[C_1] \cdot E = [C_2] \cdot E = 1$, so $\pi_*(C_1)^2= \pi_*(C_2)^2 = -1$. This implies that there are no curves of self-intersection $-2$ on $\overline{Y}$, so we can apply Lemma \ref{lem:weakdphas-2} to conclude that $Y$ is a del Pezzo surface of degree 2. 
    
    Let $G \in |-K_{\overline{X}} + E|$ in $\Pic(X_{2})$ such that $E$ is not a component of $G$. Then $[G] \cdot E = 0$, so $\pi_*(G) \cong G$. Hence, the number of singular points on $G$ and $\pi_*(G)$ is the same. We have $[\pi_*(G)] = -K_Y$. So by Proposition \ref{prop:anticanonicalcurvesdP2}\ref{item:dP}, the singular points on $\pi_*(G)$ lie on on a divisor $R$ which satisfies $R \cdot \pi_*(F) = 4$. The singular points of $G$ then lie on the strict transform $\pi^{-1}(R)$, which satisfies $[\pi^{-1}(R)] \cdot F = 4$. We conclude that there can be at most four such points on each fiber of $\overline{f}$.
\end{proof}

\begin{lemma}\label{lem:(2c)singpointsintersection2}
    Let $f \colon X \ra \P^1$ be a minimal standard conic bundle over $\F_q$ satisfying case (2c) in Proposition \ref{prop:classificationCB}. Let $m$ be the number of singular fibers which are defined over $\F_q$. Then each fiber of $f$ contains at most $4m-2$ $\F_q$-points which are singular points on two integral Galois-conjugate bisections on $\overline{X}$ of self-intersection 2 and arithmetic genus 1.
\end{lemma}

\begin{proof}
    By Lemma \ref{lem:(2c)selfintersection2classes}, two Galois-conjugate bisections $D_1$, $D_2$ as in the statement have classes $-K_{\overline{X}} + E_s$ and $-K_{\overline{X}} + (F - E_s)$ in $\Pic(\overline{X})$, where $s$ is one of $m$ distinct values, one of which is $s = 7$. If $P$ is an $\F_q$-point which is a singular point on $D_1$, then it is also a singular point on $D_2$ and vice versa. We apply Lemma \ref{lem:(2c)singpoints-KX+E7} and Lemma \ref{lem:(2c)singpoints-KX+Ei} to conclude the proof.
\end{proof}

Finally, to find points satisfying condition \ref{item:(2c)to(3)part(c)} in Proposition \ref{prop:(2c)to(3)}, we determine the number of $\F_q$-points on sections of $f$ with negative self-intersection. Recall from Lemma \ref{lem:-KXnef} that $C_1, C_2$ are the only curves of self-intersection $-2$ on $\overline{X}$. Let us, as in the cases (2a) and (2b), determine the Picard classes of sections $D$ with self-intersection $-1$ on $\overline{X}$. Recall that they satisfy $-K_{\overline{X}} \cdot [D] = 1$. As before, we write $[D] = C + aF - \sum_{n=1}^7 b_n E_n$ with $a \geq 0$ and $b_n \in \{0, 1\}$. Since $D \cdot C_i \geq 0$ for $i = 1, 2$, we obtain the inequalities
\begin{align}
    \sum_{n=1}^6 b_n &\leq 2a + 1 \nonumber\\
    \sum_{n=1}^6 b_n  + 2 b_7 &\leq 2a + 2. \label{eq:(2c)conditions}
\end{align}
Recall that $-K_{\overline{X}} \cdot [D] = 1$, implying the equality $\sum_{n = 1}^7 b_n = 2a+1$. In particular, $a \leq 3$. The inequalities \eqref{eq:(2c)conditions} do not give any extra conditions. This yields:
\begin{align}
    [D] &= C - E_i &\text{for} \ i \in \{ 1, \ldots, 7\},\nonumber\\
    [D] &= C + F - E_i - E_j - E_k &\text{for} \ i, j, k \in \{1, \ldots, 7\},\nonumber\\
    [D] &= C + 2F - \sum_{n=1}^7 E_n + E_i + E_j  &\text{for} \ i, j \in \{1, \ldots, 7\},\nonumber\\
    [D] &= C + 3F - \sum_{n=1}^7 E_n.\label{eq:(2c)sections-1}
\end{align}
\begin{comment}
    By Riemann-Roch, we obtain
\[\ell(D) = \frac{1}{2}D \cdot (D - K_X) + 1 + p_a(X) + s(D) - \ell(K_X - D) = 1 + s(D) 
using that $D \cdot (D - K_X) = 0$, $p_a(X) = 0$ and $\ell(K_X - D) = 0$ (because $(K_X - D) \cdot F = -1$ so it cannot be effective). Hence there is an effective divisor linearly equivalent to $D$, which is irreducible. So each of these classes actually represents a $(-1)$-curve. 
\end{comment}
This gives a total of 64 possible $(-1)$-curves. The possible number of rational points on these curves can be bounded by considering their partitioning into Galois orbits, which depends on the type of $f$. We treat the types separately.

\begin{comment}
     $7 + {7 \choose 3} + {7 \choose 2} + 1 = 64$
\end{comment}

\begin{lemma}\label{lem:type124ratlptssections}
    Let $f \colon X \ra \P^1$ be a minimal standard conic bundle over $\F_q$ satisfying case (2c) in Proposition \ref{prop:classificationCB}. Assume moreover that $f$ is of type 1, 2 or 4. Then there are no $\F_q$-points on sections of $\overline{f} \colon \overline{X} \ra \P^1$ with negative self-intersection.
\end{lemma}

\begin{proof}
    Recall from Lemma \ref{lem:galoisorbitnn/2} that if $f$ is of type 1 or 2, then $\Gamma_X$ contains an element of the form $\iota_{ij}$ for some $i, j \in \{1, \ldots, 7\}$. For any of the sections $D$ in \eqref{eq:(2c)sections-1}, and any $i, j$, we have $D \neq \iota_{ij}(D)$ and $D \cdot \iota_{ij}(D) = 0$. In particular, there cannot be any rational points on these sections when $f$ is of type 1 or 2.

    If $f$ is of type 4, then $\Gamma_X = \langle \tau \rangle$, with $\tau = \rho_X^c(\Frob_q)$, has order 8 (Table \ref{tab:orders}), and contains an element of the form $\tau^4 = \iota_{ijkl}$ (see Lemma \ref{lem:galoisorbitnn/2}). Since this is an even power of $\tau$, it acts as the identity on $C_1$ and $C_2$. This implies that $7 \notin \{i, j, k, l\}$, and therefore, using the description of the Galois action in Notation \ref{not:W(D7)} and the description of the Galois orbits in the singular fibers in \S\ref{sec:Galoisorbitsfibers}, we find that $\tau$ must be of the form $\tau \in \{\iota_{imn7}(ijkl), \iota_{ijkmn7}(ijkl)\}$ where $\{i, j, k, l, m, n\} = \{1, 2, 3, 4, 5, 6\}$. Then $\tau^2$ is of the form $\iota_{ij}(ik)(jl)$ for some $i, j, k, l \in \{1, \ldots, 6\}$. For all sections $D$ with self-intersection $-1$ in \eqref{eq:(2c)sections-1}, we have $\tau^4(D) = \iota_{ijkl}(D) \neq D$, so we conclude that the $(-1)$-sections must form eight Galois orbits of size 8. Furthermore, $\tau^2(D) \cdot D = 0$ for each $D$, so the $(-1)$-sections do not contain any rational points.
\end{proof}

\begin{lemma}\label{lem:type35ratlptssections}
     Let $f \colon X \ra \P^1$ be a minimal standard conic bundle over $\F_q$ satisfying case (2c) in Proposition \ref{prop:classificationCB}. Assume moreover that $f$ is of type 3 or 5. Then all $\F_q$-points on sections of $\overline{f} \colon \overline{X} \ra \P^1$ with negative self-intersection are singular points on curves in the linear system $|-K_{\overline{X}} + E|$ for $E \in \{E_7, F - E_7\}$, which do not contain $E, C_1$ or $C_2$ as a component.
\end{lemma}

\begin{proof}
    Since $f$ is of type 3 or 5, $\Gamma_X = \langle \tau \rangle$ with $\tau = \rho_X^c(\Frob_q)$ has order $n$ equal to 4 or 12, respectively (Table \ref{tab:orders}). By Lemma \ref{lem:galoisorbitnn/2}, $\tau^{n/2} = \iota_{1 \cdots \hat{j} \cdots 7}$ for some $j \in \{1, \ldots, 7\}$. Because $\{C_1, C_2\}$ is a Galois orbit of size 2, and $\tau^{n/2}$ is an even power of $\tau$, it must fix $C_1$ and $C_2$. The only possibility is when $j = 7$, so we have $\tau^{n/2} = \iota_{123456}$. Applying $\iota_{123456}$ to the classes of $(-1)$-sections $D$ in \eqref{eq:(2c)sections-1}, yields $[D + \iota_{123456}(D)] \in \{-K_{\overline{X}} + E_7, -K_{\overline{X}} + (F-E_7)\}$. Note that any rational point on $D$ also lies on $\iota_{123456}(D)$, and hence is a singular point on $D + \iota_{123456}(D)$. This is an effective divisor in $\Div(\overline{X})$ consisting of two irreducible sections of self-intersection $-1$, neither of which is equal to one of the curves $E_7, F-E_7, C_1$ or $C_2$.
\end{proof}

\begin{theorem}\label{thm:(2c)to(3)}
    Let $f \colon X \ra \P^1$ be a minimal standard conic bundle over $\F_q$ of degree 1, satisfying case (2c) in Proposition \ref{prop:classificationCB}. Let $m$ be the number of singular fibers of $f$ which are defined over $\F_q$. Suppose $q \geq 4m$. Then there is a birational map $\varphi \colon X \dashrightarrow Y$ such that $f \circ \varphi^{-1} \colon Y \ra \P^1$ is a minimal standard conic bundle, of the same type as $f$, satisfying case (3) in Proposition \ref{prop:classificationCB}, that is, $Y$ is a del Pezzo surface of degree 1.
\end{theorem}

\begin{proof}
    By Lemma \ref{lem:(2c)not67}, $f$ is of type 1, 2, 3, 4 or 5. Because $q \geq m$, we know that the number of $\F_q$-rational points of $\P^1$ is larger than $m$, so there is a smooth fiber $F_1$ which is defined over $\F_q$. By Lemma \ref{lem:(2)singpointsanticanonical}, $F_1$ contains at most two singular points of anticanonical curves. 
    
    Suppose $f$ is of type 1, 2, or 4. Then there are no $\F_q$-points on sections with negative self-intersection by Lemma \ref{lem:type124ratlptssections}. By Lemma \ref{lem:(2c)singpointsintersection2}, there are at most $4m-2$ $\F_q$-points on $F_1$ which are singular points on two irreducible Galois-conjugate bisections of self-intersection 2 and arithmetic genus 1. The hypothesis on $q$ implies that there exists an $\F_q$-point $P$ on $F_1$ which satisfies the conditions in Proposition \ref{prop:(2c)to(3)}, so $\varphi_{\scriptscriptstyle P}$ is the desired elementary transformation.

    Now suppose $f$ is of type $3$ or $5$. Then the fiber containing $E_7$ is the only singular fiber defined over $\F_q$, so $m = 1$. Combining Lemmas \ref{lem:type35ratlptssections}, \ref{lem:(2c)selfintersection2classes} and \ref{lem:(2c)singpoints-KX+E7}, we conclude that there are at most 4 rational points on $F_1$ that do not satisfy the three conditions in Proposition \ref{prop:(2c)to(3)}. Therefore, since $q \geq 4$, $F_1$ contains a rational point $P$ satisfying the conditions in Proposition \ref{prop:(2c)to(3)}, and $\varphi_{\scriptscriptstyle P}$ is the desired elementary transformation.
\end{proof}

Now we are in the position to prove most of the statement of Theorem \ref{thm:mainthm}. We obtain the following result.

\begin{theorem}\label{thm:mainthmpart1}
Let $t \in \{1, \ldots, 7\}$. There exists a del Pezzo surface of degree 1 over $\F_q$ of type $t$ if $q \geq k_t$, where $k_1 = 23$, $k_2 = 9$, $k_3 = k_5 = 4$, $k_4 = 13$ and $k_6 = k_7 = 2$.

\begin{comment}
    \begin{enumerate}
        \item There exists a del Pezzo surface of degree 1 over $\F_q$ of type 1 if $q \geq 23$.
        \item There exists a del Pezzo surface of degree 1 over $\F_q$ of type 2 if $q \geq 9$.
        \item There exists a del Pezzo surface of degree 1 over $\F_q$ of type 3 if $q \geq 4$. It does not exist when $q = 2$.
        \item There exists a del Pezzo surface of degree 1 over $\F_q$ of type 4 if $q \geq 13$.
        \item There exists a del Pezzo surface of degree 1 over $\F_q$ of type 5 if $q \geq 4$.
        \item There exists a del Pezzo surface of degree 1 over $\F_q$ of type 6 for all $q$.
        \item There exists a del Pezzo surface of degree 1 over $\F_q$ of type 7 for all $q$.
    \end{enumerate}
\end{comment}
\end{theorem}

\begin{proof}
    By Lemma \ref{lem:smallfields}, there exists a minimal standard conic bundle over $\F_q$ of degree 1 for each of the types, except for type 1 when $q = 2, 3$ and type 3 when $q = 2$. Each such conic bundle has to satisfy one of the cases in Proposition \ref{prop:classificationCB}. Then Theorems \ref{thm:(1)to(2,3)}, \ref{thm:(2a)to(2c,3)}, \ref{thm:(2b)to(2c,3)} and \ref{thm:(2c)to(3)} imply that if $q \geq k_i$, we can find a sequence of elementary transformations which turns the existing conic bundle of type $t$ into a del Pezzo surface of degree 1 of type $t$.
\end{proof}

\section{Improvement of results using Bertini twist} \label{sec:method2}

The techniques for proving Theorem \ref{thm:mainthmpart1} do not give a conclusive outcome for the existence of del Pezzo surfaces of types $1, \ldots, 5$ for some small values of $q$. In some cases, we can use the existence of a quadratic twist to show the existence of surfaces for certain $q$. 
Let $X$ be a del Pezzo surface of degree 1 over a finite field $\F_q$. The linear system $|-2K_X|$ induces a finite morphism $\varphi \colon X \ra \P^3$ of degree 2. This map induces an involution $\beta$ of $X$, called the \term{Bertini involution}. We can twist $X$ by the nontrivial element in $H^1(\F_q, \Z/2\Z) \cong \Z/2\Z$. The resulting surface $X_\alpha$ is a quadratic twist of $X$, called its \term{Bertini twist} \cite[\S5.1.1]{BFL}.

The type of a del Pezzo surface uniquely determines the type of its Bertini twist. Let $\beta^* \colon \Pic(\overline{X}) \ra \Pic(\overline{X})$ be the action induced by $\beta$. It acts on $K_{\overline{X}}^\perp$ as multiplication by $-1$. Let us denote by $\tau$ and $\tau_\alpha$ the automorphism of $\Pic(\overline{X})$ induced by the action of the Frobenius automorphism on $X$ and $X_\alpha$, respectively. Then $\tau_\alpha = \beta^* \circ \tau$ \cite[Proposition 4.4]{rybakovtrepalin}. It follows that the eigenvalues corresponding to the Frobenius action on $X_\alpha$ are precisely the eigenvalues of the Frobenius action on $X$ multiplied by $-1$ (except for the eigenvalue $1$ corresponding to $K_{\overline{X}}$). Using Magma we computed the eigenvalues of the Frobenius action for every type \cite{github_Existence_dP1}, and this way we determined the Bertini twist corresponding to each type. There are a few sets of eigenvalues that occur more than once. In these cases we use that the size of the conjugacy class corresponding to a type and its Bertini twist should be the same. The correspondence is tabulated in Table \ref{tab:bertinitwists} in the \nameref{sec:appendix}. For the minimal types with conic bundles, we obtain the following pairs:
\begin{table}[h]
    \centering
    \begin{tabular}{l|l|l}
         Type $X$ & Bertini twist $X_\alpha$ & Index of $X_\alpha$\\
         \hline
         1. & 97. & 8 \\
         2. & 105.  & 4\\
         3. & 3. & 0\\
         4. & 73. & 3\\
         5. & 5. & 0\\
\end{tabular}
    \caption{}
    \label{tab:Bertinitwistsconicbundles}
\end{table}

\begin{comment}
\end{comment}

The goal of this section is twofold. Firstly, motivated by \cite[\S5.2]{BFL}, we study the existence of types of del Pezzo surfaces which contain a $(-1)$-curve defined over $\F_q$, which includes in particular the the Bertini twists of the minimal surfaces of type 1, 2 and 4. This results in the proof of Theorem \ref{thm:mainthm}. Some of the considered types have index 8, that is, they can be blown down to $\P^2$ over $\F_q$. Our second goal is to solve the inverse Galois problem for these types of surfaces. This results in Theorem \ref{thm:index8+twist}.

\subsection{Existence of del Pezzo surfaces with a $(-1)$-curve over $\F_q$}

For types 1, 2 and 4, the Bertini twist is a del Pezzo surface of a different type which is not minimal (Table \ref{tab:Bertinitwistsconicbundles}). We show the existence of del Pezzo surfaces of degree 1 of type 97, 105 and 73 for some values of $q$ for which existence was not determined in Theorem \ref{thm:mainthmpart1}, using techniques introduced in \cite[\S5.2]{BFL}. We use the fact that surfaces of these types contain a $(-1)$-curve which is defined over the field of definition $\F_q$ of the surface. This curve can be contracted, and the resulting surface is a del Pezzo surface of degree 2 over $\F_q$. We consider the following reverse result:

\begin{proposition}[{\cite[Corollary 14]{STVA}}]\label{prop:STVcor14}
    Let $X$ be a del Pezzo surface of degree 2 over a field $k$ with a rational point $P \in X(k)$. The blow-up of $X$ at $P$ is a del Pezzo surface of degree 1 if and only if the point $P$ does not lie on any $(-1)$-curve, nor on the ramification divisor $R$ of the morphism induced by the anticanonical linear system.
\end{proposition}

If such a point $P$ exists on a del Pezzo surface $X$ of degree 2 over $\F_q$, the type of $X$ uniquely determines the type of the del Pezzo surface of degree 1 obtained after blowing up in $P$. The types of del Pezzo surfaces of degree 2 are listed in \cite[Table 2]{trepalin} and \cite[Table 1]{urabe2}. We compute for each type (as numbered in \cite{trepalin}) the corresponding type of degree 1 del Pezzo surface obtained by blowing up (as numbered in \cite[Table 2]{urabe2}) using Magma (see \cite{github_Existence_dP1}). The results are listed in column 2 of Table \ref{tab:dp2data} (\nameref{sec:appendix}). Note that the types 95, 101 and 107 of degree 1 surfaces appear twice in this column. This means that depending on which $(-1)$-curve on a del Pezzo surface of these types we contract, we may obtain two different types of del Pezzo surfaces of degree 2. All other types that appear, appear only once, and therefore the type after contracting a $(-1)$-curve over $\F_q$ is unique.

\begin{comment}
\begin{table}[h]\label{tab:typeblowupdp2}
\caption{types of dP1 corresponding to blow-up of dP2}
\begin{minipage}{0.4\textwidth}
\begin{tabular}{l|l}
    type degree 2 & type degree 1\\
    \hline
        1 & 91 \\
        2 & 92 \\
        3 & 93 \\
        4 & 94 \\
        5 & 95 \\
        6 & 95 \\
        7 & 96 \\
        8 & 97 \\
        9 & 98 \\
        10 & 66 \\
        11 & 99 \\
        12 & 100 \\
        13 & 101 \\
        14 & 101 \\
        15 & 102 \\
        16 & 67 \\
        17 & 68 \\
        18 & 69 \\
        19 & 87 \\
        20 & 103 \\
        21 & 104 \\
        22 & 70 \\
        23 & 105 \\
        24 & 106 \\
        25 & 107 \\
        26 & 107 \\
        27 & 71 \\
        28 & 72 \\
        29 & 73 \\
        30 & 74
\end{tabular}
\end{minipage}
\begin{minipage}{0.4\textwidth}
\begin{tabular}{l|l}
    type degree 2 & type degree 1\\
    \hline
    31 & 38 \\
        32 & 56 \\
        33 & 76 \\
        34 & 88 \\
        35 & 39 \\
        36 & 109 \\
        37 & 110 \\
        38 & 57 \\
        39 & 111 \\
        40 & 40 \\
        41 & 79 \\
        42 & 80 \\
        43 & 41 \\
        44 & 42 \\
        45 & 43 \\
        46 & 58 \\
        47 & 59 \\
        48 & 60 \\
        49 & 44 \\
        50 & 45 \\
        51 & 46 \\
        52 & 47 \\
        53 & 48 \\
        54 & 49 \\
        55 & 50 \\
        56 & 51 \\
        57 & 52 \\
        58 & 53 \\
        59 & 54 \\
        60 & 55
\end{tabular}
\end{minipage}
\end{table}
\end{comment}

The existence or non-existence of a del Pezzo surface of degree 2 over $\F_q$ of each of the 60 types is shown in \cite[Theorem 1.2]{trepalin} and \cite[Theorem 1.2]{LoughranTrepalin}. Because we know that the 57 types of del Pezzo surfaces of degree 1 with a $(-1)$-curve defined over $\F_q$ are obtained as the blow-up of a del Pezzo surface, the non-existence of a del Pezzo surface of degree 1 of a specific type over $\F_q$ follows from the non-existence of del Pezzo surfaces of degree 2 of types in the corresponding rows in Table \ref{tab:dp2data}. Using this reasoning, we obtain the second column in Table \ref{tab:existencedP1withline} (\nameref{sec:appendix}). 

Following \cite{BFL}, we call the union of the $(-1)$-curves and the ramification divisor on a del Pezzo surface of degree 2 the \term{bad locus}. The existence of a del Pezzo surface of degree 2 over $\F_q$ of a given type with a rational point outside the bad locus implies the existence of a del Pezzo surface of degree 1 of the corresponding type listed in column 2 of Table \ref{tab:dp2data} (Proposition \ref{prop:STVcor14}). Recall that the total number of rational points on a del Pezzo surface over $\F_q$ is determined uniquely by its trace, given by the formula in Theorem \ref{thm:numberrationalpoints}. This number is therefore uniquely determined by the type of a del Pezzo surface of degree 2, and the values are listed in column 3 of Table \ref{tab:dp2data}.

For each type $t$, we find an upper bound (in terms of $q$) for the number of rational points in the bad locus of a del Pezzo surface of degree 2 over $\F_q$ of type $t$, using an argument as in \cite[15]{BFL}. It is shown there that the number of rational points on the ramification divisor $R$ of the anticanonical morphism is bounded from above by
\begin{align}
    F(q) = \begin{cases}
    q + 1 + 6 \sqrt{q} & \text{if $q$ is odd;}\\
    2(q + 1) &\text{if $q$ is even.}
    \end{cases}
\end{align}
This follows from the fact that $R$ is an irreducible smooth curve of genus 3 when $q$ is odd (using the Hasse-Weil bound), and $R$ is a union of at most two irreducible curves of genus 0 if $q$ is even \cite[Lemma 4.1]{BFL}. 

To bound the number of rational points on the $(-1)$-curves, we explicitly study the intersection patterns within the orbits of $(-1)$-curves under the Frobenius action, as computed with Magma and listed in the ancillary file (and in \cite{github_Existence_dP1}). A $(-1)$-curve which is defined over $\F_q$ contains $q + 1$ rational points. A Galois orbit of two or more $(-1)$-curves contains rational points only if all lines in the orbit intersect at these points. Therefore, if we have an intersection pattern of the form $[m \mid a_1, \ldots, a_{m-1}]$ (Notation \ref{not:orbits}), the number of rational points on these lines is at most $\min\{a_1, \ldots, a_{m-1}\}$. These facts together allow us to give an upper bound for the number of $\F_q$-points that lie on the $(-1)$-curves on a surface of type $t$, which we denote by $L_t(q)$. The values are listed in column 4 of Table \ref{tab:dp2data}. 

Let $Y$ be a del Pezzo surface of degree 2 over $\F_q$ of type $t$. We compute \cite{github_Existence_dP1} a lower bound for $q$ such that we have $\#Y(\F_q) \geq L_t(q) + F(q) + 1$. In this case there exists an $\F_q$-point on $Y$ outside of the bad locus, and we can blow up $Y$ in a rational point, demonstrating existence of a del Pezzo surface of degree 1 of the corresponding type listed in column 2 of Table \ref{tab:dp2data}. This way we obtain the existence for all but finitely many $q$. The results are listed in column 3 of Table \ref{tab:existencedP1withline}. 

\begin{comment}
    manoy{Additional text for in thesis.

If $q$ is odd, then $\#Y(F_q) - L_t(q) + F(q) - 1$ is an expression of the form $q^2 + a q + 6 \sqrt{q} + c$, for some $a, c \in \Z$. Note that its second derivative is positive for $q \geq 2$, so the expression has at most two roots (and also at least one, because $c$ is negative). Therefore, we determine the largest root, and take the first odd prime power above this value. This is a lower bound for the existence of a del Pezzo surface of degree 2 of the corresponding type. A similar reasoning works in case $q$ is even, in which case the expression is a quadratic equation in $q$. 
}
\end{comment}

\subsection{Del Pezzo surfaces of degree 1 of index 8}

If $X$ has index 8, the surface over $\F_q$ is obtained as a blow-up of a collection of closed points in $\P^2$, the sum of the degrees of which is $8$, which lie in general position (Theorem \ref{thm:delpezzoclassification}). Each type is uniquely determined by the degree of the closed points that are blown up. The characterization can be derived from the Galois orbits of the $(-1)$-curves on $X$, and is given in Table \ref{tab:blowupindex8} in the \nameref{sec:appendix}. We can computationally demonstrate existence or non-existence of each of these types over $\F_q$ for small values of $q$ (i.e. the values not mentioned in Table \ref{tab:existencedP1withline} for those types), by showing closed points of the appropriate degrees in general position exist in $\P^2_{\F_q}$, using code presented in \cite{BFL}, where a number of cases are already treated (one for each trace). The Magma code can be found in \cite{github_Existence_dP1}. Together with the result in Table \ref{tab:existencedP1withline}, we can obtain a complete classification of the existence of types of del Pezzo surfaces of degree 1 of index 8 which contain a $(-1)$-curve defined over $\F_q$. The result is given in Table \ref{tab:dp1index8withline} (\nameref{sec:appendix}). It is moreover shown in \cite[18]{BFL} that type 112 exists for all values of $q$. Together with the results in Table \ref{tab:dp1index8withline}, this solves the inverse Galois problem for del Pezzo surfaces of degree 1 and index 8 over finite fields for all types except for type 108, which corresponds to a blow-up of $\P^2$ in two closed points of degree 4 in general position.

\begin{theorem}\label{thm:index8}
\begin{enumerate}
    \item Del Pezzo surfaces of degree 1 of type 91 exist over $\F_q$ if and only if $q = 16$ or $q \geq 19$.
    \item Del Pezzo surfaces of degree 1 of type 92 exist over $\F_q$ if and only if $q \geq 11$.
    \item Del Pezzo surfaces of degree 1 of type 93, 94 and 95 exist over $\F_q$ if and only if $q \geq 7$.
    \item Del Pezzo surfaces of degree 1 of type 97 exist over $\F_q$ if and only if $q \geq 5$.
    \item Del Pezzo surfaces of degree 1 of type 96, 98 and 99 exist over $\F_q$ if and only if $q \geq 4$.
    \item Del Pezzo surfaces of degree 1 of type 100, 101, 102, 103, 104, 105 and 107 exist over $\F_q$ if and only if $q \geq 3$.
    \item Del Pezzo surfaces of degree 1 of type 106, 109, 110, 111, 112 exist over $\F_q$ for all $q$.
\end{enumerate}
\end{theorem}

\begin{proof}
    Part of this result is shown in \cite{BFL}, and the rest is completed by Magma computations \cite{github_Existence_dP1}. In column 3 of Table \ref{tab:dp1index8withline} the references to \cite{BFL} per type are given, where applicable.
\end{proof}

Now we have all the results necessary to obtain the statements of Theorem \ref{thm:mainthm} and Theorem \ref{thm:index8+twist}.

\begin{proof}[Proof of Theorem \ref{thm:mainthm}]
    Most of the result is proven in Theorem \ref{thm:mainthmpart1}. For types 1, 2 and 4, we use that their Bertini twists are of types 97, 105 and 73, respectively. The result is completed using Magma as described above (see Magma code in \cite{github_Existence_dP1}), see Table \ref{tab:dp1index8withline} for types 97 and 105, and Table \ref{tab:existencedP1withline} for type 73.
\end{proof}

\begin{proof}[Proof of Theorem \ref{thm:index8+twist}]
    We know that a del Pezzo surface of degree 1 of a given type exists if and only if its Bertini twist exists. The result follows from Theorem \ref{thm:index8}, matching each type with its Bertini twist. The pairing can be found in Table \ref{tab:bertinitwists}.
\end{proof}

\printbibliography[heading=bibintoc]

@book{hartshorne,
  title={Algebraic geometry},
  author={Hartshorne, R.},
   SERIES = {Graduate Texts in Mathematics},
  volume={52},
  YEAR = {1977},
  publisher={Springer-Verlag}
}

@book{manin,
    AUTHOR = {Manin, Yu. I.},
     TITLE = {Cubic forms},
    SERIES = {North-Holland Mathematical Library},
    VOLUME = {4},
   EDITION = {Second edition},
   %   NOTE = {Algebra, geometry, arithmetic, Translated from the Russian by M. Hazewinkel},
 PUBLISHER = {North-Holland},
      YEAR = {1986},
}

@article {manintsfasman,
    AUTHOR = {Manin, Yu. I. and Tsfasman, M. A.},
     TITLE = {Rational varieties: algebra, geometry, arithmetic},
   %JOURNAL = {Russ. Math. Surv.},
   journal = {Russian Mathematical Surveys},
    VOLUME = {41},
      YEAR = {1986},
    NUMBER = {2},
     PAGES = {51-116}
}

@article{iskovskih,
    AUTHOR = {Iskovskih, V. A.},
     TITLE = {Minimal models of rational surfaces over arbitrary fields},
   %JOURNAL = {Math. USSR Izv.},
  JOURNAL = {Mathematics of the USSR-Izvestiya},
    VOLUME = {14},
      YEAR = {1980},
    NUMBER = {1},
     PAGES = {17--39}
}

@book {beauville,
    AUTHOR = {Beauville, A.},
     TITLE = {Complex algebraic surfaces},
    SERIES = {London Mathematical Society Student Texts},
    VOLUME = {34},
%   EDITION = {Second},
%      NOTE = {Translated from the 1978 French original by R. Barlow, with assistance from N. I. Shepherd-Barron and M. Reid},
 PUBLISHER = {Cambridge University Press},
      YEAR = {1996}
}

@book {poonen,
    AUTHOR = {Poonen, B.},
     TITLE = {Rational points on varieties},
    SERIES = {Graduate Studies in Mathematics},
    VOLUME = {186},
 PUBLISHER = {American Mathematical Society, Providence, RI},
      YEAR = {2017}
}

@article {STVA,
    AUTHOR = {Salgado, C. and Testa, D. and V\'{a}rilly-Alvarado, A.},
     TITLE = {On the unirationality of del {P}ezzo surfaces of degree 2},
   %JOURNAL = {J. Lond. Math. Soc. (2)},
  JOURNAL = {Journal of the London Mathematical Society. Second Series},
    VOLUME = {90},
      YEAR = {2014},
    NUMBER = {1},
     PAGES = {121--139}
}

@article{kollarmella,
  title={Quadratic families of elliptic curves and unirationality of degree 1 conic bundles},
  author={Koll{\'a}r, J. and Mella, M.},
  journal={American Journal of Mathematics},
  volume={139},
  number={4},
  pages={915--936},
  year={2017},
  publisher={Johns Hopkins University Press}
}

@inproceedings{BFL,
  title={Del Pezzo surfaces over finite fields and their Frobenius traces},
  author={Banwait, B. and Fit{\'e}, F. and Loughran, D.},
  booktitle={Mathematical Proceedings of the Cambridge Philosophical Society},
  volume={167},
  number={1},
  pages={35--60},
  year={2019},
  organization={Cambridge University Press}
}

@article {urabe1,
    AUTHOR = {Urabe, T.},
     TITLE = {Calculation of {M}anin's invariant for {D}el {P}ezzo surfaces},
   %JOURNAL = {Math. Comp.},
  JOURNAL = {Mathematics of Computation},
    VOLUME = {65},
      YEAR = {1996},
    NUMBER = {213},
     PAGES = {247--258, S15--S23}
}

@article{urabe2,
  title={Supplement to Calculation of Manin's Invariant for Del Pezzo Surfaces},
  author={Urabe, T.},
  journal={Mathematics of Computation},
  volume={65},
  number={213},
  pages={S15--S23},
  year={1996}
}

@article {rybakov,
    AUTHOR = {Rybakov, S.},
     TITLE = {Zeta functions of conic bundles and {D}el {P}ezzo surfaces of
              degree 4 over finite fields},
   %JOURNAL = {Mosc. Math. J.},
  JOURNAL = {Moscow Mathematical Journal},
    VOLUME = {5},
      YEAR = {2005},
    NUMBER = {4},
     PAGES = {919--926, 974}
}

@article {trepalindeg2,
    AUTHOR = {Trepalin, A.},
     TITLE = {Minimal del {P}ezzo surfaces of degree 2 over finite fields},
  % JOURNAL = {Bull. Korean Math. Soc.},
  JOURNAL = {Bulletin of the Korean Mathematical Society},
    VOLUME = {54},
      YEAR = {2017},
    NUMBER = {5},
     PAGES = {1779--1801}
}

@article {trepalin,
    AUTHOR = {Trepalin, A.},
     TITLE = {Del {P}ezzo surfaces over finite fields},
   %JOURNAL = {Finite Fields Appl.},
  JOURNAL = {Finite Fields and their Applications},
    VOLUME = {68},
      YEAR = {2020},
     PAGES = {101741, 32}
}

@article {KunyavskiiTsfasman,
    AUTHOR = {Kunyavski\u i, B. \`E. and Tsfasman, M. A.},
     TITLE = {Zero-cycles on rational surfaces and {N}\'eron-{S}everi tori},
 %  JOURNAL = {Izv. Akad. Nauk SSSR Ser. Mat.},
  JOURNAL = {Izvestiya Akademii Nauk SSSR. Seriya Matematicheskaya},
    VOLUME = {48},
      YEAR = {1984},
    NUMBER = {3},
     PAGES = {631--654}
}

@article{magma,
  title={The Magma algebra system I: The user language},
  author={Bosma, W. and Cannon, J. and Playoust, C.},
%  journal={J. Symb. Comput.},
  journal={Journal of Symbolic Computation},
  volume={24},
  number={3-4},
  pages={235--265},
  year={1997}
}

@article {LoughranTrepalin,
    AUTHOR = {Loughran, D. and Trepalin, A.},
     TITLE = {Inverse {G}alois problem for del {P}ezzo surfaces over finite
              fields},
   %JOURNAL = {Math. Res. Lett.},
  JOURNAL = {Mathematical Research Letters},
    VOLUME = {27},
      YEAR = {2020},
    NUMBER = {3},
     PAGES = {845--853}
}

@book{SerreTopics,
  title={Topics in Galois theory},
  author={Serre, J.-P.},
  year={2016},
  publisher={CRC Press}
}

@inproceedings{SD,
  title={Cubic surfaces over finite fields},
  author={Swinnerton--Dyer, P.},
  booktitle={Mathematical Proceedings of the Cambridge Philosophical Society},
  volume={149},
  number={3},
  pages={385--388},
  year={2010}
}

@article{rybakovtrepalin,
  title={Minimal cubic surfaces over finite fields},
  author={Rybakov, S. and Trepalin, A.},
  journal={Sbornik: Mathematics},
  year={2017},
  volume={208},
  number={9},
  pages={1399--1419}
}

@article{luroth,
  title={Beweis eines Satzes {\"u}ber rationale Curven},
  author={L{\"u}roth, J.},
  journal={Mathematische Annalen},
  volume={9},
  number={2},
  pages={163--165},
  year={1875}
}

@article{FLS,
  title={Rational points of bounded height on general conic bundle surfaces},
  author={Frei, C. and Loughran, D. and Sofos, E.},
  journal={Proceedings of the London Mathematical Society},
  volume={117},
  number={2},
  pages={407--440},
  year={2018}
}

@article{KST,
  title={Del Pezzo surfaces of degree 4},
  author={Kunyavski\u i, B. {\`E}. and Skorobogatov, A. N.  and Tsfasman, M. A.},
  journal={M{\'e}moires de la Soci{\'e}t{\'e} Math{\'e}matique de France},
  number={37}
}

@book{BLT,
  title={Geometry and Arithmetic of Surfaces},
  author={Bright, M. and van Luijk, R. and Testa, D.},
  year={2018}
}

@book{liu,
  title={Algebraic geometry and arithmetic curves},
  author={Liu, Q.},
  volume={6},
  year={2002},
  publisher={Oxford University Press}
}

@article{kawamata,
  title={A generalization of Kodaira-Ramanujam's vanishing theorem},
  author={Kawamata, Y.},
  journal={Mathematische Annalen},
  volume={261},
  number={1},
  pages={43--46},
  year={1982},
  publisher={Springer}
}

@article{viehweg,
  title={Vanishing theorems},
  author={Viehweg, E.},
  journal={Journal f\"{u}r die reine und angewandte Mathematik},
  year={1982},
  pages={1--8},
  volume={335}
}

@incollection{demazure,
  title={Surfaces de del pezzo - III. Positions presque g\'{e}n\'{e}rales},
  author={Demazure, M.},
  booktitle={S{\'e}minaire sur les Singularit{\'e}s des Surfaces: Centre de Math{\'e}matiques de l’Ecole Polytechnique, Palaiseau 1976--1977},
  pages={36--49},
  year={2006},
  publisher={Springer}
}

@book{badescu,
  title={Algebraic surfaces},
  author={B{\u{a}}descu, L.},
  year={2001},
  publisher={Springer}
}

@book{lazarsfeldI,
  title={Positivity in Algebraic Geometry I: Classical setting: Line Bundles and Linear Series},
  author={Lazarsfeld, R.},
  volume={48},
  year={2017},
  publisher={Springer}
}

@article{sarkisov,
  title={On conic bundle structures},
  author={Sarkisov, V. G.},
  journal={Mathematics of the USSR-Izvestiya},
  volume={20},
  number={2},
  pages={355},
  year={1983},
  publisher={IOP Publishing}
}

@software{github_Existence_dP1,
author = {Trip, M. T.},
title = {{Existence\_dP1}},
year = {2026},
URL = {https://github.com/ManoyTrip/Existence_dP1}
}

@misc{ICP,
  title={Inverse curve problems on del Pezzo surfaces},
  author={Kaya, E. and McKean, S. and Streeter, S. and Uppal, H.},
  eprint  = {2511.08688v1},
  archivePrefix = {arXiv},
  primaryClass = {math.AG},
  year={2025}
}

@book{dolgachev,
  title={Classical algebraic geometry: a modern view},
  author={Dolgachev, I. V.},
  year={2012},
  publisher={Cambridge University Press}
}

\section*{Appendix}\label{sec:appendix}
\begin{table}[ht]
    \centering
    \caption{Types of del Pezzo surfaces of degree 1 with the type of the associated Bertini twist.}
    \begin{minipage}{0.32\textwidth}
        \begin{tabular}{|c|c|}
        \hline
         Type & Bertini twist\\
         \hline
              1 & 97 \\
2 & 105 \\
3 & 3 \\
4 & 73 \\
5 & 5 \\
6 & 84 \\
7 & 82 \\
8 & 91 \\
9 & 37 \\
10 & 68 \\
11 & 35 \\
12 & 60 \\
13 & 42 \\
14 & 33 \\
15 & 94 \\
16 & 100 \\
17 & 17 \\
18 & 78 \\
19 & 102 \\
20 & 111 \\
21 & 21 \\
22 & 109 \\
23 & 23 \\
24 & 36 \\
25 & 25 \\
26 & 59 \\
27 & 58 \\
28 & 56 \\
29 & 34 \\
30 & 30 \\
31 & 31 \\
32 & 32 \\
33 & 14 \\
34 & 29 \\
35 & 11 \\
36 & 24 \\
37 & 9 \\
38 & 93\\
\hline
        \end{tabular}
    \end{minipage}
    \begin{minipage}{0.32\textwidth}
        \begin{tabular}{|c|c|}
        \hline
Type & Bertini twist\\
\hline
        39 & 39 \\
40 & 99 \\
41 & 89 \\
42 & 13 \\
43 & 85 \\
44 & 92 \\
45 & 72 \\
46 & 65 \\
47 & 47 \\
48 & 96 \\
49 & 106 \\
50 & 103 \\
51 & 64 \\
52 & 90 \\
53 & 63 \\
54 & 86 \\
55 & 61 \\
56 & 28 \\
57 & 57 \\
58 & 27 \\
59 & 26 \\
60 & 12 \\
61 & 55 \\
62 & 107 \\
63 & 53 \\
64 & 51 \\
65 & 46 \\
66 & 66 \\
67 & 75 \\
68 & 10 \\
69 & 95 \\
70 & 70 \\
71 & 87 \\
72 & 45 \\
73 & 4 \\
74 & 81 \\
75 & 67 \\
76 & 101\\
\hline
        \end{tabular}
    \end{minipage}
    \begin{minipage}{0.32\textwidth}
        \begin{tabular}{|c|c|}
        \hline
        Type & Bertini twist\\
        \hline
        77 & 77 \\
78 & 18 \\
79 & 79 \\
80 & 88 \\
81 & 74 \\
82 & 7 \\
83 & 83 \\
84 & 6 \\
85 & 43 \\
86 & 54 \\
87 & 71 \\
88 & 80 \\
89 & 41 \\
90 & 52 \\
91 & 8 \\
92 & 44 \\
93 & 38 \\
94 & 15 \\
95 & 69 \\
96 & 48 \\
97 & 1 \\
98 & 98 \\
99 & 40 \\
100 & 16 \\
101 & 76 \\
102 & 19 \\
103 & 50 \\
104 & 104 \\
105 & 2 \\
106 & 49 \\
107 & 62 \\
108 & 108 \\
109 & 22 \\
110 & 110 \\
111 & 20 \\
112 & 112 \\
 & \\
 &\\
 \hline
        \end{tabular}
    \end{minipage}
    \label{tab:bertinitwists}
\end{table}

\begin{tabular}{c|c}

\end{tabular}

\begin{longtable}[h]{|c|c|l|l|}
    \caption{For each type $t$ of del Pezzo surface of degree 2 (numbering from \cite[Table 2]{trepalin}), we list the type of del Pezzo surface of degree 1 obtained after blow-up, the number of $\F_q$-points on a surface of type $t$, and an upper bound on the number of $\F_q$-points that can lie on $(-1)$-curves on the surface.}\label{tab:dp2data}\\
    
    \hline
    Type of dP2 $Y$ & Type of blow-up dP1 & $\#Y(\F_q)$ & $L_t(q)$ \\
    \hline
    1 & 91 & $q^2 + 8q + 1$ & $56q + 56$ \\
2 & 92 & $q^2 + 6q + 1$ & $32q + 32$ \\
3 & 93 & $q^2 + 4q + 1$ & $16q + 20$ \\
4 & 94 & $q^2 + 5q + 1$ & $20q + 20$ \\
5 & 95 & $q^2 + 2q + 1$ & $8$ \\
6 & 95 & $q^2 + 2q + 1$ & $8q + 20$ \\
7 & 96 & $q^2 + 3q + 1$ & $8q + 8$ \\
8 & 97 & $q^2 + 4q + 1$ & $12q + 14$ \\
9 & 98 & $q^2 + 1$ & $20$ \\
10 & 66 & $q^2 + 1$ & $8q + 32$ \\
11 & 99 & $q^2 + q + 1$ & $4q + 8$ \\
12 & 100 & $q^2 + 2q + 1$ & $2q + 8$ \\
13 & 101 & $q^2 + 2q + 1$ & $6$ \\
14 & 101 & $q^2 + 2q + 1$ & $4q + 6$ \\
15 & 102 & $q^2 + 3q + 1$ & $6q + 6$ \\
16 & 67 & $q^2 + 3q + 1$ & $8q + 14$ \\
17 & 68 & $q^2 + 4q + 1$ & $8q + 8$ \\
18 & 69 & $q^2 - 2q + 1$ & $32$ \\
19 & 87 & $q^2 - q + 1$ & $8$ \\
20 & 103 & $q^2 + 1$ & $2q + 8$ \\
21 & 104 & $q^2 + 1$ & $10$ \\
22 & 70 & $q^2 + 1$ & $4q + 10$ \\
23 & 105 & $q^2 + q + 1$ & $2$ \\
24 & 106 & $q^2 + q + 1$ & $2q + 2$ \\
25 & 107 & $q^2 + 2q + 1$ & $2$ \\
26 & 107 & $q^2 + 2q + 1$ & $2q + 4$ \\
27 & 71 & $q^2 + q + 1$ & $8$ \\
28 & 72 & $q^2 + 2q + 1$ & $4$ \\
29 & 73 & $q^2 + 2q + 1$ & $4q + 6$ \\
30 & 74 & $q^2 + 3q + 1$ & $4q + 4$ \\
31 & 38 & $q^2 - 4q + 1$ & $44$ \\
32 & 56 & $q^2 - q + 1$ & $2q + 20$ \\
33 & 76 & $q^2 - 2q + 1$ & $14$ \\
34 & 88 & $q^2 - q + 1$ & $6$ \\
35 & 39 & $q^2 + 1$ & $12$ \\
36 & 109 & $q^2 + 1$ & $0$ \\
37 & 110 & $q^2 + 1$ & $2$ \\
38 & 57 & $q^2 + 1$ & $2q + 4$ \\
39 & 111 & $q^2 + q + 1$ & $0$ \\
40 & 40 & $q^2 - q + 1$ & $14$ \\
41 & 79 & $q^2 + 1$ & $4$ \\
42 & 80 & $q^2 + q + 1$ & $2$ \\
43 & 41 & $q^2 + q + 1$ & $4$ \\
44 & 42 & $q^2 + 2q + 1$ & $0$ \\
45 & 43 & $q^2 + 2q + 1$ & $2$ \\
46 & 58 & $q^2 + q + 1$ & $2q + 4$ \\
47 & 59 & $q^2 + 2q + 1$ & $2q + 2$ \\
48 & 60 & $q^2 + 3q + 1$ & $2q + 4$ \\
49 & 44 & $q^2 - 6q + 1$ & $56$ \\
50 & 45 & $q^2 - 2q + 1$ & $20$ \\
51 & 46 & $q^2 - q + 1$ & $2$ \\
52 & 47 & $q^2 + 1$ & $2$ \\
53 & 48 & $q^2 - 3q + 1$ & $20$ \\
54 & 49 & $q^2 - q + 1$ & $6$ \\
55 & 50 & $q^2 + 1$ & $2$ \\
56 & 51 & $q^2 + 1$ & $2$ \\
57 & 52 & $q^2 + q + 1$ & $0$ \\
58 & 53 & $q^2 + q + 1$ & $2$ \\
59 & 54 & $q^2 + 2q + 1$ & $0$ \\
60 & 55 & $q^2 + 3q + 1$ & $2$\\
\hline
\end{longtable}

\begin{table}[ht]
    \centering
        \caption{For each type of del Pezzo surface of degree 1 which has a $(-1)$-curve defined over $\F_q$, we list values of $q$ for which non-existence is demonstrated by non-existence of the corresponding del Pezzo surface(s) of degree 2 (as in Table \ref{tab:dp2data}), and for which existence is demonstrated by showing there exists a rational point outside of the bad locus on the corresponding del Pezzo surface(s) of degree 2. The highlighted rows indicate types for which the existence is characterized for all possible values of $q$.}
    \begin{minipage}{0.44\textwidth}
    \centering
        \begin{tabular}{|c|l|l|}
        \hline
            Type dP1 & Does not & Exists for:\\ 
            & exist for: & \\
            \hline
            38 & $q \leq 4$ & $q \geq 11$\\
            39 & $q = 2$ & $q \geq 7$\\
            40 & $q = 2 $ & $q \geq 7$\\
            41 & & $q \geq 4$\\
            \rowcolor{lightgray}
            42 & & all $q$\\
            43 & & $q = 2, q \geq 4$\\
            44 & $q \leq 8$ & $q \geq 17$\\
            45 & $q = 2 $ & $q \geq 9$\\
            46 & & $q \geq 7$\\
            47 & & $q \geq 5$\\
            48 & $q = 2 $ & $q \geq 9$\\
            49 & & $q \geq 7$\\
            50 & $q = 2 $ & $q \geq 5$\\
            51 & & $q \geq 5$\\
            52 & & $q \geq 4$\\
            53 & & $q \geq 4$\\
            \rowcolor{lightgray}
            54 & & all $q$\\
            \rowcolor{lightgray}
            55 & $q = 2 $ & $q \geq 3$\\
            56 & $q = 2 $ & $q \geq 9$\\
            57 & $q = 2 $ & $q \geq 7$\\
            58 & & $q \geq 7$\\
            59 & & $q \geq 5$\\
            60 & & $q \geq 4$\\
            \arrayrulecolor{lightgray}\hline
            \arrayrulecolor{black}
            66 & $q \leq 5$ & $q \geq 17$\\
            67 & & $q \geq 11$\\
            68 & $q = 2 $ & $q \geq 9$\\
            69 & $q \leq 4$ & $q \geq 9$\\
            70 & $q = 2 $ & $q \geq 9$\\
            71 & & $q \geq 5$\\
            \hline
        \end{tabular}
    \end{minipage}
    \begin{minipage}{0.55\textwidth}
    \centering
        \begin{tabular}{|c|l|l|}
        \hline
            Type dP1 & Does not & Exists for:\\ 
            & exist for: & \\
            \hline  
            72 & $q = 2$ & $q \geq 4$\\
            73 & & $q \geq 7$\\
            74 & & $q \geq 7$\\
            \arrayrulecolor{lightgray}\hline
            \arrayrulecolor{black}
            76 & $q = 2 $ & $q \geq 9$\\
            \arrayrulecolor{lightgray}\hline
            \arrayrulecolor{black}
            79 & & $q \geq 5$\\
            80 & & $q \geq 4$\\
            \arrayrulecolor{lightgray}\hline
            \arrayrulecolor{black}
            87 & & $q \geq 7$\\
            88 & & $q \geq 7$\\
            \arrayrulecolor{lightgray}\hline
            \arrayrulecolor{black}
            91 & $q \leq 8$ & $q = 53, 59, 61, q \geq 67$\\
            92 & $q \leq 4$ & $q = 31, q \geq 37$\\
            93 & $q \leq 4$ & $q \geq 17$\\
            94 & $q = 2 $ & $q \geq 19$\\
            95 & $q = 2 $ & $q \geq 5$\\
            96 & $q = 2 $ & $q \geq 9$\\
            97 & $q = 2 $ & $q = 13, q \geq 17$\\
            98 & $q = 2 $ & $q \geq 7$\\
            99 & & $q \geq 9$\\
            100 & $q = 2 $ & $q \geq 7$\\
            101 & $q = 2 $ & $q \geq 4$\\
            102 & & $q \geq 9$\\
            103 & & $q \geq 7$\\
            104 & $q = 2 $ & $q \geq 7$\\
            105 & & $q \geq 4$\\
            106 & & $q \geq 7$\\
            107 & & $q \geq 4$\\
            \arrayrulecolor{lightgray}\hline
            \arrayrulecolor{black}
            109 & & $q \geq 4$\\
            110 & & $q \geq 5$\\
            111 & & $q \geq 4$\\
            & & \\
            \hline
        \end{tabular}
    \end{minipage}
    \label{tab:existencedP1withline}
\end{table}

\begin{table}[ht]
    \centering
    \caption{For each type of del Pezzo surface of degree 1 of index 8, sorted based on their trace value, we give the degrees of the closed points in $\P^2$ that need to be blown up to obtain a surface of this type.}
    \begin{tabular}{|c|c|l|}
    \hline
    type & trace & degree of closed points in center of blow-up\\
    \hline
    \hline
    91 & 9 & $1^8$ \\
    \hline
    92 & 7 & $1^6\cdot 2$\\
    \hline
    94 & 6 & $1^5 \cdot 3$\\
    \hline
    93 & 5 & $1^4 \cdot 2^2$\\
    97 & 5 & $1^4 \cdot 4$\\
    \hline
    96 & 4 & $1^3 \cdot 2 \cdot 3$\\
    102& 4 & $1^3 \cdot 5$ \\
    \hline
    95 & 3 & $1^2 \cdot 2^3$\\
    100 & 3 & $1^2 \cdot 3^2$\\
    101 & 3 & $1^2 \cdot 2 \cdot 4$\\
    107 & 3 & $1^2 \cdot 6$\\
    \hline
    99 & 2 & $1 \cdot 2^2 \cdot 3$\\
    105 & 2 & $1 \cdot 3 \cdot 4$\\
    106 & 2 & $1 \cdot 2 \cdot 5$\\
    111 & 2 & $1 \cdot 7$\\
    \hline
    98 & 1 & $2^4$\\
    103 & 1 & $2 \cdot 3^2$\\
    104 & 1 & $2^2 \cdot 4$\\
    108 & 1 & $4^2$\\
    109 & 1 & $3 \cdot 5$\\
    110 & 1 & $2 \cdot 6$\\
    112 & 1 & $8$\\
    \hline
\end{tabular}
    \label{tab:blowupindex8}
\end{table}

\begin{table}[ht]
    \centering
    \caption{For each type of del Pezzo surface of degree 1 over finite fields of index 8 which has a $(-1)$-curve defined over $\F_q$, we list in column 2 precisely for which values of $q$ they can be realized. When the result was previously known, a reference is given in column 3.}
    \begin{tabular}{|c|c|l|}
    \hline
            Type & Exists if and only if: & Reference (if existing result)\\
            \hline
            91 & $q = 16$ or $q \geq 19$ & \cite[16]{BFL}\\
            92 & $q \geq 11$ & \cite[16]{BFL}\\
            93 & $q \geq 7$ & new for $q \geq 5$, \cite[17]{BFL} for $q = 2, 3, 4$ \\
            94 & $q \geq 7$ & \cite[16]{BFL}\\
            95 & $q \geq 7$ & new for $q \geq 3$, \cite[18]{BFL} for $q = 2$\\ 
            96 & $q \geq 4$ & new for $q \geq 3$, \cite[17]{BFL} for $q = 2$\\
            97 & $q \geq 5$ & \cite[17]{BFL}\\
            98 & $q \geq 4$ & new\\
            99 & $q \geq 4$ & new\\
            100 & $q \geq 3$ & new for $q \geq 3$, \cite[18]{BFL} for $q = 2$\\
            101 & $q \geq 3$ & new for $q \geq 3$, \cite[18]{BFL} for $q = 2$ \\
            102 & $q \geq 3$ & \cite[17]{BFL}\\
            103 & $q \geq 3$ & new\\
            104 & $q \geq 3$ & new\\
            105 & $q \geq 3$ & new\\
            106 & for all $q$ & new\\
            107 & $q \geq 3$ & \cite[17]{BFL}\\
            109 & for all $q$ & new\\
            110 & for all $q$ & new\\
            111 & for all $q$ & \cite[18]{BFL}\\
            \hline
    \end{tabular}
    \label{tab:dp1index8withline}
\end{table}

\end{document}